\documentclass[aap]{amsart}

\usepackage[dvipsnames]{xcolor}
\newcommand{\rv}[1]{{\color{black}{#1}}}

\usepackage{amsfonts}
\usepackage{amsmath}
\usepackage{amsthm} %
\usepackage{amssymb}
\usepackage{epsfig, graphics}
\usepackage{hyperref}
\hypersetup{
    colorlinks = true,
    linkcolor = {blue},
    citecolor = {blue}
}
\usepackage{url}
\usepackage{color}
\usepackage{setspace}
\usepackage{float}
\usepackage{subfig}
\usepackage{colortbl}
\usepackage{bm}
\usepackage{bbm}
\usepackage{mathtools}
\mathtoolsset{showonlyrefs}
\usepackage{dsfont}
\usepackage{mathrsfs}
\usepackage{fancyvrb}

\usepackage{enumitem}

\usepackage[square, numbers, comma, sort&compress]{natbib}

\usepackage[margin=1.35in]{geometry}

\newtheorem{theorem}{Theorem}[section]
\newtheorem{remark}[theorem]{Remark}

\newtheorem{definition}[theorem]{Definition}

\newtheorem{proposition}[theorem]{Proposition}
\newtheorem{lemma}[theorem]{Lemma}

\newtheorem{assumption}{Assumption}

\newcounter{subassumption}[assumption]

\makeatletter
\renewcommand{\p@subassumption}{}
\makeatother

\usepackage[dvipsnames]{xcolor}

\newcommand\cB{\mathcal B}

\newcommand\cG{\mathcal G}

\newcommand\cI{\mathcal I}
\newcommand\cP{\mathcal P}

\def\balpha{\bar \alpha}
\def\bX{\bar X}

\newcommand\EE {\mathbb E}
\newcommand\FF {\mathbb F}

\newcommand\NN {\mathbb N}
\newcommand\RR {\mathbb R}

\newcommand\PP {\mathbb P}

\newcommand\cC {\mathcal C}

\newcommand\cF {\mathcal F}
\newcommand\cL {\mathcal L}

\newcommand{\cJ}{\mathcal{J}}

\DeclareMathOperator*{\esssup}{ess\,sup}

\newcommand{\cW}{\mathcal{W}}
\newcommand{\cM}{\mathcal{M}}

\newcommand{\bH}{\mathbb{H}}

\def\h{\hat}
\def\t{\tilde}
\def\o{\bar}
\def\u{\underline}
\def\w{\widehat}

\makeatletter
\newtheoremstyle{nodot}  %
  {\topsep}  %
  {\topsep}  %
  {\itshape} %
  {}         %
  {\bfseries} %
  {}         %
  { }        %
  {\thmname{#1} \thmnumber{#2} \thmnote{(#3)}} %
\makeatother

\theoremstyle{nodot}
\newtheorem{assumptionalt}{Assumption}[assumption]

\newenvironment{assumptionp}[1]{
  \renewcommand\theassumptionalt{#1}
  \assumptionalt
}{\endassumptionalt}

\newcounter{counterExtraAssumptions}

\title[Probabilistic Analysis of GMFC]
  {Probabilistic Analysis of Graphon Mean Field Control}

\author[Zhongyuan Cao]{Zhongyuan Cao\textsuperscript{*}}
\author[Mathieu Lauri\`{e}re]
  {Mathieu Lauri\`{e}re\textsuperscript{\textdagger}}

\thanks{\textsuperscript{*}School of Science and Engineering,
  The Chinese University of Hong Kong (Shenzhen), Guangdong 518172,
  People's Republic of China (\mbox{\href{mailto:caozhongyuan@cuhk.edu.cn}{caozhongyuan@cuhk.edu.cn}}).}

\thanks{\textsuperscript{\textdagger}NYU-ECNU Institute of Mathematical
  Sciences, NYU Shanghai, Shanghai 200126, People's Republic of China (\mbox{\href{mailto:mathieu.lauriere@nyu.edu}{mathieu.lauriere@nyu.edu}}).}

\date{}

\begin{document}

\begin{abstract}
Motivated by recent interest in graphon mean field games and their applications, this paper provides a comprehensive probabilistic analysis of graphon mean field control (GMFC) problems, where the controlled dynamics are governed by a graphon mean field stochastic differential equation with heterogeneous mean field interactions. We formulate the GMFC problem with general graphon mean field dependence and establish the existence and uniqueness of the associated coupled graphon mean field forward and backward stochastic differential equations (FBSDEs). We then derive a version of the Pontryagin stochastic maximum principle tailored to GMFC problems. Furthermore, we analyze the solvability of the GMFC problem for linear dynamics and study the continuity and stability of the graphon mean field FBSDEs under the optimal control profile. Finally, we show that the solution to the GMFC problem provides an approximately optimal solution for large systems with heterogeneous mean field interactions, based on a propagation of chaos result.
\end{abstract}

\maketitle



\tableofcontents

\section{Introduction}
\label{se:intro}
The aim of this paper is to study control problems for large particle systems with heterogeneous interactions and the corresponding asymptotic problems in the infinite population limit. The study of mean field systems with homogeneous interactions dates back to the work of Boltzmann, Vlasov, McKean, and others (see, e.g.,~\cite{mckean1967propagation, kac1959probability}). Since then, both forward and backward stochastic differential equations (SDEs and BSDEs) of McKean–Vlasov (MKV) type, also known as mean field type, have been extensively studied (see~\cite{buckdahn2007mean,buckdahn2009mean,buckdahn2017mean} and the references therein). Building on this, the theory of mean field games (MFGs), introduced by Lasry and Lions~\cite{lasry2007mean} and Huang, Caines, and Malham\'e~\cite{huang2006large, huang2007large}, has gained significant attention. See~\cite{carmonabook1,carmonabook2} and the references therein. 
Beyond the study of MFGs, mean field control (MFC) problems have also attracted significant interest, see, e.g., \cite{carmona2013control,bensoussan2013mean}. In MFC problems, agents cooperate to minimize a common cost. They collectively choose a control that influences the population distribution (i.e., the mean field term). 
For more details on the differences between MFG and MFC, see e.g.~\cite{carmona2019price,cardalia2019efficiency,carmona2023nash}.

In this paper, we focus on MFC type problems. %
A deterministic approach based on partial differential equations (PDEs) has been developed (see e.g. \cite{bensoussan2013mean,achdou2015system,pham2018bellman}), and two stochastic approaches have been considered: Under suitable differentiability and convexity assumptions, Carmona and Delarue \cite{carmona2015forward} provide a systematic study of the control problem using coupled forward and backward stochastic differential equations (FBSDEs).  Alternatively, \cite{lackercontrolweak} studies the MFC problem through a weak formulation using a controlled martingale problem. Here we will focus on the FBSDE approach. 

Most of the MFC literature focuses on homogeneous populations but, motivated by applications, several extensions have been studied in recent years, see~\cite{lauriere2025extensions} for a survey. In particular, mean field systems on large networks and mean field games with heterogeneous interactions have been studied for different random graph models, including Erdős–Rényi graphs \cite{delarue2017mean} and heterogeneous random graphs \cite{oliveira2019interacting}. The concept of graphons provides a natural continuum limit for large dense graphs, and has proven useful for analyzing heterogeneous interactions in infinite populations~\cite{lovasz2012, borgs2008convergent, borgs2012convergent}. %
Static graphon games were introduced in~\cite{parise2023graphon} and later extended to the stochastic setting in~\cite{mlstatic}. The framework of graphon mean field games (GMFGs) and the associated GMFG equations were formulated in \cite{caines2022graphon,caines2021graphon}, where the authors also studied the $\epsilon$-Nash property. Graphon mean field FBSDEs, arising from a form of Pontryagin's maximum principle, were analyzed in~\cite{graphon2022lauriere} for a class of linear quadratic problems. More generally, the well posedness and propagation of chaos for graphon mean field SDEs and coupled graphon mean field FBSDEs have been studied in \cite{erhangmfs} and \cite{wu2022}, respectively. Graphon mean field BSDEs with jumps were considered in \cite{amicaosu2022graphonbsde}. A joint label and state formulation of GMFGs, adopting a weak approach, was explored in \cite{lacker2021soret}. In \cite{cao2023nash}, the authors analyzed a general GMFG model with jumps, where both drift and volatility are controlled, and the cost function exhibits a nonlinear mean field dependence. Graphon mean field systems with nonlinear interactions and the associated propagation of chaos property were studied in \cite{phamnonlinear}. 
A financial application to portfolio management has been studied in~\cite{tangpi2024optimal}, and solved numerically using deep learning in~\cite{lauriere2025deep}. \cite{tangpi2024graphon} studied convergence of a system of interacting particles to a limit whose interactions are governed by a random graph generated by $W$. 
Furthermore, the study of mean field limits for nonexchangeable systems has attracted considerable interest in the PDE community. See, e.g., \cite{jabin2024gmfs,ayi2024gmfs}.
GMFGs in finite state space have also been investigated, both in continuous time~\cite{aurell2022finite} and in discrete time~\cite{cui2021graphon,vasal2021sequential}. In general, these problems cannot be reduced to standard MFGs due to the special role played by the label, which requires an ad hoc treatment of measurability issues.

\noindent\textbf{Related works and novelty.}  To the best of our knowledge, MFC problems in nonexchangeable systems with heterogeneous interactions have received little attention and the general probabilistic framework remained to be developed. The works \cite{gao2018graphon,gao2019optimal,liang2021finite,xu2024social,pham2025lqgmfc,dunyak2024quadratic} studied graphon mean field control (GMFC) problems in the linear quadratic setting. Through dynamic programming on the flow of probability measure families, \cite{pham2024gmfc} approached GMFC problems via PDE methods. %
Simultaneously and independently of our work, \cite{pham2025maximum} studied a version of the maximum principle for the optimal control of nonexchangeable systems. In discrete time and finite state spaces, GMFC problems were studied in \cite{hu2023graphon}. 
\rv{The present paper focuses on the continuous time and continuous space setting and systematically studies the GMFC problem from a probabilistic perspective. We consider heterogeneous dynamics that depend on agent labels, with graphon interactions that depend on the full mean field distribution, beyond just the linear interaction. Both the general formulation and the coupled optimally controlled FBSDE analysis allow possibly controlled heterogeneous volatility. Controls are given label by label as functions of the initial state and the agent's own Brownian path. Their dependence on the label is imposed only in the original $L^2$ sense of Definition~\ref{def:admisscontrol}. To deal with the measurability issues that arise from heterogeneity in the problem formulation and related FBSDEs, we work on a canonical space constructed through suitable couplings, instead of following the usual way as in most literature (e.g. \cite{phamnonlinear,graphon2022lauriere}) by working on the rich Fubini extensions. We first construct solutions on the coupled canonical space, freeze the resulting deterministic law profile and put it into original independent noise systems. We then recover the original solutions. The Well-posedness of the fixed control equation is obtained directly, whereas the one for fully coupled optimal FBSDE is recovered by a law invariance result.}

\noindent\textbf{Main contributions.} The main contributions of this work are as follows:
\begin{enumerate}[leftmargin=0.5cm,topsep=0pt,itemsep=0pt,partopsep=0pt, parsep=0pt]
    \item For any fixed control profile, we introduce a graphon mean field FBSDE in which the backward variable plays the role of an adjoint state,  and we establish its \textbf{well posedness} (Theorem~\ref{thm:exfb}).  
    \item We prove \textbf{necessary} (Theorem~\ref{thm:nece}, Proposition~\ref{prop:nece-weak}) and \textbf{sufficient} (Theorem~\ref{th:suffi}) versions of \textbf{Pontryagin stochastic maximum principle}.
    \item We establish the \textbf{solvability} (Theorems~\ref{th:exo}), \textbf{continuity} (Theorem~\ref{thm:conti}) and \textbf{stability} (Theorem~\ref{thm:sta}) of the \textbf{FBSDE system at optimum}. %
     \item We prove a \textbf{propagation of chaos property} (Theorem~\ref{thm:poc}), and show the GMFC optimal control is \textbf{approximately optimal} for the corresponding finite population problem (Theorem~\ref{th:limit cost}).
 \end{enumerate}

On the technical side, proving these results requires treating the label carefully and in particular the question of measurability of the processes involved. In most cases, we do not require measurability of the processes themselves but only a form of weak measurability of their laws. Furthermore, a significant difference between graphon \emph{games}, whose theory has been more studied, and  graphon \emph{control} problems is that the optimality conditions of the latter involve differentiating with respect to a family of measures, which also needs to be treated carefully.

 \noindent\textbf{Outline of the paper.} We begin in Section~\ref{sec:setting} by specifying the probabilistic setting and formulating the GMFC problem, including essential preliminary definitions and notations. In Section~\ref{sec:eu}, we introduce the adjoint equations, leading to the FBSDE system for a fixed control profile. We establish existence and uniqueness results for the controlled graphon mean field SDEs and the associated FBSDE system. 
Section~\ref{sec:pontryagin} is dedicated to deriving a GMFC version of the Pontryagin stochastic maximum principle, providing both necessary and sufficient conditions for optimality. In Section~\ref{sec:fbsde}, we analyze the solvability of the FBSDE system arising from the GMFC Pontryagin maximum principle and establish two key properties: continuity and stability.
Finally, in Section~\ref{sec:chaos}, we study the controlled heterogeneously interacting particle system and its connection to the GMFC problem. Building on the continuity and stability results of Section~\ref{sec:fbsde}, we demonstrate that the optimal control profile of the GMFC problem can be used to construct an approximately optimal control for large population systems and show that the optimal cost of the large population system converges to the optimal cost of the GMFC problem.

\section{Probabilistic Setup and Notations}
\label{sec:setting}

\subsection{General notations}

Let $I:=[0,1]$ be the label set, equipped with the Lebesgue $\sigma$ algebra $\mathfrak{L}(I)$ on $I$. Let $\Upsilon$ be the open unit interval $(0,1)$, equipped with the Borel $\sigma$ algebra $\cB(\Upsilon)$ on $\Upsilon$. For an interval $\cI$ in $\RR$, let $\iota(\cI)$ be the Lebesgue measure on $\cI$ and we equip both $I$ and $\Upsilon$ with Lebesgue measure. Let $T \in (0,+\infty)$ be a finite time horizon. Let $d$ and $m$ be two integers. \rv{We work on a complete probability space $(\Omega,\cF,\PP)$ carrying a family of independent $m$-dimensional Brownian motions $(W^u)_{u\in I}$ and a family of independent, square integrable initial conditions $(\xi^u)_{u\in I}$ taking values in $\RR^d$. The two families are independent. We write $\mu_0^u=\cL(\xi^u)$ and for each $u\in I$, let $\mathbb F^u=(\cF_t^u)_{0\leq t\leq T}$ be the completed natural filtration
\[
 \cF_t^u=\sigma\bigl(\xi^u,W_s^u:0\leq s\leq t\bigr)\vee\mathcal N_{\PP},
\]
where $\mathcal{N}_{\PP}$ is the set of all $\PP$-null events. The measurability and moment assumptions on $(\mu_0^u)_{u\in I}$ are stated in Assumption~\ref{ass:dynamic}.}

For a random variable $X$, we denote by $\mathcal{L}(X)$ its law. \rv{For a generic measurable space $(Y,\mathcal{Y})$, without specification, when we say $F:I\ni u\mapsto y^u\in Y$ is measurable, we mean that for any set $\mathfrak{Y}\in \mathcal{Y}$, the preimage $F^{-1}(\mathfrak{Y})$ belongs to $\mathfrak{L}(I)$.  
We shall denote by $|A|$ the Euclidean norm for a vector $A$, and by $|A|=\sqrt{\text{trace}(A A^\top)}$ the Frobenius norm of a matrix A, where $\text{trace}(\cdot)$ is
the trace operator.} For two vectors $x$ and $y$ of the same dimension, $x \cdot y$ denotes  their inner product. For a $L^p$ function, we denote by $\|\cdot\|_{p}$ its $L^p$ norm. \rv{We denote by $\cC^d := \mathcal{C}([0,T],\RR^d)$ the set of continuous functions from $[0,T]$ to $\RR^d$, equipped with the supremum norm $\|x\|=\sup_{t\in[0,T]}\|x_t\|$}. \rv{For a generic metric space $(E,d_E)$, We denote by $\cP(E)$ (resp. $\cP_p(E)$) the set of probability measures on $(E,d_E)$ (resp. with finite $p$ moment), equipped with the topology of weak convergence (resp. $p$-Wasserstein metric). For two probability measures $\mu,\mu'\in \cP_p(E)$, $\cW_p(\mu,\mu')$ denotes the $p$-Wasserstein distance between $\mu$ and $\mu'$.}
We denote by $\mathbf{S}^{2}(\FF^u)$ (resp. $\mathbf{H}^{2}(\FF^u)$) the space of \rv{$\mathbb{F}^u$-progressively measurable}, $\RR^d$-valued \rv{continuous} processes $Y$ (resp. the set of $\mathbb{F}^u$-predictable and $\mathbb{R}^{d\times m}$-valued processes $Z$) such that  
$$
	\|Y\|_{\mathbf{S}^2}\coloneqq\Big(\mathbb{E}[\sup_{t\in[0,T]} |Y_t|^2]\Big)^{\frac{1}{2}} <  \infty,
	\quad 
	\Big(\hbox{resp. } \|Z\|_{\mathbf{H}^2}\coloneqq\bigg(\mathbb{E}\Big[\int_0^T |Z_t|^2 dt\Big]\bigg)^{\frac{1}{2}} < \infty\Big). 
$$
We will denote by $\mathbb{W}$ the Wiener measure on $\cC^m$.
\rv{We write $\cC_0^m:=\{w\in\cC^m:w_0=0\}$.}

\subsection{\texorpdfstring{$L$}{TEXT}-differentiability of functions of measures}\label{sec:lderiva}

The notion of  
$L$-differentiability of functions of probability measures was introduced by Lions in his lectures at the {\it Coll\`ege de France}. We refer to \cite[Chapter 5]{carmonabook1} for more details. Let $(\t \Omega,\t \cF,\t \PP)$ be a probability space and $\t H$ be the lifting of the function $\cP_2(\RR^d) \ni \mu\mapsto H(\mu)\in\RR$, in the sense that 
$
    \t H(\t X)=H(\mu),
$
where $\t X\in L^2(\t\Omega;\RR^d)$ and $\t\PP\circ \t X^{-1}=\mu$. 
We say that $H$ is $L$-differentiable at $\mu\in \cP_{2}(\RR^d)$ if there exists a random variable $\t X\in L^2(\t \Omega;\RR^d)$ such that the lifted function $\t H$ is Fr\'echet differentiable at \rv{$\tilde{X}$}. By identifying $L^2$ with its dual, the Fr\'echet derivative $[D\t H](\t X)$ can be viewed as an element $D\t H(\t X)$ of $L^2(\t \Omega; \RR^d)$. We denote by $\partial_{\mu}H(\mu_0): \RR^d\ni x\mapsto \partial_\mu H(\mu_0)(x)\in\RR^d$ the derivative of $H$ at $\mu_0$. By definition, 
\begin{equation}\label{eq:l}
	H(\mu)=H(\mu_0)+D\t H(\t X_0)\cdot(\t X-\t X_0)+o(\|\t X-\t X_0\|_2),
\end{equation}
\rv{whenever $\tilde{X}$ and $\tilde{X}_0$ are random variables with distributions $\mu$ and $\mu_0$ respectively. Note that for the lifting subjects, the usual vector product notation ``$\cdot$" means the inner product on $L^2(\tilde{\Omega};\RR^d)$. Then $\partial_{\mu} H(\mu_0)(\tilde{X}_0)$ reads as $D\tilde{H}(\tilde{X}_0)$}. 
We say a function $H$ on $\cP_{2}(\RR^d)$ is convex if for any $\mu_1$ and $\mu_2$ in $\cP_{2}(\RR^d)$, we have 
$$
	\rv{H(\mu_1)-H(\mu_2)-\t \EE[\partial_\mu H(\mu_2)(\t X_2)\cdot(\t X_1-\t X_2)]\geq 0,}
$$
\rv{whenever $\tilde{X}_1$ and $\tilde{X}_2$ are square integrable random variables with distributions $\mu_1$ and $\mu_2$ respectively.} The above definitions extend naturally %
to multivariate functions.

\subsection{Graphons}

Let us first define the graphons we will study in this paper. The concept of graphon is defined in \cite{lovasz2012} as a measurable symmetric function from $I\times I\to [0,1]$. In this paper, we study a more general class of graphons: a graphon is defined as a measurable and bounded function $G: I \times I \rightarrow \RR_0^+$, where $\RR^+_0$ is the positive half axis containing the origin. Under suitable assumptions, graphons can be regarded as the limits of adjacency matrices of weighted graphs, when the number of vertices goes to infinity. For more details about graphons, including the notions of cut norm and operator norm, we refer to~\cite{lovasz2012}.
Note that by the above definition of graphon, every graphon $G$ is in $L^p(I\times I,\RR)$ and has a finite norm $\|G\|_p$, for any $p\geq 1$.

In the sequel, we make the following standing assumption on every graphon we consider. It is satisfied for instance for graphons that are bounded away from \rv{$0$}.
\begin{assumption}\label{ass:graphon}
Recall that we consider only bounded graphons. Moreover, we suppose: $G^{-1}_\infty:=\sup_{u\in I}(\|G(u,\cdot)\|_1)^{-1}= \sup_{u\in I}\frac{1}{\rv{\int_IG(u,v)dv}}< \infty$.
\end{assumption}

\subsection{Graphon mean field control problems}
\label{sec:gmfc-def}

Differently from the classical mean field case, the interaction between one agent and the rest of the population does not depend on the state distribution of the whole population, but on a \emph{weighted} average of the neighbors' state distributions. %

When deriving optimality conditions, we will need to differentiate with respect to the aggregate distribution. To this end, we will assume that the graphon mean field control problem to be defined below depends only on a normalized version of the graphon. This renormalization is used to guarantee that the neighborhood mean field is a probability measure so that we can use the Wasserstein metric and the $L$-derivative. Note that the renormalization, already used in \cite{phamnonlinear,pham2024gmfc}, does not change the interaction structure of the underlying system. 
Given a graphon $G$, let $\t G$ \rv{denote} its normalized version, defined as: 
\begin{equation}\label{eq:kappa}
    \tilde{G}(u,v)
    = \|G(u,\cdot)\|_1^{-1}G(u,v)
    = \frac{ G(u,v)}{\int_I G(u,v)dv}, \qquad u,v \in I.
\end{equation}

Given $\mu=(\mu^u)_{u\in I}\in \cP(\RR^d)^I$ satisfying for any $B\in\mathcal{B}(\RR^d)$, $I\ni u\mapsto \mu^u(B)$ is measurable, let $\cG\mu:I \to\cP(\mathbb{R}^d)$ be the function with values in measures defined as follows: 
\begin{equation}\label{Lab}
	[\cG\mu]^u(dx)=\t G\mu[u](dx):=\int_I \t G(u,v)\mu^v(dx)dv, \quad u\in I. 
\end{equation} 
For any  $\phi\in L^\infty(\RR^d; \RR)$ and $u\in I$, we denote $[\cG\mu]^u(\phi):=\int_I\int_{\RR^d}\t G(u,v)\phi(x)\mu^v(dx)dv$.

\begin{remark}\label{rq:2}
Let $\mu=(\mu^u_t)_{u\in I,t\in[0,T]}\in \cP(\cC^d)^I$ \rv{satisfy} for any $t\in[0,T]$ and $B\in\mathcal{B}(\RR^d)$, $I\ni u\mapsto \mu_t^u(B)$ is \rv{measurable}. Then $[\cG\mu_t]^u$ is a probability measure and for any $t\in[0,T]$ and $B\in\mathcal{B}(\RR^d)$, $I\ni u\mapsto [\cG\mu_t]^u(B)$ is \rv{measurable}, see \cite[Lemma 3.1]{phamnonlinear} for more details. In addition,  for all $u\in I$, if $[0,T]\ni t\mapsto \mu^u_t\in\cP(\RR^d)$ is measurable, then we also have $[0,T]\ni t\mapsto [\cG\mu_t]^u\in\cP(\RR^d)$ is measurable. 
\end{remark}

\rv{
\begin{remark}\label{rq:gudefine}
Note that since the underlying state space is $\RR^d$, which is a Polish space, then by \cite[Proposition 2.1]{tangpi2024graphon}, \eqref{Lab} is well defined whenever $u\mapsto \mu^u\in \cP(\RR^d)$ is measurable. 
\end{remark}
}

Let the action space $A$ be an open subset of $\RR^k$. Let $\mathbb{H}^2_{T,u}(A)$ \rv{(resp. $\mathbb{H}^2_{T,u}(\RR^k)$)} be the set of $A$-valued \rv{(resp. $\RR^k$-valued)} $\mathbb{F}^u$-progressively measurable processes $\phi$ satisfying 
$$
	\|\phi\|_{\mathbb{H}^2_{T,u}}:=\Bigl(\EE\Bigl[\int_0^T |\phi_t|^2dt\Bigr]\Bigr)^{1/2}<\infty.
$$

\begin{lemma}[Measurable realization of the initial laws]\label{lem:coupleini}
Suppose \(u\mapsto\mu_0^u\) is measurable and for some
\(\varepsilon>0\),
\[
 \esssup_{u\in I}\int_{\RR^d}|x|^{2+\varepsilon}\mu_0^u(dx)<\infty.
\]
Then there exists a measurable function $\rho_{\RR^d}:\cP(\RR^d)\times\Upsilon\to\RR^d$, such that if we set
\(\rho^u(r)=\rho_{\RR^d}(\mu_0^u,r)\), \(\rho^u\) has law \(\mu_0^u\) and
 $u\mapsto\rho^u$ is measurable from $I$ into
 $L^2(\Upsilon;\RR^d)$. 
\end{lemma}

The proof is given in
Appendix~\ref{app:proofs-sec2}. \rv{We consider the following controls which use only the initial state and the agent's own Brownian history. For each label $u\in I$, let
\[
 a^u:[0,T]\times\cC^m\times\RR^d\longrightarrow A
\]
be Borel measurable and non-anticipative, in the sense that
$a^u(t,w,x)=a^u(t,w_{\cdot\wedge t},x)$.  We set
$\alpha_t^u=a^u(t,W^u_{\cdot\wedge t},\xi^u)$, $dt\otimes d\PP$-a.e.  The maps $(a^u)_{u\in I}$ are specified label by label. Note that the pointwise joint measurability of $(u,t,w,x)\mapsto a^u(t,w,x)$ is not assumed.}


\begin{definition}\label{def:admisscontrol}
\rv{Let the measurable map $\rho_{\RR^d}$ be the same one in Lemma~\ref{lem:coupleini} and let
\[
 \rho^u(r)=\rho_{\RR^d}(\mu_0^u,r),\qquad
 \mathfrak a^u(t,w,r)=a^u(t,w_{\cdot\wedge t},\rho^u(r)).
\]
A control profile $(\alpha^u)_{u\in I}$ is admissible if, for each $u$, $\alpha^u\in \mathbb{H}^2_{T,u}(A)$ has the labelwise representation above, and satisfies
\[
 u\longmapsto\mathfrak a^u\in L^2([0,T]\times\cC^m\times\Upsilon;
      A,dt\otimes\mathbb W\otimes dr)
\] is measurable, and further
\begin{equation}\label{eq:control-direct-L2}
 \int_I\|\mathfrak a^u\|_{L^2(dt\otimes\mathbb W\otimes dr)}^2du<\infty.
\end{equation}}
The set of admissible control profiles is denoted by $\cM \mathbb{H}^2_T(A)$.
\end{definition}
\rv{Note that different from the control class considered in
\cite{pham2024gmfc,pham2025maximum}, we do not require joint measurability in
$(u,t,w,x)$ for the control function $a$.  We require only that, for each
$u\in I$, the control function $a^u$ is jointly measurable in $(t,w,x)$.  To
maintain measurability with respect to $u$, we impose in
Definition~\ref{def:admisscontrol} a measurability condition in an $L^2$ sense
on the induced functions
$\mathfrak a^u(t,w,r)=a^u(t,w_{\cdot\wedge t},\rho^u(r))$.  This $L^2$
measurability is implied by joint measurability in $(u,t,w,x)$ together with
$\alpha^u\in\mathbb H^2_{T,u}(A)$.}

When the context is clear, we omit the underlying space $A$ and write $\cM \mathbb{H}^2_T = \cM \mathbb{H}^2_T(A)$. Let $X^u$ be the state process of label $u$ and let $X = (X^u)_{u\in I}$ be the family of state processes. We denote by $\mu^u_t = \cL(X^u_t) \in \cP(\RR^d)$ the distribution of the state of agent $u$ at time $t$. The state satisfies the following controlled graphon mean field dynamics, which is an infinite system of SDEs: for $u\in I$, 
\begin{equation}
	\label{eq:gcd}
	\begin{split}
		dX^u_t= b^u(t,X^u_t,[\cG\mu_t]^u,\alpha^u_t)dt 
		+\sigma^u(t,X^u_t,[\cG\mu_t]^u,\alpha^u_t)dW^u_t, \quad X^u_0= \xi^u.
	\end{split}
\end{equation} %
For every $u \in I$, $b^u:[0,T]\times \RR^d\times \cP(\RR^d)\times A\to \RR^d$ and $\sigma^u:[0,T]\times\RR^d\times\cP(\RR^d)\times A\to \RR^{d\times m}$ are the drift and diffusion coefficients. $[\cG\mu]^u_t$ serves as the neighborhood mean field of label $u$ at time $t$. We may use a superscript $\cdot^{\alpha}$ to stress that the process depends on the control profile $\alpha$. When the context is clear, we omit this superscript to alleviate the notations. When the context is clear, we will denote $\cG^u_t=[\cG\cL(X_t)]^u$ in the rest of the paper.

\begin{definition}
	Given $\alpha\in \cM\mathbb{H}^2_T$, a solution to~\eqref{eq:gcd} is a family $X=(X^u)_{u\in I}$ identified up to $du$-a.e. indistinguishability, such that $X^u\in\mathbf{S}^2(\FF^u)$ solves \eqref{eq:gcd} for a.e. $u\in I$, satisfies that $u\mapsto \cL(X^u_t)$ is measurable for all $t\in[0,T]$ and moreover $\int_{I}\|X^u\|^{2}_{\mathbf{S}^2}du<\infty
	$.
\end{definition}

\rv{ Given $\alpha\in \cM\mathbb{H}^2_T$, we say that the solution is unique if whenever $(X^u)_{u\in I}$ and $(\tilde{X})_{u\in I}$ are solutions to \eqref{eq:gcd}, the
processes $X^u$ and $\tilde{X}^u$
coincide in $\mathbf{S}^2(\FF^u)$, for a.e. $u\in I$.}

The graphon mean field control (GMFC) problem we consider in the rest of the paper is to minimize over $\cM \mathbb{H}^2_T(A)$:
\begin{equation}
	\label{eq:sigob}
	J(\alpha)=\int_I\EE\left[\int_0^T f^u(t,X^u_t,\cG^u_t,\alpha^u_t)dt+g^u(X^u_T,\cG^u_T)\right]du ,
\end{equation}
where the running cost function $f^u: [0,T]\times\RR^d\times \cP(\RR^d)\times A \to \RR$ and the terminal cost function $g^u: \RR^d\times \cP(\RR^d)\to \RR$ satisfy suitable conditions given below. 
\rv{ Note that under some
integrability and growth
assumptions that will be specified later, provided with suitable joint measurability of functions $f^u,g^u$, $J(\alpha)$ is well defined for $\alpha \in \cM \mathbb{H}^2_T(A)$.}

\section{Graphon Mean Field FBSDEs}
\label{sec:eu}

 In this section, we will study the graphon mean field FBSDEs  induced by the graphon mean field control problem~\eqref{eq:sigob}. We focus on the FBSDE for a fixed control, in preparation for the Pontryagin stochastic maximum principle in Section~\ref{sec:pontryagin}. 

\subsection{Hamiltonian and adjoint equation}\label{sec:3.1}

 The following assumptions will be used in this section. 

\begin{assumption}%
\label{ass:dynamic}
\begin{enumerate}[label=\normalfont{\textbf{(\ref*{ass:dynamic}\arabic{*})}}, ref=\normalfont{\textbf{(\ref*{ass:dynamic}\arabic{*})}}, topsep=2pt,itemsep=0pt,partopsep=0pt, parsep=0pt]
    \item\label{ass:dynamic-1} The laws $\mu^u_0$ of initial conditions $\xi^u$, $u\in I$ satisfy that $I\ni u\mapsto \mu_0^u\in\cP(\RR^d)$ is measurable. 

    \item\label{ass:dynamic-2} The functions $[0,T]\times \RR^d \times \cP_2(\RR^d)\times A\times I \ni (t,x,\mu,\alpha,u) \mapsto b^u(t,x, \mu, \alpha)\in \RR^d$ and $[0,T]\times \RR^d \times \cP_2(\RR^d)\times A\times I \ni (t,x,\mu,\alpha,u) \mapsto \sigma^u(t,x, \mu, \alpha)\in \RR^{d\times m}$ are jointly measurable. \rv{The functions $b^u,\sigma^u,u\in I$ satisfy that 
    $$\int_{I}\int_0^T\big(|b^u(t,0,\delta_0,0)|^2+|\sigma^u(t,0,\delta_0,0)|^2\big)dtdu<\infty$$.}
    \item\label{ass:dynamic-3} $b^u,\sigma^u, u\in I$, are Lipschitz continuous with respect to all parameters except $t$ uniformly in $u$. That is, there exists a constant $C$ such that, for each $u\in I$ and $t\in[0,T]$, for any $\alpha_1,\alpha_2\in A$, $x_1,x_2\in \RR^d$, and $\mu_1,\mu_2\in \cP_2(\RR^d)$, 
	\begin{align*}
		|b^u(t,x_1,\mu_1,\alpha_1)&-b^u(t,x_2,\mu_2,\alpha_2)|+|\sigma^u(t,x_1,\mu_1,\alpha_1)-\sigma^u(t,x_2,\mu_2,\alpha_2)| \\
		& \leq C(|x_1-x_2|+\cW_2(\mu_1,\mu_2)+|\alpha_1-\alpha_2|).
	\end{align*}
    
    \item\label{ass:dynamic-4} There exists $\varepsilon>0$, such that $(\mu^u_0)_{u\in I}$ satisfies $\esssup_{u\in I}\int_{\RR^d}|x|^{2+\varepsilon}\mu^u_0(dx)<\infty$.
    
	\item\label{ass:dynamic-5} The functions $[0,T]\times \RR^d \times \cP_2(\RR^d)\times A\times I \ni (t,x,\mu,\alpha,u) \mapsto f^u(t,x, \mu, \alpha)\in \RR$, and $\RR^d \times \cP_2(\RR^d)\times I\ni (x,\mu,u)\mapsto g^u(x,\mu)\in \RR$ are jointly measurable. For each $u\in I$, the functions $b^u,\sigma^u,f^u,g^u$ are (jointly) differentiable with respect to $x$ and $\mu$. In addition, $[0,T]\times \RR^d \times \cP_2(\RR^d)\times A\times I \ni (t,x,\mu,\alpha,u) \mapsto (\partial_{x}(b^u,\sigma^u,f^u)(t,x, \mu, \alpha),\partial_xg^u(x,\mu))$ are jointly measurable and continuous in $(x,\mu,\alpha)$ for any $(u,t)\in I\times[0,T]$. For any $x'\in\RR^d$, $ (t,x,\mu,\alpha,u) \mapsto \partial_{\mu}(b^u,\sigma^u,f^u)(t,x, \mu, \alpha)(x')$ are jointly measurable and $(x,\mu,u)\mapsto \partial_{\mu}g^u(x,\mu)(x')$ is also jointly measurable. For any $ (t,x,\mu,\alpha,u)\in[0,T]\times \RR^d \times \cP_2(\RR^d)\times A\times I$, $x'\mapsto\partial_{\mu}(b^u,\sigma^u,f^u)(t,x, \mu, \alpha)(x')$ and $x'\mapsto \partial_{\mu}g^u(x,\mu)(x')$ are continuous. 
\end{enumerate}

\end{assumption}

\rv{\noindent\textbf{Comments on Assumption~\ref{ass:dynamic}.} 
Assumption~\ref{ass:dynamic-3} is common even for classical mean field equations to guarantee the well posedness of forward SDEs, see e.g. \cite{carmona2015forward}. The additional measurability condition in~\ref{ass:dynamic-1}--\ref{ass:dynamic-2} are used to maintain the measurability of the forward graphon dynamics. Assumptions~\ref{ass:dynamic-4}--\ref{ass:dynamic-5} are needed to guarantee the well posedness of the backward adjoint equation that will be studied later.}


The Hamiltonian of the GMFC is the family $H = (H^u)_{u\in I}$ of functions 
\begin{equation}
	\label{fo:hamiltonian}
	H^u(t,x,\mu,y,z,\alpha)=b^u(t,x,\mu,\alpha)\cdot y +\sigma^u(t,x,\mu,\alpha)\cdot z + f^u(t,x,\mu,\alpha), \quad u \in I.
\end{equation}

As discussed in Section~\ref{sec:lderiva}, we denote by $\partial_\mu H^u(t,x,\mu_0,y,z,\alpha)$ the derivative with respect to $\mu$ computed at $\mu_0$ whenever all the other variables $t$, $x$, $y$, $z$ and $\alpha$ are fixed. For each $u\in I$, $\partial_\mu H^u(t,x,\mu_0,y,z,\alpha)$ is an element of $L^2(\t \Omega,\RR^d)$ and we identify it with a function from $\RR^d$ to $\RR^d$, 
$
	\partial_\mu H^u(t,x,\mu_0,y,z,\alpha)(\, \cdot \,) : \tilde{x} \mapsto \partial_\mu H^u(t,x,\mu_0,y,z,\alpha)(\t x).
$ 
It satisfies
$D \t H^u(t,x,\t X,y,z,\alpha) = \partial_\mu H^u(t,x,\mu_0,y,z,\alpha)(\t X)$ almost surely under $\t \PP$, where $\t\PP$ is a probability measure defined on another space $\tilde\Omega$ such that $\t\PP\circ \t X^{-1}=\mu_0$.

We summarize here the main distinctions between the heterogeneous mean field control problem considered here and the classical MFC problems. we also shortly discuss and highlight the associated difficulties and our technical contributions.

\begin{itemize}
\item \textbf{Distinctions.}
In the classical model, agents interact through a single population
distribution. In our setting, the dynamics and costs may vary with the
label \(u\in I\) and depend nonlinearly on the neighborhood mean field
\(\mathcal G_t^u\). The states \(X^u\) are driven by mutually independent
Brownian motions \(W^u\), while their adjoint processes satisfy a continuum
of heterogeneous BSDEs coupled through the graphon mean field interactions.
\medskip
\item \textbf{Difficulties.}
The graphon interaction assigns total mass \(\int_I G(u,v)\,dv\) to the neighborhood mean field
for each \(u\), which can differ from one. To apply the \(L\)-derivative, we
must first represent each neighborhood mean field as a probability
distribution. Moreover, a variation at label \(u\) affects the distributions of all states at all other labels \(v\neq u\) heterogeneously. The driver of adjoint equation at label \(u\) therefore involves an integral over \(v\) of the derivative of Hamiltonian 
\(H^v\). Making this integral well defined requires measurability in \(v\)
of the joint distributions of the state, adjoint processes and control, even though the driving noises are independent across labels.
\medskip
\item \textbf{Technical contributions.}
We combine several ideas and techniques to address the difficulties arised in our model. We normalize the graphon and regard \(\mathcal G_t^u\) as the weighted average of the state
distributions of all labels. This gives a probabilistic interpretation of
the \(L\)-derivative term in the adjoint equation. Under the stated
measurability assumptions on \(b,\sigma,f,g\) and the initial
distributions \(\mu_0\), we use a canonical coupling method to solve a ``canonical" version of the problem, and then construct a
jointly measurable kernel \((u,t)\mapsto q_t^u\) satisfying
\[
q_t^u=\mathcal L(X_t^u,Y_t^u,Z_t^u,\alpha_t^u)
\quad\text{for }du\otimes dt\text{-almost every }(u,t).
\]
Finally through a law invariance argument, in the original independent noise system, we recover the existence and uniqueness for the state and adjoint system
associated with each fixed admissible control profile (also the existence and uniqueness for the coupled graphon mean field FBSDE at optimum in section 5).
\end{itemize}

\noindent\textbf{Construction of the representative variable of $\cG$.}
To express the derivative with respect to the graphon mean field $\cG$ according to the definition of $L$-derivatives, we need to introduce a family of random variables with distributions $\cG^u, u\in I$.\footnote{In classical mean field contexts, the mean field is typically the probability distribution of the state of a representative agent.} 


\rv{For each $u\in I$, we define a representative random variable $\Theta^u$ on $\Omega$ with distribution $\cG^u=[\cG\cL(X)]^u$ in the following way: Choose a joint measurable mapping  $(\omega,u)\mapsto \check{X}^u(\omega)$ satisfying $\cL(\check{X}^u)=\cL(X^u)$ for each $u\in I$. For each $u\in I$, let $\vartheta^u$ be a random variable defined on  $\Omega$ (assuming $\Omega$ is rich enough or enlarged to support this independent randomization), taking values in $I$, with density $\tilde{G}(u,\cdot)$ and independent of $\check{X}$ and all other random variables and stochastic processes defined before. Note that the mapping $\check{X}$ exists when $u\mapsto \cL(X^u)$ is measurable. Then we define $\Theta^u:=\check{X}^{\vartheta^u}, u\in I$. For each $u\in I$, $\Theta^u$ has distribution \rv{$[\cG\cL(X)]^u$}.}

The graphon FBSDE system with given control profile $\alpha \in \cM\bH^2_T$ is:
\begin{equation}
	\label{eq:fb}
	\left\{
	\begin{aligned}
		&\textstyle dX^u_t= b^u(t,X^u_t,\cG^u_t,\alpha^u_t)dt 
		+\sigma^u(t,X^u_t,\cG^u_t,\alpha^u_t)dW^u_t,\\
		&\textstyle dY^u_t=-\partial_xH^u(t,X^u_t,\cG^u_t,Y^u_t,Z^u_t,\alpha^u_t)dt + Z^u_t dW^u_t\\
		&\textstyle \qquad\qquad-\int_I\tilde{G}(v,u)\t\EE[\partial_\mu  H^v(t,\t X^v_t,\cG^v_t,\t Y^v_t,\t Z^v_t,\t\alpha^v_t)(X^u_t)]dvdt,\\
		&\textstyle  Y^u_T= \partial_xg^u(X^u_T,\cG^u_T)+\int_I\tilde{G}(v,u)\t\EE[\partial_\mu g^v(\t X^v_T,\cG^v_T)(X^u_T)]dv, \; X^u_0= \xi^u,\; u\in I.
	\end{aligned}
	\right.
\end{equation}
The family of BSDEs is called the adjoint equation associated to the control profile $\alpha$.
\rv{Let
\(
 \mathsf Z=L^2([0,T];\RR^{d\times m}),
 \mathsf A=L^2([0,T];\RR^k),
\)
and
\[
 \mathsf E=\RR^d\times\cC_0^m\times\cC^d\times\cC^d
              \times\mathsf Z\times\mathsf A.
\]
The spaces above are endowed with their Borel $\sigma$-fields. 

\begin{definition}\label{def:FBSDE}
Given \(\alpha\in\cM\bH^2_T(A)\), a solution of the graphon mean field
FBSDE~\eqref{eq:fb} is a family
\((X^u,Y^u,Z^u)_{u\in I}\) such that, for \(du\)-a.e. \(u\),
\[
 (X^u,Y^u,Z^u)\in
 \mathbf S^2(\mathbb F^u)\times\mathbf S^2(\mathbb F^u)
 \times\mathbf H^2(\mathbb F^u)
\]
and \((X^u,Y^u,Z^u)\) solves \eqref{eq:fb}.  Let \([Z^u]\in\mathsf Z\) and \([\alpha^u]\in\mathsf A\) denote their
equivalence classes in \(L^2(dt)\), we require the solution satisfies that
\begin{equation}\label{eq:full-solution-law-kernel}
 u\longmapsto
 \Pi^u:=\cL\bigl(\xi^u,W^u,X^u,Y^u,[Z^u],[\alpha^u]\bigr)
 \in\cP_2(\mathsf E)
\end{equation}
is measurable, and that
\begin{equation}\label{eq:fbregu}
 \int_I\left(
   \|X^u\|_{\mathbf S^2}^2+\|Y^u\|_{\mathbf S^2}^2
   +\|Z^u\|_{\mathbf H^2}^2\right)du<\infty.
\end{equation}
Here \(X,Y\) are identified up to \(d\PP\)-a.e. equality for $du$-a.e. $u$, whereas \(Z\) is identified up to
\(dt\otimes d\PP\)-a.e. equality for $du$-a.e. $u$. We say the solution is unique if for any other solution
\((\widetilde X,\widetilde Y,\widetilde Z)\),
\[
 \int_I\left(
  \|X^u-\widetilde X^u\|_{\mathbf S^2}^2
 +\|Y^u-\widetilde Y^u\|_{\mathbf S^2}^2
 +\|Z^u-\widetilde Z^u\|_{\mathbf H^2}^2\right)du=0.
\]
\end{definition}

Recall the notation
\(
 \mu_t^u=\cL(X_t^u).
\)
The map \((u,t)\mapsto\mu_t^u\) is measurable since
\((t,x)\mapsto x_t\) is continuous on \([0,T]\times\cC^d\) and
\(u\mapsto\Pi^u\) is measurable. Moreover,
\(t\mapsto\mu_t^u\) is continuous.
Next, for
\[
 D\in\cB\bigl([0,T]\times\RR^d\times\RR^d
                 \times\RR^{d\times m}\times A\bigr),
\]
define the occupation measure
\[
 \mathcal Q^u(D)
 =\EE\left[\int_0^T
 {\bf1}_D(t,X_t^u,Y_t^u,Z_t^u,\alpha_t^u)\,dt\right].
\]
We have \(u\mapsto\mathcal Q^u\) is measurable. Then the measurable disintegration gives a
jointly measurable kernel
\[
 (u,t)\longmapsto q_t^u\in
 \cP_2(\RR^d\times\RR^d\times\RR^{d\times m}\times A),
\]
which is unique \(du\otimes dt\)-a.e., such that
\begin{equation}\label{eq:occupation-law-disintegration}
 \mathcal Q^u(dt,d\theta)=dt\,q_t^u(d\theta),\qquad
 q_t^u=\cL(X_t^u,Y_t^u,Z_t^u,\alpha_t^u)
 \quad du\otimes dt\text{-a.e.}
\end{equation}
Moreover, \(q_t^u(\RR^d\times\RR^d\times
\RR^{d\times m}\times A)=1\) for \(du\otimes dt\)-a.e. \((u,t)\) and the \(x\)-marginal of \(q_t^u\) is \(\mu_t^u\) for
\(du\otimes dt\)-a.e. \((u,t)\).  All law terms depending on time in
\eqref{eq:fb} should be understood through this kernel. Note also that any changing of the kernel on a \(du\otimes dt\)-null set does not change the related stochastic integrals and associated equations.
}


\rv{
\begin{remark}\label{rq:int}
    \rv{Note that by \eqref{eq:occupation-law-disintegration}, for each
    $B\in \cB(\RR^d\times\RR^d\times\RR^{d\times m}\times A)$,
    $(u,t)\mapsto q_t^u(B)$ is measurable.} Then by
    Assumption~\ref{ass:dynamic-5}, the two mappings
    \[
    \begin{split}
    v&\mapsto \tilde{G}(v,u)\t\EE[\partial_\mu H^v
    (t,\t X^v_t,\cG^v_t,\t Y^v_t,\t Z^v_t,\t\alpha^v_t)(x)],\\
    v&\mapsto \tilde{G}(v,u)\t\EE[\partial_\mu g^v
    (\t X^v_T,\cG^v_T)(x)]
    \end{split}
    \]
    are measurable for a.e. $u\in I$, provided that the expectations are well defined. We will give conditions to guarantee that the integrals are well defined in the sequel.
    \end{remark}
}

\subsection{Existence and uniqueness results}\label{sec:ex}
 
\begin{theorem}\label{thm:exf}
	Let $\alpha\in\cM\bH^2_T(A)$. Suppose Assumptions~\ref{ass:graphon} and~\ref{ass:dynamic-1}--\ref{ass:dynamic-3} hold. Then there exists a unique solution to the forward controlled graphon mean field system \eqref{eq:gcd}.
\end{theorem}

\begin{proof}
   We will proceed to the proof by using a fixed point argument following techniques similar to e.g. the proof of~\cite[Proposition 2.1]{erhangmfs} for graphon systems (without control). The bound can also be obtained through the fixed point iteration argument. We will construct the iteration equation and one can show that the solutions of iteration equations finally converge to a fixed point. However in order to show the measurability of $u\mapsto \cL(X^u_t,\alpha^u_t), t\in[0,T]$, a crucial step is to show this property holds at every step of the fixed point iterations, which is the main difficulty. Hence, here we only focus on the proof of measurability. The existence, asserted bound and uniqueness at each iteration step remain classical.
	Let  $X^{u,0}_t=\xi^u\in\mathbf{S}^2(\FF^u)$ for all $t\in[0,T]$ and each $u\in I$. Then by Assumption~\ref{ass:dynamic-1}, we have for any $t\in[0,T]$ and $B\in\mathcal{B}(\RR^d)$, $u\mapsto\cL(X^{u,0}_t)(B)$ is measurable. For $n \ge 1$, denote $\mu^{u,n-1}:=\cL(X^{u,n-1})$. For each $u\in I$, define $\cG^{u,n-1}: [0,T] \to \cP(\RR^d)$ as
	$$
        \cG^{u,n-1}_t(dx)=\t G\mu^{n-1}_t[u](dx):=\int_I \t G(u,v)\mu^{v,n-1}_t(dx)dv,
    $$
    and define further
	\begin{equation}\label{eq:itn}
		X^{u,n}_t=X^{u,n-1}_0+\int_0^t b^u(s,X^{u,n-1}_s,\cG^{u,n-1}_s,\alpha^u_s) ds+\int_0^t \sigma^u(s,X^{u,n-1}_s,\cG^{u,n-1}_s,\alpha^u_s) dW^u_s.
	\end{equation} 
	    Then by Remark~\ref{rq:gudefine}, for each $t\in[0,T]$, $\cG^{u,n}_t$ is well defined if we show that $u\mapsto\cL(X^{u,n}_t)$ is measurable. \rv{Let $L^2_A=L^2([0,T];A)$ and regard the control as its equivalence class $[\alpha^u]\in L^2_A$. Since the Hilbert
space \(L^2([0,T]\times\cC_0^m\times\Upsilon,
dt\otimes\mathbb W\otimes dr;\RR^k)\) is separable, by Definition~\ref{def:admisscontrol}, the map
    \[
       (u,r,w)\longmapsto
       \bigl(\rho^u(r),w,[t\mapsto\mathfrak a^u(t,w,r)]\bigr)
       \in\RR^d\times\cC_0^m\times L^2_A
    \]
    has a jointly measurable version. Hence its push forward under $dr\otimes \mathbb W(dw)$ is a measurable kernel
    \[
       u\longmapsto\cL(\xi^u,W^u,[\alpha^u]).
    \]
    Let $\chi^{u,0}(r,w)=\rho^u(r)$ and define $\chi^{u,n}$ from \eqref{eq:itn} with $(\xi^u,W^u,\alpha^u)$ replaced by $(\rho^u(r),w,\mathfrak a^u(\cdot,w,r))$. Following arguments in the proof of \cite[Proposition 2.1]{erhangmfs}, by an simple function approximation for Lebesgue integrals and a predictable step approximation for It\^o integrals, one can inductively show that $(u,r,w)\mapsto\chi^{u,n}(r,w)\in\cC^d$ has a jointly measurable version. Labelwise uniqueness of every iteration step gives
    \[
      \cL\bigl(\chi^{u,n}(R,W),[\mathfrak a^u(\cdot,W,R)]\bigr)
      =\cL(X^{u,n},[\alpha^u]),
    \]
    where $(R,W)$ has law $dr\otimes\mathbb W$. Consequently
    \[
       u\longmapsto\cL(X^{u,n},[\alpha^u])\in\cP(\cC^d\times L^2_A)
    \]
    is measurable. Since By using Assumption~\ref{ass:dynamic-3} and boundedness of $\tilde G$, the standard Picard iteration estimates show the convergence of $\chi^{n}$ in $L^2(du\otimes dr\otimes\mathbb W;\cC^d)$. Then pushing its law to the original independent noise space gives the same kernel measurability for $(X^u,[\alpha^u])$.} 

\end{proof}

\rv{
Following~\cite{amicaosu2022graphonbsde,wu2022}, we take a canonical filtered probability space $(\bar{\Omega},\bar{\cF},\bar{\mathbb{F}},\bar{\PP})$, which supports an $m$-dimensional Brownian motion $\bar W$ and a uniform random variable $\bar\Lambda$ on $\Upsilon$, independent of $\bar W$. Let $\bar\cF_0=\sigma(\bar\Lambda)$, and let $\bar\FF$ be the completed filtration generated by $\bar W$ and $\bar\cF_0$. For $u\in I$, set $\bar\xi^u=\rho_{\RR^d}(\mu_0^u,\bar\Lambda)$. Then $\cL(\bar\xi^u)=\cL(\xi^u)$, $\bar\xi^u$ is $\bar\cF_0$-measurable, and $u\mapsto\bar\xi^u\in L^2(\bar\Omega;\RR^d)$ is measurable by Assumptions~\ref{ass:dynamic-1}, \ref{ass:dynamic-4}, and Lemma~\ref{lem:coupleini}.

For any $t\in[0,T]$, we define
$\bar{\mathbb{H}}^2_{t}(A)$ as the set of $A$-valued $\bar{\mathbb{F}}$-progressively measurable processes $\phi$ defined on the canonical space and satisfying
$$
    \|\phi\|_{\bar{\mathbb{H}}^2_t}:=\bar{\EE}\Bigl[\int_0^t |\phi_s|^2ds\Bigr]<\infty.
$$
Let $\overline{\cM \mathbb{H}}^2_T(A)$ denote the set of all $(\bar{\alpha}^u)_{u\in I}$ such that for each $u\in I$, $\bar{\alpha}^u\in \bar{\mathbb{H}}^2_{T}(A)$, and for any $t\in[0,T]$, the mapping $u\mapsto\o\alpha^u\in \bar{\mathbb{H}}^2_t(A)$ is measurable and
$$
	\int_{I}\bar{\EE}\Bigl[\int_0^T |\o\alpha^u_t|^2dt\Bigr]du <\infty.
$$
For each $\alpha\in\cM \mathbb{H}^2_T(A)$, we define 
\rv{$\bar{\alpha}^u_t=a^u(t,\bar{W}_{\cdot\wedge t},\bar\xi^u)=\mathfrak a^u(t,\bar W,\bar\Lambda)$, up to the $dt\otimes d\bar\PP$-null set.} 
Notice that by the definition of admissibility (see Definition~\ref{def:admisscontrol}), $\o\alpha$ defined above belongs to $\overline{\cM \mathbb{H}}^2_T(A)$. Hence for each $\alpha\in\cM \mathbb{H}^2_T(A)$,
we identify a corresponding control profile $\bar{\alpha}\in\overline{\cM \mathbb{H}}^2_T(A)$ such that \rv{$\cL(\bar\xi^u,\bar W,[\bar\alpha^u])=\cL(\xi^u,W^u,[\alpha^u])$ for a.e. $u\in I$.} 
On the canonical space, we define the following controlled graphon system, with $\o\cG^u_t=[\cG\cL(\o X^{\o\alpha}_t)]^u$:
\begin{equation}
		d\o X^{\bar{\alpha},u}_t= b^u(t,\o X^{\bar{\alpha},u}_t,\o \cG^u_t,\o \alpha^u_t)dt 
		+\sigma^u(t,\o X^{\bar{\alpha},u}_t,\o \cG^u_t,\o \alpha^u_t)d\o W_t, \quad \o X^u_0=\o \xi^u,\quad u\in I.
	\label{eq:cs}
\end{equation}

 With the common Brownian motion $\bar{W}$ and the coupled initial condition $(\bar{\xi}^u)_{u\in I}$, the canonical FBSDE is
 \begin{equation}
	\label{eq:fbww}
    \left\{
	\begin{aligned}
		&\textstyle	d\o X^u_t= b^u(t,\o X^u_t,\o\cG^u_t,\o\alpha^u_t)dt 
		+\sigma^u(t,\o X^u_t,\o\cG^u_t,\o\alpha^u_t)d\o W_t,\\
		&d\o Y^u_t=-\partial_xH^u(t,\o X^u_t,\o\cG^u_t,\o Y^u_t,\o Z^u_t,\o\alpha^u_t)dt + \o Z^u_t d\o W_t\\
		&\textstyle\qquad\qquad-\int_I\tilde{G}(v,u)\t\EE[\partial_\mu  H^v(t,\t {\o X}^v_t,\o\cG^v_t,\t {\o Y}^v_t,\t{\o Z}^v_t,\t{\o\alpha}^v_t)(\o X^u_t)]dvdt,\\
		&\textstyle \o Y^u_T= \partial_xg^u(\o X^u_T,\o\cG^u_T)+\int_I\tilde{G}(v,u)\t\EE[\partial_\mu g^v(\t {\o X}^v_T,\o\cG^v_T)(\o X^u_T)]dv, \quad \o X^u_0= \o\xi^u,\quad\quad u\in I .
	\end{aligned}
    \right.
\end{equation}
To alleviate notation, we denote $\bar{\mathbf{S}}^2=\mathbf{S}^2(\bar{\FF})$ and $\bar{\mathbf{H}}^2=\mathbf{H}^2(\bar{\FF})$.
}

\begin{lemma}\label{lem:barx}
 Let $\o\alpha\in \overline{\cM\bH}^2_T(A)$. Suppose Assumptions~\ref{ass:dynamic-1}-\ref{ass:dynamic-4} hold, then there exists a unique solution to \eqref{eq:cs} and \rv{for any $t\in[0,T]$, $I\ni u\mapsto \bar{X}^u_t\in L^2(\bar{\Omega};\RR^d)$ is measurable. Moreover, $I\ni u\mapsto \bar{X}^u\in \bar{\mathbf{S}}^2$ is also measurable and $(\bar{X}^u)_{u\in I}\in L^2(I;\bar{\mathbf{S}}^2)$.}
 
 \end{lemma}
 The proof is given in Appendix~\ref{app:proofs-sec2}.

\rv{
We denote a deterministic law profile by
\(\mathbf q=(\mu^{q,u},q^u)_{u\in I}\), where
\(\mu_t^{q,u}\in\cP_2(\RR^d), q_t^u\in\cP_2(\RR^d\times\RR^d\times
                   \RR^{d\times m}\times A)
\), \(q_t^u(\RR^d\times\RR^d\times\RR^{d\times m}\times A)=1\) and the
 \(x\)-marginal of \(q_t^u\) is \(\mu_t^{q,u}\). Note that both kernels have
 jointly measurable versions in \((u,t)\). We set $\cG_t^{q,u}=[\cG\mu_t^q]^u$ and for a.e. \((u,t)\),
\begin{align}
 K_t^{q,u}(x)
 &=
 \int_I\tilde{G}(v,u)
 \int \partial_\mu H^v
   (t,x',\cG_t^{q,v},y',z',a')(x)\,
 q_t^v(dx',dy',dz',da')\,dv,
 \label{eq:frozen-K}\\
 K_T^{g,q,u}(x)
 &=
 \int_I\tilde{G}(v,u)
 \int \partial_\mu g^v(x',\cG_T^{q,v})(x)\,
 \mu_T^{q,v}(dx')\,dv.
 \label{eq:frozen-Kg}
\end{align}
For the fully coupled system at the optimum, where the optimal control will be characterized as a family of functions $(\hat{\alpha}^u)_{u\in I}$ of $(t,X^u_t,\cG^u_t,Y^u_t,Z^u_t)$ (see more details and related assumptions in Section~\ref{sec:5.1}). Here we simply assume that for all $u\in I$, $\hat{\alpha}^u$ is Lipschitz continuous in all parameters except $t$ and always taking values in $A$. Let
\[
 A_t^{q,u}(x_0,w;x,y,z)
   =\hat\alpha^u(t,x,\cG_t^{q,u},y,z).
\]
We write \({\rm FBSDE}(u,\mathbf q;x_0,W)\) for
\begin{equation}\label{eq:frozen-fbsde}
\left\{
\begin{aligned}
 dX_t={}&b^u(t,X_t,\cG_t^{q,u},A_t^{q,u})dt
       +\sigma^u(t,X_t,\cG_t^{q,u},A_t^{q,u})dW_t,\\
 dY_t={}&-\bigl[
   \partial_xH^u(t,X_t,\cG_t^{q,u},Y_t,Z_t,A_t^{q,u})
   +K_t^{q,u}(X_t)\bigr]dt+Z_t\,dW_t,\\
 X_0={}&x_0,\qquad
 Y_T=\partial_xg^u(X_T,\cG_T^{q,u})+K_T^{g,q,u}(X_T).
\end{aligned}
\right.
\end{equation}
The arguments \((x_0,W,X_t,Y_t,Z_t)\) of \(A_t^{q,u}\) are omitted in
\eqref{eq:frozen-fbsde}. For notation simplicity, we denote the three coefficients on the
right hand sides of the forward and backward equations by
\(B_t^{q,u}(x_0,w;x,y,z)\), \(\Sigma_t^{q,u}(x_0,w;x,y,z)\), and
\(F_t^{q,u}(x_0,w;x,y,z)\), respectively, and set
\[
 h^{q,u}(x)=\partial_xg^u(x,\cG_T^{q,u})+K_T^{g,q,u}(x).
\]
Writting
\(\nu(du,dx_0,dw)=du\,\mu_0^u(dx_0)\,\mathbb W(dw)
\), we suppose the following conditions
\begin{equation}\label{eq:frozen-profile-moment}
 \int_I\sup_{t\leq T}\int_{\RR^d} |x|^2\mu_t^{q,u}(dx)\,du
 +\int_I\!\int_0^T\!\int_{(\RR^d)^2\times\RR^{d\times m}\times A}
 (|x|^2+|y|^2+|z|^2+|a|^2)\,
 q_t^u(dx,dy,dz,da)\,dt\,du<\infty,
\end{equation}
and 
\begin{equation}\label{eq:frozen-optimal-L2-data}
\begin{split}
 \int\!\bigg[|h^{q,u}(0)|^2+
  \int_0^T\!\big(&|B_t^{q,u}(x_0,w;0,0,0)|^2+
   |\Sigma_t^{q,u}(x_0,w;0,0,0)|^2+|F_t^{q,u}(x_0,w;0,0,0)|^2\big)dt\bigg]d\nu<\infty.
\end{split}
\end{equation}

We next state a result that allows us to recover $(X^u,Y^u,Z^u)$ for a.e. $u$ when the mean field profile is fixed. The result also applies to the fully coupled FBSDE at optimum under Assumptions~\ref{ass:d--1}--\ref{ass:d--3} delayed in Section~\ref{sec:5.1}.

\begin{lemma}
\label{lem:frozen-causal-realization}
Let \(\mathbf q\) satisfy \eqref{eq:frozen-profile-moment}-\eqref{eq:frozen-optimal-L2-data} and set
\(
 \nu(du,dx,dw)=du\,\mu_0^u(dx)\,\mathbb W(dw).
\)
Assume \ref{ass:d--1}--\ref{ass:d--3} hold. Then there is a Borel set \(I_{\mathbf q}\subset I\) of Lebesgue measure $1$ such that,
for every \(u\in I_{\mathbf q}\), there are measurable maps
\[
 (\Phi_X^{\mathbf q,u},\Phi_Y^{\mathbf q,u},\Phi_Z^{\mathbf q,u}):
 \RR^d\times\cC_0^m\longrightarrow
 \cC^d\times\cC^d\times\mathsf Z
\]
with finite \(\mathbf S^2,\mathbf S^2,\mathbf H^2\) norms under
\(\nu_u(dx,dw)=\mu_0^u(dx)\mathbb W(dw)\).
The first two coordinates admit progressive versions
\(\chi^{\mathbf q,u},\upsilon^{\mathbf q,u}\), and the third one admits a predictable
representative \(z^{\mathbf q,u}\), regarding the filtration
\((\mathcal B(\RR^d)\otimes\sigma(w_s:s\leq t))_{t\geq 0}\), \(\nu_u\)-a.s.. Moreover, for $u\in I_{\mathbf q}$, we have for any Brownian motion \(B^u\) and any \(\mathcal F_0\)-measurable \(\zeta^u\sim\mu_0^u\) independent of
\(B^u\),
\[
 \bigl(\Phi_X^{\mathbf q}(u,\zeta^u,B^u),
       \Phi_Y^{\mathbf q}(u,\zeta^u,B^u),
       \Phi_Z^{\mathbf q}(u,\zeta^u,B^u)\bigr)
\]
is the unique solution of \({\rm FBSDE}(u,\mathbf q;\zeta^u,B^u)\). Furthermore, for $u\in I_{\mathbf q}$
\begin{equation}\label{eq:frozen-feedback-map}
 a_{\mathbf q}^u(t,w,x)=
 \hat\alpha^u\!\left(
 t,\Phi^{\mathbf q,u}_X(t,x,w),\cG_t^{q,u},
 \Phi^{\mathbf q,u}_Y(t,x,w),\Phi^{\mathbf q,u}_Z(t,x,w)\right), \qquad dt\otimes\nu_u-a.e.,
\end{equation}
is Borel measurable, non-anticipative, and square integrable.
\end{lemma}

The proof is given in Appendix~\ref{app:proofs-sec2}. The following lemma states the law invariance when we change the input initial conditions and driven Brownian motions for the coupled FBSDEs under a fixed population mean field. 

\begin{lemma}[Law invariance]
\label{lem:frozen-law-transfer}
Let $(\bar X,\bar Y,\bar Z)$ be a solution of the canonical version of system \eqref{eq:fbo}, and let
\[
 \bar\mu_t^u=\cL(\bar X_t^u),\qquad
 \bar q_t^u=\cL(\bar X_t^u,\bar Y_t^u,\bar Z_t^u,\bar\alpha_t^u)
\]
for all \(t\), respectively for a.e. \(t\). Write
\(\bar{\mathbf q}=(\bar\mu,\bar q)\).  Assume that
\(\bar{\mathbf q}\) satisfies the hypotheses of
Lemma~\ref{lem:frozen-causal-realization}.  Then, for a.e. \(u\),
\begin{equation}\label{eq:canonical-causal-representation}
 \bar X^u=\Phi_X^{\bar{\mathbf q},u}(\bar\xi^u,\bar W),\quad
 \bar Y^u=\Phi_Y^{\bar{\mathbf q},u}(\bar\xi^u,\bar W),\quad
 [\bar Z^u]=\Phi_Z^{\bar{\mathbf q},u}(\bar\xi^u,\bar W).
\end{equation}
On the original independent noise space, define
\begin{equation}\label{eq:physical-transfer-map}
 X^u=\Phi_X^{\bar{\mathbf q},u}(\xi^u,W^u),\quad
 Y^u=\Phi_Y^{\bar{\mathbf q},u}(\xi^u,W^u),\quad
 [Z^u]=\Phi_Z^{\bar{\mathbf q},u}(\xi^u,W^u).
\end{equation}
Then $(X^u,Y^u,Z^u)_{u\in I}$ is a solution for the original independent noise system \eqref{eq:fbo}, and with the representation \eqref{eq:frozen-feedback-map} \(\alpha_t^u=a_{\bar{\mathbf q}}^u(t,W^u,\xi^u)\in\cM\mathbb H_T^2(A)\). Moreover,
\begin{equation}\label{eq:full-law-transfer}
 \cL(\xi^u,W^u,X^u,Y^u,[Z^u],[\alpha^u])
 =
 \cL(\bar\xi^u,\bar W,\bar X^u,\bar Y^u,[\bar Z^u],
       [\bar\alpha^u])
\end{equation}
for a.e. \(u\).
\end{lemma}

\begin{proof}
The canonical solution solves \({\rm FBSDE}(u,\bar{\mathbf q};\bar{\xi}^u,\bar{W})\) for a.e. $u$.  Pathwise
uniqueness for the frozen fully coupled FBSDE is standard under the stated
assumptions. See, e.g., \cite[Chapter~8]{zhang2017backward}.  Therefore,
Lemma~\ref{lem:frozen-causal-realization}, evaluated at
\((\bar\xi^u,\bar W)\), yields \eqref{eq:canonical-causal-representation}.
Note that the input pairs have the same law
\[
 \cL(\xi^u,W^u)=\mu_0^u\otimes\mathbb W
              =\cL(\bar\xi^u,\bar W).
\]
By \eqref{eq:frozen-feedback-map}, the two controls are the same measurable
functional of the corresponding input pair.  Applying this map together with
\(\Phi^{\bar{\mathbf q}}\) gives
 \eqref{eq:full-law-transfer}. Then by Fubini and disintegration in \(t\), \eqref{eq:full-law-transfer} identifies the occupation
kernels with respect to \(dt\) and gives \(q_t^u=\bar q_t^u\) for \(du\otimes dt\)-a.e.. Continuity
of \(X\) gives \(\mu_t^u=\bar\mu_t^u\) for every \(t\) after the chosen
versions. Consequently the law profile generated by
\eqref{eq:physical-transfer-map} is exactly \(\bar{\mathbf q}\) in the
stated equivalence sense. Note that the coefficients with frozen mean field are the graphon coefficients generated by the original noise $(\xi^u,W^u)$ family itself. Hence \eqref{eq:physical-transfer-map} solve the original noise system \eqref{eq:fbo}. \end{proof}
}

\begin{theorem}\label{thm:exfb}
	Let  $\alpha\in\cM\bH^2_T$ be an admissible control profile and let $X=X^\alpha$ be the corresponding state process satisfying~\eqref{eq:gcd}. Suppose Assumptions~\ref{ass:dynamic-1}--\ref{ass:dynamic-5} hold and assume
	\begin{equation}
		\label{eq:ass:b1}
		\int_{ I}\left\{\EE \int_{0}^T \left[ \vert \partial_{x} f^u(t,X^u_{t},\cG^u_t,\alpha^u_{t}) \vert^2  + \int_I\tilde{G}(v,u)
		\t \EE \bigl[ \vert \partial_{\mu} f^v(t,\t X^v_{t},\cG^v_t,\t \alpha^v_{t})(X^u_{t}) \vert^2 \bigr]dv\right] dt \right\}du
		< + \infty,
	\end{equation}
	and
	\begin{equation}
		\label{eq:ass:b2}
		\int_{I}\left\{\EE\Big[ | \partial_{x} g^u (X^u_T,\cG^u_T)|^2 + \int_I\tilde{G}(v,u) \t\EE [| \partial_{\mu} g^v(\t X^v_{T},\cG^v_T)(X^u_T)|^2]  dv \Big]\right\}du
		< \infty.
	\end{equation}
Then there exists a unique solution to \eqref{eq:fb}.
\end{theorem}

\begin{proof}
	 It suffices to construct a unique solution $(\o X^u,\o Y^u,\o Z^u)_{u\in I}\in L^2(I;\bar{\mathbf{S}}^2)\times L^2(I;\bar{\mathbf{S}}^2)\times L^2(I;\bar{\mathbf{H}}^2)$ to \eqref{eq:fbww} whose kernel of full path laws is measurable as in \eqref{eq:full-solution-law-kernel}. Since $\o\alpha\in \overline{\cM\bH}^2_T$ is fixed, the forward equation is independent of $(\o Y,\o Z)$, and Lemma~\ref{lem:barx} gives its unique solution. It remains to solve the adjoint equation.
    
    We first verify the Lipschitz continuity of the driver.
    Equation \eqref{eq:ass:b2} guarantees the terminal condition is in $L^2(\o \Omega;\RR^d)$ for each $u\in I$. Notice that $\partial_{x} b$ and $\partial_{x} \sigma$ are bounded since $b$ and $\sigma$ are Lipschitz continuous in $x$ by Assumption~\ref{ass:dynamic-3}. Since $H$ is linear in $y$ and $z$, $\partial_{x}H$ is Lipschitz continuous in both $y$ and $z$. We also have that for some constant $C$, 
    \begin{align*}
    	\bigg(\int_I \tilde{G}(u,v)\EE [ \vert \partial_{\mu} b^u(t,\t {\o X}^u_t,\o \cG^u_t,\t{\o\alpha}^u_t)(\o X^v_t) \vert^2] dv\bigg)^{\frac 1 2}
	&\leq C,
	\\
	\bigg(\int_I \tilde{G}(u,v)\EE [ \vert \partial_{\mu} \sigma^u(t,\t {\o X}^u_t,\o\cG^u_t,\t{\o \alpha}^u_t)(\o X^v_t) \vert^2]dv\bigg)^{\frac 1 2}
	&\leq C,
    \end{align*}
since $\o X^u\in \mathbf{S}^2(\o\FF)$ by Lemma~\ref{lem:barx}, and 
	$\mu \mapsto b^u(t,x,\mu,a)$ and $\mu \mapsto \sigma^u(t,x,\mu,a)$ are assumed to be Lipschitz continuous with respect to the Wasserstein-2 distance again by Assumption~\ref{ass:dynamic-3}, see more details in \cite[Remark 5.27]{carmonabook1}. 
    
	
	Next, we verify the Lipschitz continuity of $\int_I\EE[\int_I \tilde{G}(u,v)\t\EE[\partial_\mu \t H^u(t,\t{\o X}^u_t,\o \cG^u_t,\t Y^u_t,\t {\o Z}^u_t,\t{\o \alpha}^u_t)(\o X^v_t)]du]dv$ with respect to the integration of the family of laws $(\cL(\o Y^u_t,\o Z^u_t))_{u\in I}$. It involves an expectation over the independent copy $(\t{\o X},\t{\o Y},\t{\o  Z},\t{\o\alpha})$ of $(\o X,\o Y,\o Z,\o\alpha)$. 
    Using the above two displayed inequalities, Fubini's theorem and the Cauchy--Schwarz inequality, we have there exists some constant $C$, such that 
	\begin{equation*}
		\begin{split}
			&\int_I\int_I\tilde{G}(u,v){\mathbb E} \t \EE \bigl[ \vert \partial_\mu b^u(t,\t {\o X}^u_t,\o \cG^u_t,\t{\o\alpha}^u_t)(\o X^v_t)\cdot \t {\o Y}^u_t) \vert^2 \bigr]dudv \leq C \int_I\t\EE [ \vert \t {\o Y}^u_{t} \vert^2 ]du=C\int_I \EE [ \vert \o Y^u_{t} \vert^2 ]du,
			\\
			&\int_I\int_I\tilde{G}(u,v){\mathbb E} \t \EE \bigl[ \vert \partial_\mu \sigma^u(t,\t {\o X}^u_t,\o \cG^u_t,\t{\o\alpha}^u_t)(\o X^v_t)\cdot \t {\o Z}^u
            _t) \vert^2 \bigr]dudv \leq C \int_I\t\EE [ \vert \t {\o Z}^u_{t} \vert^2 ]du=C\int_I \EE [ \vert \o Z^u_{t} \vert^2 ]du.
		\end{split}
	\end{equation*}
	Standard contraction arguments for graphon mean field BSDEs (see, e.g., \cite{amicaosu2022graphonbsde}) and \eqref{eq:ass:b1}--\eqref{eq:ass:b2} yield a unique family $(\o X^u,\o Y^u,\o Z^u)_{u\in I}\in L^2(I;\bar{\mathbf{S}}^2)\times L^2(I;\bar{\mathbf{S}}^2)\times L^2(I;\bar{\mathbf{H}}^2)$ satisfying \eqref{eq:fbww} and
	$$
	\int_{I}\big[\|\o X^u\|^2_{\mathbf{S}^2} +  \|\o Y^u\|^2_{\mathbf{S}^2} + \|\o Z^u\|^2_{\mathbf{H}^2} \big]du < \infty.
	$$

Then similarly as proofs in Lemma~\ref{lem:barx}, we need to show for the adjoint equation, plugging with the joint state and control $(\o X^u,\o\alpha^u)_{u\in I}$ of the forward dynamics, at each iteration step n: $$(\o Y^{u,n},\o Z^{u,n})_{u\in I}\in L^2(I;\bar{\mathbf{S}}^2)\times L^2(I;\bar{\mathbf{H}}^2)$$ and satisfies $u\mapsto \cL(\o Y^{u,n}_t,\o Z^{u,n}_t)$ is measurable for a.e. $t\in[0,T]$. It is trivial that we can choose $(\o Y^{u,0},\o Z^{u,0})_{u\in I}$ at initial step such the required property holds. Now suppose at step $n-1, n\geq 1$, the required property holds, we aim to prove the same property holds at step $n$. 
    
    
    \rv{Let $F^u$ denote the driver of label $u$ in the adjoint equation, in particular, we simply denote
    \begin{align*}
    & F^u\big(t,\o Y^{u,n-1}_t,\o Z^{u,n-1}_t,\cL(\o Y^{\cdot,n-1}_t,\o Z^{\cdot,n-1}_t)\big) \\
    &=-\partial_xH^u(t,\o X^u_t,\o\cG^u_t,\o Y^{u,n-1}_t,\o Z^{u,n-1}_t,\o\alpha^u_t)-\int_I\tilde{G}(v,u)\t\EE[\partial_\mu  H^v(t,\t {\o X}^v_t,\o\cG^v_t,\t {\o Y}^{v,n-1}_t,\t{\o Z}^{v,n-1}_t,\t{\o\alpha}^v_t)(\o X^u_t)]dv.
    \end{align*}
    Notice that by the martingale representation theorem, we have for any $t\in[0,T]$ and $u\in I$, there exists $\bar{\FF}$-predictable $\o Z^{u,n}$, which satisfies
    \begin{equation}\label{eq:mre}
        \bar{\EE}\bigg[\o Y^{u,n-1}_T+\int_0^T F^u\big(s,\o Y^{u,n-1}_s,\o Z^{u,n-1}_s,\cL(\o Y^{\cdot,n-1}_s,\o Z^{\cdot,n-1}_s)\big)ds\bigg|\bar{\cF}_t\bigg]=\o Y^{u,n}_0+\int_0^t\o Z^{u,n}_sd\o W_s.
    \end{equation}
    By Assumptions~\ref{ass:dynamic-2}--\ref{ass:dynamic-3} and \ref{ass:dynamic-5} and  employing similar arguments used in the proof of Lemma~\ref{lem:barx}, we have that 
    $$u\mapsto (\o Y^{u,n-1}_T,\int_0^T F^u\big(s,\o Y^{u,n-1}_s,\o Z^{u,n-1}_s,\cL(\o Y^{\cdot,n-1}_s,\o Z^{\cdot,n-1}_s)\big)ds)\in L^2(\o\Omega;\RR^d\times\RR^d)$$
    is measurable. Since conditional expectation is a contraction, by following similar arguments as in the proof of \cite[Lemma 2.2]{wu2022} and choosing the conditional expectation at time $0$, we get $u\mapsto\o Y^{u,n}_0\in L^2(\o\Omega;\RR^d)$ is measurable. Then by varying $t\in[0,T]$ for every $t\in[0,T]$, we obtain $u\mapsto \int_0^t\o Z^{u,n}_sd\o W_s\in L^2(\o\Omega;\RR^d)$, and hence $u\mapsto\o Y^{u,n}_t\in L^2(\o\Omega;\RR^d)$ are also measurable for any $t\in[0,T]$. Similar arguments as in the proof of Lemma~\ref{lem:barx} give that $u\mapsto\o Y^{u,n}\in \bar{\mathbf{S}}^2$ is measurable.
    
    Notice that for every $u\in I$, $\int_0^\cdot\o Z^{u,n}_sd\o W_s$ is a square integrable martingale starting from the origin. By the It\^o isometry, we have for any $u_1,u_2\in I$,
 $$
    \EE\bigg[\int_0^T|\o Z^{u_1,n}_s-\o Z^{u_2,n}_s|^2ds\bigg]=\EE\bigg[\bigg|\int_0^T(\o Z^{u_1,n}_s-\o Z^{u_2,n}_s)d\o W_s\bigg|^2\bigg],
$$
	    which means $\bar{\mathbf{H}}^2\ni \o Z^{u,n}\mapsto  \int_0^T\o Z^{u,n}_sd\o W_s\in L^2(\o\Omega;\RR^d)$ is an isometric map. Hence the inverse of this isometry on its range is continuous. This suffices to give that $I\ni u\mapsto\o Z^{u,n}\in\bar{\mathbf{H}}^2$ is measurable. \rv{Since $\bar{\mathbf H}^2$ is a separable space of predictable equivalence classes in $L^2(dt\otimes d\bar\PP)$, there exists a jointly measurable predictable version in $(u,t,\bar\omega)$. Fubini's theorem then implies that $u\mapsto\cL(\bar Z_t^{u,n})$ is measurable for a.e. $t$.}} 
    

    

\rv{Therefore, we have proved at each iteration step $n$, 
$$I\ni u\mapsto (\bar{Y}^{u,n},\o Z^{u,n})\in \bar{\mathbf{S}}^2\times \bar{\mathbf{H}}^2$$
is measurable, and meanwhile $u\mapsto\cL(\bar{Y}^{u,n}_t,\bar{Z}^{u,n}_t)$ is also measurable for a.e. $t\in[0,T]$. The integrability follows from the standard BSDE estimate. Hence we have $(\o Y^{u,n},\o Z^{u,n})_{u\in I}\in L^2(I;\bar{\mathbf{S}}^2)\times L^2(I;\bar{\mathbf{H}}^2)$.
}

\rv{Since all canonical variables take values in common separable spaces, the
established measurability of
\(u\mapsto(\bar\xi^u,\bar X^u,\bar Y^u,\bar Z^u,\bar\alpha^u)\) with
\(\bar W\) implies that
\[
 u\longmapsto
 \cL(\bar\xi^u,\bar W,\bar X^u,\bar Y^u,
             [\bar Z^u],[\bar\alpha^u])
 \in\cP_2(\mathsf E)
\]
is measurable. By using Lemma~\ref{lem:frozen-law-transfer}, we have on the original space, the uniquely identified solution $(X^u,Y^u,Z^u)_{u\in I}$ associated with $(\o X^u,\o Y^u,\o Z^u)_{u\in I}$ (here in a decoupled case) solve \eqref{eq:fb} with the prescribed control $\alpha$, and moreover,
\[
 \cL(\xi^u,W^u,X^{u},Y^u,[Z^u],[\alpha^u])
 =\cL(\bar\xi^u,\bar W,\bar X^u,\bar Y^u,
       [\bar Z^u],[\bar\alpha^u]).
\]
The proof is complete.
} 

\end{proof}

\section{Pontryagin Principle of Optimality}
\label{sec:pontryagin}

In this section, we study necessary and sufficient conditions for optimality
when the Hamiltonian satisfies appropriate assumptions of convexity and the coefficients satisfy appropriate differentiability conditions, which will be specified in the sequel.  Throughout the section, we suppose that Assumptions~\ref{ass:dynamic-1}-\ref{ass:dynamic-5} hold. In addition, we need the following regularity properties.

\vspace{5pt}

\begin{assumption}%
\label{ass:c}
\begin{enumerate}[label=\normalfont{\textbf{(\ref*{ass:c}\arabic{*})}}, ref=\normalfont{\textbf{(\ref*
{ass:c}\arabic{*})}}, topsep=2pt,itemsep=0pt,partopsep=0pt, parsep=0pt]
	\item \label{ass:c-1} For each $u\in I$, the functions $b^u, \sigma^u, f^u$ are differentiable with respect to $(x,\mu,\alpha)$, the mappings $(x,\mu,\alpha) \mapsto 
	\partial_{x} (b^u,\sigma^u,f^u) (t,x,\mu,\alpha)$ and $(x,\mu,\alpha) \mapsto 
	\partial_{\alpha} (b^u,\sigma^u,f^u) (t,x,\mu,\alpha)$ are continuous for any $t \in [0,T]$, the mapping 
	$\RR^d \times 
	L^2(\Omega;\RR^d) 
	\times A \ni (x,X,\alpha) \mapsto \partial_{\mu} (b^u,\sigma^u,f^u)(t,x,\PP_{X},\alpha)(X) \in L^2(\Omega;\RR^{d \times d} \times \RR^{(d \times m) \times d} \times \RR^d)$ is continuous for any 
	$t \in [0,T]$. Similarly, the function $g^u$ is differentiable with respect to $x$ and $\mu$, and the partial derivatives are continuous.
	\vspace{2pt}
	
	\item \label{ass:c-2} $((b^u,\sigma^u,f^u)(t,0,\delta_{0},0))_{0 \leq t \leq T}$ are uniformly bounded. 
	The partial derivatives $\partial_{x} (b^u,\sigma^u)$ and $\partial_{\alpha} (b^u,\sigma^u)$ are bounded uniformly in $(u,t,x,\mu,\alpha)$. The norm of the mappings $x' \mapsto
	\partial_{\mu}(b^u,\sigma^u)(t,x,\mu,\alpha)(x')$ in $L^2(\RR^d,\mu)$ is also uniformly bounded in $(u,t,x,\mu,\alpha)$.  There exists a constant $L$ such that, for all $u\in I$, any $R \geq 0$ and $(t,x,\mu,\alpha)$ such that $\vert x \vert, \| \mu \|_{2}=(\int_{\RR^d}|x|^2d\mu(x))^{1/2},\vert \alpha \vert \leq R$, $\vert \partial_{(x,\alpha)} f^u(t,x,\mu,\alpha) \vert$ 
	and $\vert \partial_{x} g^u(x,\mu) \vert$ are bounded by $L(1+R)$ and the norm in $L^2(\RR^d,\mu)$ of 
	$x' \mapsto
	\bigl(\partial_{\mu}f^u(t,x,\mu,\alpha)(x'),\partial_{\mu}g^u(x,\mu)(x')\bigr)$ is bounded by $L(1+R)$. 
	
	\item \label{ass:c-3} The action set $A$ is convex.

	\item \label{ass:c-4} $[0,T]\times \RR^d \times \cP_2(\RR^d)\times A\times I \ni (t,x,\mu,\alpha,u) \mapsto \partial_{\alpha}(b^u,\sigma^u,f^u)(t,x, \mu, \alpha)$ is jointly measurable.
\end{enumerate}
\end{assumption}

Conditions~\ref{ass:c-1}--\ref{ass:c-2} are adapted from \cite{carmona2015forward}, but here we state them label by label.
Assumption~\ref{ass:c-4} is used to guarantee the measurability since the calculations will involve the derivative with respect to $\alpha$.

\subsection{Necessary condition}
\label{subsec:4:1}

In this subsection, we study a necessary condition of optimality. Let us denote by $\cM A$ the set of all measurable functions $I\ni u \mapsto a^u \in A$ such that $\int_I |a^u|^2 du< \infty$. The main result is the following. 

\begin{theorem}\label{thm:nece}
	Assume for each $u\in I$, the Hamiltonian $H^u$ is convex in $\alpha$. Let $\alpha\in\cM\bH^2_T$ be an optimal control. Let $ X$ be the associated optimally controlled state, and $( Y,  Z)$ be the associated adjoint processes solving the adjoint equation (see~\eqref{eq:fb}). Under Assumptions~\ref{ass:c-1}--\ref{ass:c-4} we have, for all $\beta\in \cM\bH^2_T$, 
	$$
	\int_I \EE\bigl[	H^u(t,X^u_t,\cG^u_t,Y^u_t,Z^u_t,\alpha^u_t)\bigr]du\le \int_I \EE\bigl[H^u(t,X^u_t,\cG^u_t,Y^u_t,Z^u_t,\beta^u_t)\bigr]du, \quad dt-a.e. 
	$$
	Furthermore, we have for a.e. $u\in I$, for any $\gamma\in \cM A$, 
	\begin{equation}
		\label{fo:Pminimum}
		H^u(t,X^u_t,\cG^u_t,Y^u_t,Z^u_t,\alpha^u_t)\le H^u(t,X^u_t,\cG^u_t,Y^u_t,Z^u_t,\gamma^u), \quad d\PP\otimes dt\text{-a.e.}
	\end{equation}
\end{theorem}

Before providing its proof, we introduce some preliminary results.
We note that as a consequence of Assumption~\ref{ass:c-3}, the set of admissible controls $\cM\bH^2_T$ is convex. For a fixed $\alpha\in\cM\bH^2_T$, we denote by $(X^u)_{u\in I}$ the corresponding controlled state process, namely the solution of \eqref{eq:gcd} where the whole system is controlled by $\alpha$ with given initial condition $X_0=\xi$. We will compute the G\^ateaux derivative of the cost functional $J$ at any $\alpha\in\cM\bH^2_T$ in all possible directions. \rv{To do this, for any $\check{\alpha}\in\cM\bH^2_T$ close enough to $\alpha$, there exists $\beta=(\beta^u)_{u\in I}$ satisfying for each $u\in I, \beta^u\in\bH^2_{T,u}(\RR^k)$, and such that $\check{\alpha}=\alpha + \epsilon \beta \in \cM\bH^2_T$ for $\epsilon >0$ small enough. For any given ${\alpha}\in\cM\bH^2_T$, we call a direction $\beta$ admissible if $\alpha + \epsilon \beta \in \cM\bH^2_T$ for $\epsilon >0$ small enough.} We then compute the variation of $J$ at $\alpha$ in the target admissible direction $\beta$.

We use the notation $\Phi^u_{t}:= (X^u_{t},\cG^u_t,\alpha^u_{t})$, and then define the variation process $ V=(V^u_t)_{0\le t\le T,u\in I}$ as the solution of the following graphon SDE system:
\begin{equation}
	\label{fo:Voft}
	\begin{split}
		dV^u_t& =  \bigl[\partial_xb^u(t,\Phi^u_{t}) \cdot V^u_t+\int_ I \tilde{G}(u,v)\t\EE \bigl[ \partial_\mu b^u(t,\Phi^u_{t})(\t X^v_t)\cdot \t V^v_t \bigr]dv+\partial_\alpha b^u(t,\Phi^u_{t}) \cdot \beta^u_t \bigr]dt \nonumber \\
		& +\Big[\partial_x\sigma^u(t,\Phi^u_{t})\cdot V^u_t+\int_I\tilde{G}(u,v)\t\EE \bigl[
		\partial_\mu \sigma^u(t,\Phi^u_{t})(\t X^v_t)\cdot \t V^v_t\big]dv+\partial_\alpha \sigma^u(t,\Phi^u_{t})\cdot\beta^u_t \Big]dW^u_t,
	\end{split}
\end{equation}
with $V^u_{0}=0$. 
Similarly as the proof of Theorem~\ref{thm:exfb}, under Assumptions~\ref{ass:c-1}-\ref{ass:c-2}, we have for all $u\in I$, 
$$\int_ I \tilde{G}(u,v)\t\EE \bigl[ \partial_\mu b^u(t,\Phi^u_{t})(\t X^v_t)\cdot \t V^v_t \bigr]dv\leq C_1\int_{I}\EE \bigl[ \vert V^v_{t} \vert^2 \bigr]^{1/2}dv,$$
$$\int_ I \tilde{G}(u,v)\t\EE \bigl[ \partial_\mu \sigma^u(t,\Phi^u_{t})(\t X^v_t)\cdot \t V^v_t \bigr]dv\leq C_2\int_{I}\EE \bigl[ \vert V^v_{t} \vert^2 \bigr]^{1/2}dv$$.
Then following classical contraction arguments for forward graphon systems (see \cite{erhangmfs}), the existence and uniqueness of the variation process is guaranteed by the Lipschitz continuity of the coefficients in the dynamics of $(V^u)_{u\in I}$. Moreover, %
$\int_{I}\| V^u
\|^2_{\mathbf{S}^2}du<\infty,$ 
and $u\mapsto \cL(V^u_t)$ is measurable for $t \in [0,T]$.

We first compute an expression for the derivative of the cost. We denote by $J^u(\alpha)$ the total cost of label $u$ under control profile $\alpha$, i.e.
$$
    J^u(\alpha)=\EE\left[\int_0^T f^u(t,X^u_t,\cG^u_t,\alpha^u_t)dt+g^u(X^u_T,\cG^u_T)\right].
$$

\begin{lemma}
	\label{le:gateaux}
	Under Assumptions~\ref{ass:c-1}--\ref{ass:c-3}, $\cM\bH^2_T\ni\alpha\mapsto J(\alpha)\in \RR$ is G\^ateaux differentiable and its derivative \rv{in admissible direction $\beta$} is given by:
	\begin{equation}
		\label{eq:gat}
		\begin{split}
			&\frac{d}{d\epsilon}J(\alpha+\epsilon\beta)\big|_{\epsilon=0} \\
            &=\int_I\EE\bigg[\int_0^T \Big(\partial_x f^u(t,\Phi^u_t) \cdot V^u_t+\int_I \tilde{G}(u,v)\t\EE
			[\partial_\mu 
			f^u(t,\Phi^u_t)(\t X^v_t)\cdot\t V^v_t ]dv+\partial_\alpha f^u(t,\Phi^u_t) \cdot \beta^u_t \Big)dt\bigg]du\\
			&\qquad\qquad +\int_I\EE \left[ \partial_x g^u(X^u_T,\cG^u_T) \cdot V^u_T + \int_I \tilde{G}(u,v)\t \EE [\partial_\mu g^u(X^u_T,\cG^u_T)(\t X^v_T)\cdot \t V^v_T]dv\right]du.
		\end{split}
	\end{equation}
\end{lemma} 
The proof is given in Appendix~\ref{app:proofs-sec4}.

\begin{lemma}
	\label{le:duality}
	Let $\beta$ be an admissible direction. Under Assumptions~\ref{ass:c-1}--\ref{ass:c-3}, for any $u\in I$, the following holds: 
	\begin{align*}
		&\EE[Y^u_T \cdot V^u_T]
        \\
        &=\EE\int_0^T \Big[Y^u_t \cdot \bigl( \partial_\alpha b^u(t,\Phi^u_t) \cdot \beta^u_t \bigr) 
		+Z^u_t \cdot \bigl( \partial_\alpha\sigma^u(t,\Phi^u_t) \cdot \beta^u_t \bigr) 
		\nonumber\\
		&\phantom{?}-\partial_xf^u(t,\Phi^u_t) \cdot V^u_t -\int_I \tilde{G}(v,u)
		\EE\t\EE \bigl[ \partial_\mu f^v(t,\Phi^v_t)(\t X^u_t)\cdot \t V^u_t \bigr]dv \\
		&\phantom{?}+Z^u_t \cdot \int_I\tilde{G}(u,v)\t\EE \bigl[ \partial_\mu \sigma^u(t,\Phi^u_t)(\t X^v_t)\cdot\t V^v_t \bigr]dv+Y^u_t \cdot \int_I \tilde{G}(u,v)\t\EE \bigl[ \partial_\mu b^u(t,\Phi^u_t)(\t X^v_t)\cdot\t V^v_t \bigr]dv\\
		&\phantom{?}- \int_I\tilde{G}(v,u)\EE\t\EE \bigl[Z^v_t \cdot \partial_\mu \sigma^v(t,\Phi^v_t)(\t X^u_t)\cdot\t V^u_t \bigr]dv- \int_I \tilde{G}(v,u)\EE\t\EE \bigl[ Y^v_t \cdot\partial_\mu b^v(t,\Phi^v_t)(\t X^u_t)\cdot\t V^u_t \bigr]dv
		\Big]\,dt.
	\end{align*}
\end{lemma} 
The proof is given in Appendix~\ref{app:proofs-sec4}.

\begin{lemma}\label{thm:j}
	Under Assumptions~\ref{ass:c-1}--\ref{ass:c-4}, $J$ is G\^ateaux differentiable and the G\^ateaux derivative of $J$ at $\alpha\in\cM\bH^2_T$ \rv{in admissible direction $\beta$} has the following form:
	\begin{equation}
		\frac{d}{d\epsilon}J(\alpha+\epsilon\beta)\big|_{\epsilon=0}=\int_I \EE\left[\int_0^T \partial_\alpha H^u(t,X^u_t,\cG^u_t,Y^u_t,Z^u_t,\alpha^u_t) \cdot \beta^u_tdt\right]du. 
	\end{equation}
\end{lemma} 
The proof is given in Appendix~\ref{app:proofs-sec4}. 
\begin{proof}[Proof of Theorem~\ref{thm:nece}]
Since $A$ is open and convex, for any $\beta\in\cM\bH^2_T$ we can choose $\epsilon$ small enough such that the perturbation $\alpha^\epsilon_t=\alpha_t+\epsilon (\beta_t-\alpha_t)$ remains in $\cM\bH^2_T$. By the optimality of $\alpha$ and Lemma~\ref{thm:j},
$$
    0 \le \frac{d}{d\epsilon}J(\alpha+\epsilon(\beta-\alpha))\big|_{\epsilon=0}
    =\int_I\EE\left[\int_0^T\partial_\alpha H^u(t,X^u_t,\cG^u_t,Y^u_t,Z^u_t,\alpha^u_t) \cdot (\beta^u_t-\alpha^u_t)dt\right].
$$
By convexity of the Hamiltonian with respect to $\alpha$, for all $u\in I$,
\begin{equation*}
	\begin{split}
		H^u(t,X^u_t,\cG^u_t,Y^u_t,Z^u_t,\beta^{u}_t)-H^u(t,X^u_t,\cG^u_t,Y^u_t,Z^u_t,\alpha^u_t)&\geq 
		\partial_\alpha H^u(t,X^u_t,\cG^u_t,Y^u_t,Z^u_t,\alpha^u_t) \cdot (\beta^u_t-\alpha^u_t).
	\end{split}
\end{equation*}
It follows that for all $\beta\in\cM\bH^2_T$,
$$
    \int_I\int_0^T\EE\bigl[	H^u(t,X^u_t,\cG^u_t,Y^u_t,Z^u_t,\beta^{u}_t)-H^u(t,X^u_t,\cG^u_t,Y^u_t,Z^u_t,\alpha^u_t)\bigr] dtdu\ge 0.
$$

Next, we prove the second result in the statement. \rv{Let $(\gamma^u)_{u\in I}\in \cM A$, $\cI\in{\mathcal B}(I)$, and $D$ be a jointly measurable causal subset of
$I\times[0,T]\times\RR^d\times\cC^m$, namely ${\bf1}_D(u,t,x,w)={\bf1}_D(u,t,x,w_{\cdot\wedge t})$.  For each $u$ set
\[
 B_D^u=\{(t,\omega):(u,t,\xi^u(\omega),W^u(\omega))\in D\}.
\]
Then $B_D^u$ is progressive for the filtration generated by
$(\xi^u,W^u)$.}
We define $\beta$ by:
$$
\beta^{u}_t(\omega)
= \begin{cases}
	\gamma^u,&\text{ if } u\in\cI \text{ and }(t,\omega)\in \rv{B_D^u},\\
	\alpha^u_t(\omega),&\text{ otherwise.}
\end{cases}
$$
\rv{Indeed for each label $u$, we can define the map
\[
 b^u(t,w,x)=
 \gamma^u{\bf1}_{\cI}(u){\bf1}_{D}(u,t,x,w)
 +a^u(t,w_{\cdot\wedge t},x)
  \bigl(1-{\bf1}_{\cI}(u){\bf1}_{D}(u,t,x,w)\bigr).
\]
After composition with $x=\rho^u(r)$,  by Definition~\ref{def:admisscontrol} for $\alpha$ and the joint measurability of
$D$, $u\mapsto b^u$ is measurable
in $L^2([0,T]\times\cC^m\times\Upsilon;
      A,dt\otimes\mathbb W\otimes dr)$. The integrability also follows from the definition of $\cM A$. Thus $\beta$ is admissible.} Then we have that 
$$
\int_I\int_0^T \EE\big[{\bf 1}_{\cI}(u){\bf 1}_{\rv{B_D^u}}[	H^u(t,X^u_t,\cG^u_t,Y^u_t,Z^u_t,\beta^{u}_t)-H^u(t,X^u_t,\cG^u_t,Y^u_t,Z^u_t,\alpha^u_t)]\bigr] dtdu \ge 0. 
$$
\rv{Since $D$ and $\cI$ are arbitrary, we can conclude
\eqref{fo:Pminimum}.}
\end{proof}

Without the convexity of $A$ and $H$, the following weak necessary condition holds.

\begin{proposition}[Necessary optimality condition, weak version]
\label{prop:nece-weak}
		Suppose Assumptions~\ref{ass:c-1}, \ref{ass:c-2} and~\ref{ass:c-4} hold. Let $\alpha\in\cM\bH^2_T$ be optimal, $X$ be the associated optimally controlled state, and $(Y, Z)$ be the associated adjoint processes solving the adjoint equation (see~\eqref{eq:fb}). Then, for a.e. $u\in I$,
	$$
		\partial_\alpha H^u(t,X^u_t,\cG^u_t,Y^u_t,Z^u_t,\alpha^u_t)=0, \quad d\PP\otimes dt\text{-a.e.}
	$$
\end{proposition}

\begin{proof}
We prove the result by contradiction. Let $\epsilon_0>0$ and $\gamma\in\RR^k$ with $|\gamma|=1$. 
Similarly, for any given $\cI\in{\mathcal B}(I)$ and \rv{$(B_D^u)_{u\in I}$ defined from a jointly measurable causal set $D$ as in the proof of Theorem~\ref{thm:nece}}, we define
$$
    \beta^u_t=\gamma{\bf 1}_{\{u\in \cI\}}{\bf 1}_{\rv{B_D^u}\cap \{\mathrm{dist}(\alpha^u_t,\partial A)>\rv{\epsilon_0}\}},
$$
where $\partial A$ means the boundary of $A$. By construction, $\alpha^u_t+\epsilon\beta^u_t\in A$ for all $t\in[0,T]$ and \rv{$\epsilon\in(0,\epsilon_0)$}. Then, following the proof of Theorem~\ref{thm:nece}, we have
$\int_0^T \int_I \EE \bigl[\partial_\alpha H^u(t,X^u_t,\cG^u_t,Y^u_t,Z^u_t,\alpha^u_t) \cdot \beta^u_t \bigr]dudt\ge 0,$
from which we obtain that outside a Lebesgue null set $\mathfrak{H}$ in $I$, for every $u\in I$, 
$$ 
    {\bf 1}_{\{\mathrm{dist}(\alpha^u_t,\partial A)>\epsilon_0\}}\partial_\alpha H^u(t,X^u_t,\cG^u_t,Y^u_t,Z^u_t,\alpha^u_t) \cdot \gamma \ge 0, \quad d\PP\otimes dt\text{-a.e.}
$$
Moreover, since $\gamma$ is arbitrary on the unit ball, letting $\epsilon_0\to 0$, we get that
$${\bf 1}_{\{\mathrm{dist}(\alpha^u_t,\partial A)>0\}}\partial_\alpha H^u(t,X^u_t,\cG^u_t,Y^u_t,Z^u_t,\alpha^u_t)=0, \quad d\PP\otimes dt\text{-a.e.},$$
for $u\in I\setminus\mathfrak{H}$.
Since $A$ is open, we can conclude.
\end{proof}

\subsection{Sufficient condition}
In this subsection, we establish a sufficient condition of optimality, with a few extra assumptions.

\begin{theorem}
	\label{th:suffi}
	Suppose Assumptions~\ref{ass:c-1}--\ref{ass:c-4} hold. Let $\alpha\in\cM \bH^2_T$ be an admissible control profile, $X= X^{\alpha}$ be the corresponding controlled state process, and $( Y, Z)$ the corresponding adjoint processes. Suppose in addition that 
	\begin{enumerate}[topsep=0pt,itemsep=0pt,partopsep=0pt, parsep=0pt]
		\item For a.e. $u\in I$, $ \RR^d \times \cP_{2}(\RR^d) \ni (x,\mu)\mapsto g^u(x,\mu)$ is convex.
		\item For a.e. $u\in I$, $\RR^d \times \cP_{2}(\RR^d) \times A \ni (x,\mu,a)\mapsto H^u(t,x,\mu,Y^u_t,Z^u_t,a)$ 
		is convex $d\PP\otimes dt$-a.e. and
		$H^u(t,X^u_t,\cG^u_t,Y^u_t,Z^u_t, \alpha^u_t)=\inf_{a\in A}H^u(t,X^u_t,\cG^u_t,Y^u_t,Z^u_t,a)$, $d\PP\otimes dt$-a.e.
	\end{enumerate} 
	Then $\alpha$ is an optimal control.
\end{theorem}

\begin{proof}
Let $ \alpha'\in\cM\bH^2_T$ be an admissible control, and $ \bX= X^{\alpha'}$ be the corresponding controlled state.
Then we calculate the difference of the objective functions for each $u\in I$:
\begin{align}
		J^u(\alpha)-J^u(\alpha')&= \EE \bigl[g^u(X^u_T,\cG^u_T)-g^u(\bX^u_T,\bar{\cG}^u_T) \bigr]+ \EE\int_0^T
		\bigl[f^u(t,\theta^u_t)-f^u(t,\o\theta^u_t) \bigr]dt
		\\
		&= \EE \bigl[g^u(X^u_T,\cG^u_T)-g^u(\bX^u_T,\bar{\cG}^u_T) \bigr] +\EE\int_0^T \bigl[H^u(t,\Phi^u_t)-H^u(t,\o\Phi^u_t) \bigr]dt
		\\
		&\qquad- \EE\int_0^T \bigl\{ \bigl[b^u(t,\theta^u_t)-b^u(t,\o\theta^u_t) \bigr]\cdot Y^u_t + 
		\bigl[\sigma^u(t,\theta^u_t)-\sigma^u(t,\o\theta^u_t)]\cdot Z^u_t \bigr\} dt,
	\label{eq:24:1:1}
\end{align}
where $\theta^u_{t} = (X^u_{t},\cG^u_t,\alpha^u_t)$, $\Phi^u_{t} = (X^u_{t},\cG^u_{t},Y^u_{t},Z^u_{t},\alpha^u_{t})$ and similarly for 
$\o\theta_{t}'$ and $\o\Phi^u_{t}$. 
Using the convexity of $g^u$, we have that
\begin{equation*}
	\begin{split}
		&\EE[g^u(X^u_T,\cG^u_T)-g^u(\o X^u_T,\o\cG^u_{T})]
		\\
		&\le \EE\bigl[\partial_xg^u(X^u_T,\cG^u_{T})\cdot (X^u_T-\bX^u_T) + \t\EE [\partial_\mu g^u(X^u_T,\cG^u_{T})(\t\Theta^u_T)\cdot (\t\Theta^u_T-\tilde{\bar{\Theta}}^u_T)]\bigr]
		\\
		&= \EE \bigl[ \partial_xg^u(X^u_T,\cG^u_{T})\cdot (X^u_T-\bX^u_T) +\int_I \tilde{G}(u,v) \t\EE[\partial_\mu g^u( X^u_T, \cG^u_{T})(\t X^v_T)  \cdot (\t X^v_T-\t{\bX}^v_T)]dv\bigr],
	\end{split}
\end{equation*}
where the equality comes from the characterization of $\cG^u_T$ and \rv{$\tilde{\Theta}^u, u\in I$ are a i.i.d. copies of $\Theta^u, u\in I$ with $\Theta^u$ defined above \eqref{eq:fb}, and $\Theta^u,\tilde{\bar{\Theta}}$ are similarly defined with $\cL(\bar{X})$.} Again, integrating over $I$ and using Fubini's theorem, we have
\begin{align}\label{fo:g}
	\int_I \EE[g^u(X^u_T,\cG^u_T)-g^u(\o X^u_T,\o\cG^u_{T})]du
	\le \int_I\EE \bigl[(X^u_T-\bX^u_T)\cdot Y^u_T \bigr]du.
\end{align}
Using the adjoint equation and taking expectations gives
\begin{equation*}
	\begin{split}
		&\int_I\EE \bigl[ (X^u_T-\bX^u_T)\cdot Y^u_T \bigr]du 
		\\
		&=- \int_I\EE \biggl[\int_0^T
		\partial_xH^u(t,\Phi^u_{t}) \cdot (X^u_t-\bX^u_t) + \int_I \tilde{G}(v,u)
		\t\EE \bigl[ \partial_\mu H^v(t,\t \Phi^v_t)(X^u_t) \cdot ( X^u_t-\bX^u_t)  \bigr]dv dt\biggr]du
		\\
		&\qquad\qquad+ \int_I \EE \int_0^T 
		\left[
		[b^u(t,\theta^u_t)-b^u(t,\o\theta^u_t)] \cdot Y^u_{t}   + [\sigma^u(t,\theta^u_t)-\sigma^u(t,\o\theta^u_t)]\cdot  Z^u_t \right] dt du.
	\end{split}
\end{equation*}
By Fubini's theorem,
\begin{align}
		&\int_I\EE \bigg[\int_I \tilde{G}(v,u) \t\EE \bigl[ \partial_\mu H^v(t,\t \Phi^v_t)(X^u_t)  \cdot (X^u_t-\bX^u_t)\bigr]dv\bigg]du 
		\\
		&
		=\int_I \EE \t\EE \bigl[ \partial_\mu H^u(t,\Phi^u_t)(\t \Theta^u_t)\cdot (\t\Theta^u_t - \t{\bar{\Theta}}^u_t) \bigr] du.
	\label{eq:24:1:2}
\end{align}
Then, by \eqref{eq:24:1:1}, \eqref{fo:g}, \eqref{eq:24:1:2}, the second condition in the statement, and the convexity of $H$,
\begin{equation*}
	\begin{split}
		&J(\alpha)-J(\alpha')
		\\
		&\le \int_I\Bigl\{\EE\Bigl[\int_0^T[H^u(t,\Phi^u_t)-H^u(t,\o\Phi^u_t)]dt\Bigl]
		\\
		&\quad-\EE\Bigl[\int_0^T \left\{ \partial_xH^u(t,\Phi^u_t) \cdot (X^u_t-\bX^u_t) + 
		\t\EE \bigl[ \partial_\mu H^u(t, \Phi^u_t)(\t\Theta^u_t) \cdot (\t \Theta^u_{t} - \t{\bar{\Theta}}^u_{t} ) \bigr] \right\} dt\Bigr]\Bigr\}du,
	\end{split}
\end{equation*}
which is nonpositive. Thus we obtain $J(\alpha)\leq J(\alpha')$ for any $\alpha'\in\cM\bH^2_T$.
\end{proof}

\section{Properties of the Optimal FBSDE System}
\label{sec:fbsde}

In this section, we study the graphon FBSDE system under optimal controls resulting from the application of the Pontryagin stochastic maximum principle proved above. We establish the existence and uniqueness of the optimally controlled graphon FBSDE system \eqref{eq:fbo} by the {\it continuation method}. Next, we study two important properties of the system \eqref{eq:fbo}, including continuity and stability. These properties play a crucial role to obtain the approximate optimality results of Section~\ref{sec:chaos}.

\subsection{Setting and assumptions}\label{sec:5.1}
For simplicity, we study the case where the (forward) dynamics are linear. %
In the rest of this section and in the following section, we suppose the following assumption holds.
\begin{assumption}
    \label{asm:Afull-linear}
    $A=\RR^k$ and for each $u\in I$, the drift $b^u$ and the volatility $\sigma^u$ are affine in $x$, $\bar\mu = \int x \mu(dx)$,  and $\alpha$, namely,
		$b^u(t,x,\mu,\alpha) = b^u_{0}(t) + b^u_{1}(t)x + b^u_{2}(t)\o\mu + b^u_{3}(t) \alpha,$ 
    		$\sigma^u(t,x,\mu,\alpha) = \sigma^u_{0}(t) + \sigma^u_{1}(t)x + \sigma^u_{2}(t)\o\mu + \sigma^u_{3}(t) \alpha,$ 
    where $b^u_{0}$, $b^u_{1}$, 
    $b^u_{2}$ and $b^u_{3}$ are bounded measurable functions with values in $\RR^d$, $\RR^{d \times d}$, $\RR^{d \times d}$ and $\RR^{d \times k}$
    and $\sigma^u_{0}$, $\sigma^u_{1}$, $\sigma^u_{2}$ and $\sigma^u_{3}$ are bounded and measurable with values in $\RR^{d \times m}$, 
    $\RR^{(d \times m) \times d}$, $\RR^{(d \times m) \times d}$ and $\RR^{(d \times m) \times k}$.
\end{assumption} 
The parentheses $(d\times m)$ mean that for any $\sigma\in \RR^{(d \times m) \times d}$ and $x\in\RR^d$, $\sigma x \in \RR^{d\times m}$. Notice that since the dynamics is linear, some Lipschitz and differentiability conditions in Assumptions~\ref{ass:dynamic} are automatically satisfied, but the rest conditions are still needed. We state other additional conditions that will be used in this section.

\begin{assumption}%
\label{ass:d}
\begin{enumerate}[label=\normalfont{\textbf{(\ref*{ass:d}\arabic{*})}}, ref=\normalfont{\textbf{(\ref*
{ass:d}\arabic{*})}}, topsep=2pt,itemsep=0pt,partopsep=0pt, parsep=0pt]
    \item \label{ass:d--1}
Suppose Assumption~\ref{ass:dynamic} and \ref{asm:Afull-linear} hold. The mappings $I\times[0,T] \ni (u,t) \mapsto (b^u_0,b^u_1,b^u_2,b^u_3,\sigma^u_0,\sigma^u_1,\sigma^u_2,\sigma^u_3)(t)$ are jointly measurable and bounded. \rv{Suppose $(b^u_3,\sigma^u_3), u\in I$ are the same.} \rv{Suppose there exists a large enough constant $K_b$ such that $x\cdot b^u_1(t)x\leq -K_b|x|^2$ for all $x\in\RR^d$ uniformly in $(u,t)$.} For each $u\in I$, $f^u$ and $g^u$ satisfy the same assumptions as in~\ref{ass:c-1}, \ref{ass:c-2} and~\ref{ass:c-4} in Section~\ref{sec:pontryagin}. In particular, there exists a constant $L$ such that for all $u\in I$,
    \begin{equation*}
    	\begin{split}
    		&\bigl\vert f^u(t,x',\mu',\alpha') - f^u(t,x,\mu,\alpha) \bigr\vert + \bigl\vert g^u(x',\mu') - g^u(x,\mu) \bigr\vert
    		\\
    		&\hspace{15pt} \leq L \bigl[ 1 + \vert x' \vert + \vert x \vert + \vert \alpha' \vert +
    		\vert \alpha \vert + \|  \mu  \|_{2} + \|\mu' \|_{2} \bigr] \bigl[ \vert (x',\alpha') - (x,\alpha) \vert +
    		\cW_{2}(\mu',\mu) \bigr].
    	\end{split}
    \end{equation*}  
    \item \label{ass:d--2} \rv{Suppose $\partial_\alpha f^u, u\in I$ are the same.} There exists a constant \(c>0\), independent of \(u\), such that, for every \(u\in I\), the derivatives of $f^u$ and $g^u$ 
    w.r.t. $(x,\alpha)$ and $x$ respectively
    are Lipschitz continuous with constant $c$. Moreover, for any
    $t \in [0,T]$, any  
    $x,x' \in \RR^d$, any $\alpha,\alpha' \in A$, any 
    $\mu,\mu' \in {\mathcal P}_{2}(\RR^d)$ and 
    any random variables $X$ and $X'$ taking values in $\RR^d$, we have
    
    \begin{equation*}
    	\begin{split}
    		&\EE \bigl[ \vert \partial_{\mu} f^u(t,x',\mu',\alpha')(X') - \partial_{\mu} f^u(t,x,\mu,\alpha)(X) \vert^2 \bigr]
    		\\
    		&\qquad\qquad
    		\leq c \bigl( \vert (x',\alpha') - (x,\alpha) \vert^2 +
    		\EE \bigl[ \vert X'- X \vert^2\bigr]+\cW^2_2(\mu',\mu) \bigr),
    		\\
    		&\EE \bigl[ \vert \partial_{\mu} g^u(x',\mu')(X') - \partial_{\mu} g^u(x,\mu)(X) \vert^2 \bigr] 
    		\leq c \bigl( \vert x'-x \vert^2 + \EE \bigl[ \vert X'- X \vert^2 \bigr]+\cW^2_2(\mu',\mu) \bigr). 
    	\end{split}
    \end{equation*}
    \item \label{ass:d--3} There exists $\lambda>0$ such that, for each $u\in I$ and $t\in [0,T]$, the function $(x,\mu,\alpha)\mapsto f^u(t,x,\mu,\alpha)$ is strongly convex in the sense that for all $x,\mu,x',\mu',\alpha$ and $\alpha'$, 
    \begin{align*}
    		&f^u(t,x',\mu',\alpha') - f^u(t,x,\mu,\alpha) 
    		\\
    		&\hspace{15pt} - \partial_{(x,\alpha)} f^u(t,x,\mu,\alpha)\cdot (x'-x,\alpha'-\alpha)
    		- \tilde{\mathbb E} \bigl[\partial_{\mu} f^u(t,x,\mu,\alpha)(\t X) \cdot (\tilde{X}' - \tilde{X}) \bigr]
    		\\ 
    		&\geq \lambda 
    		\vert \alpha' - \alpha \vert^2.
        \setcounter{counterExtraAssumptions}{\value{enumi}}
    \end{align*} 
    For every $u \in I$, the function $g^u$ is convex in $(x,\mu)$. 
\setcounter{counterExtraAssumptions}{\value{enumi}}
\end{enumerate}
\end{assumption}
\rv{The variables \(X,X'\) in Assumption~\ref{ass:d--2} are arbitrary. In
particular, the assumption gives the Lipschitz bounds required for
the frozen fields \(K^{q,u}\) and \(K^{g,q,u}\).} These assumptions are analogous to assumptions for classical FBSDEs, see \cite[Assumption B]{carmona2015forward}.
Under Assumptions~\ref{ass:d--1}--\ref{ass:d--3}, given $(t,x,\mu,y,z) \in [0,T] \times \RR^d \times {\mathcal P}_{2}(\RR^d) \times \RR^d \times \RR^{d \times m}$, by definition, for each $u\in I$, the function $A \ni \alpha \mapsto H^u(t,x,\mu,y,z,\alpha)$ is strictly convex so that 
there exists a unique minimizer:
\begin{equation}
	\label{fo:alphahat}
	\hat{\alpha}^u(t,x,\mu,y,z) = \textrm{argmin}_{\alpha \in A} H^u(t,x,\mu,y,z,\alpha),
\end{equation}
which is the unique root of  $\alpha\mapsto\partial_{\alpha} H^u(t,x,\mu,y,z,\alpha)$. In addition, for each fixed $u \in I$, following the analysis in \cite{carmona2015forward}, we know that
the mapping $[0,T] \times \RR^d \times {\mathcal P}_{2}(\RR^d) \times \RR^d \times \RR^{d \times m}
\ni (t,x,\mu,y,z)   \mapsto \hat{\alpha}^u(t,x,\mu,y,z)$ is measurable, locally bounded and Lipschitz continuous with respect to 
$(x,\mu,y,z)$, uniformly in $t\in [0,T]$, the Lipschitz constant depending only on the uniform supremum norm of $b^u_3,\sigma^u_3$, the Lipschitz constant of $\partial_\alpha f^u$ in $(x,\mu,\alpha)$, and $\lambda$.
\rv{Moreover, Assumptions~\ref{ass:d--1}--\ref{ass:d--2} give
\[
 \partial_\alpha H^u(t,x,\mu,y,z,\alpha)
 =b_3(t)^\top y+\sigma_3(t)^\top z
   +\partial_\alpha f(t,x,\mu,\alpha),
\]
independent of \(u\).  Hence its unique root \(\hat\alpha^u\) is independent
of \(u\) and jointly measurable in \((u,t,x,\mu,y,z)\).}

We define the following graphon mean field FBSDE system:
\begin{equation}
	\label{eq:fbo}
	\left\{
	\begin{aligned}
		&	dX^u_t= b^u(t,X^u_t,\cG^u_t,\h\alpha^u_t)dt 
		+\sigma^u(t,X^u_t,\cG^u_t,\h\alpha^u_t)dW^u_t,\\
		&dY^u_t=-\partial_xH^u(t,X^u_t,\cG^u_t,Y^u_t,Z^u_t,\h\alpha^u_t)dt + Z^u_t dW^u_t\\
		&\textstyle\qquad\qquad-\int_I\tilde{G}(v,u)\t\EE[\partial_\mu H^v(t,\t X^v_t,\cG^v_t,\t Y^v_t,\t Z^v_t,\t{\h\alpha}^v_t)(X^u_t)]dvdt,\\
		&\textstyle Y^u_T= \partial_xg^u(X^u_T,\cG^u_T)+\int_I\tilde{G}(v,u)\t\EE[\partial_\mu g^v(\t X^v_T,\cG^v_T)(X^u_T)]dv, \quad X^u_0= \xi^u,\quad u\in I ,
	\end{aligned}
	\right.
\end{equation}
where, for chosen predictable representatives, $\hat\alpha^u_t=\hat\alpha^u(t,X^u_t,\cG^u_t,Y^u_t,Z^u_{t})$ with $\hat\alpha^u$ the minimizer of $H^u$ constructed in \eqref{fo:alphahat}, and $\t{\h \alpha}^u_t=\hat{\alpha}^u(t,\t X^u_t,\cG^u_t,\t Y^u_t,\t Z^u_t)$, \rv{for $du$-a.e. $u$ and $dt\otimes d\PP$-a.e. on the original independent noise space.}
\rv{By a solution of \eqref{eq:fbo} on the original independent noise system we mean a quadruple  \((X,Y,Z,\hat\alpha)\) such that \(\hat\alpha\in\cM\mathbb H_T^2(A)\) and \((X,Y,Z)\) is a solution in the sense of Definition~\ref{def:FBSDE}.}


\subsection{Solvability}

We use the continuation method for fully coupled FBSDEs. See
\cite{peng1999cm,carmona2015forward} and its graphon adaptation in
\cite{wu2022}.  We first work on the coupled canonical system and omit the
symbol $\bar{\cdot}$ throughout this section.

Recall that $(\Phi^u_{t})_{0 \leq t \leq T}$ stands for a tuple of processes $(X^u_{t},\cG^u_t,Y^u_{t},Z^u_{t},\alpha^u_{t})_{0 \leq t \leq T}$ with values in $\RR^d \times {\mathcal P}_{2}(\RR^d)  \times \RR^d \times \RR^{d\times m} \times A$. 
Let us denote by $\mathcal{S}_1$ the space of tuples of processes 
$(\Phi^u_{t})_{0 \leq t \leq T,u\in I}$ such that for each $u\in I$,  $(X^u_{t},Y^u_{t},Z^u_{t},\alpha^u_{t})_{0 \leq t \leq T}$ is $\bar{\FF}$-progressively measurable, $(X^u_{t})_{0 \leq t \leq T}$ and $(Y^u_{t})_{0 \leq t \leq T}$ have continuous trajectories, such that
$$u\mapsto(X^u,Y^u,Z^u,\alpha^u)\in (\bar{\mathbf{S}}^2)^2\times \bar{\mathbf{H}}^2\times\bar{\mathbb{H}}^2_T(A)$$
is measurable,
and the following norm satisfies 
\begin{equation}
	\label{eq:norm S1}
	\| \Phi \|_{\mathcal{S}_1}:=
	\biggl(\int_I\EE \Big[ \sup_{0 \leq t \leq T} \bigl[ 
	\vert X^u_{t} \vert^2 + \vert Y^u_{t} \vert^2 \bigr] + \int_{0}^T \big[ \vert Z^u_{t} \vert^2 + \vert \alpha^u_{t}\vert^2 \big] dt
	\Big]du\biggr)^{1/2} < \infty. 
\end{equation}
Moreover, we use the notation 
$(\theta^u_{t})_{0 \leq t \leq T} := (X^u_{t},\cG^u_t,\alpha^u_{t})_{0 \leq t \leq T}$.

We call a family of quadruples ${\mathcal I}=(
({\mathcal I}^{b,u}_{t},
{\mathcal I}^{\sigma,u}_{t},{\mathcal I}^{f,u}_{t})_{0 \leq t \leq T},{\mathcal I}^{g,u})_{u\in I}$ an \emph{input} for \eqref{eq:fbo}, where for each $u\in I$,  $({\mathcal I}^{b,u}_t)_{0 \leq t \leq T}$, 
$({\mathcal I}^{\sigma,u}_t)_{0 \leq t \leq T}$ and $({\mathcal I}^{f,u}_t)_{0 \leq t \leq T}$ are three $\bar{\FF}$-progressively measurable processes with values in $\RR^d$, $\RR^{d \times m}$ and $\RR^d$ respectively, and
${\mathcal I}^{g,u}$ denotes a $\bar{{\mathcal F}}_{T}$-measurable random variable with values in $\RR^d$, satisfying
$$u\mapsto(
({\mathcal I}^{b,u}_{t},
{\mathcal I}^{\sigma,u}_{t},{\mathcal I}^{f,u}_{t})_{0 \leq t \leq T},{\mathcal I}^{g,u})\in \bar{\mathbb{H}}^2_T(\RR^d)\times \bar{\mathbb{H}}^2_T(\RR^{d\times m})\times \bar{\mathbb{H}}^2_T(\RR^d)\times L^2(\bar{\cF}_T;\RR^d)$$
is measurable.
For convenience, for the tuple $\cI^u$ of any label $u\in I$, we define the norm 
\begin{equation*}
	\| {\mathcal I}^u \|_{\mathbb I} :=
	\left(\EE \biggl[ \vert {\mathcal I}^{g,u} \vert^2 + \int_{0}^T \bigl[ \vert {\mathcal I}^{b,u}_{t} \vert^2 + \vert {\mathcal I}^{\sigma,u}_{t} \vert^2 + \vert {\mathcal I}^{f,u}_{t} \vert^2 \bigr] dt\biggr]\right)^{1/2}. 
\end{equation*}
We denote by ${\mathbb I}_1$ the subspace of inputs $\mathcal I$ satisfying:
\begin{equation}
	\label{eq:norm I}
    \| {\mathcal I} \|_{\mathbb I_1} :=
	\left(\int_I \rv{ \| {\mathcal I}^u \|_{\mathbb I}^2} du
	 \right)^{1/2}< \infty .
\end{equation}

Then we define the following perturbed FBSDE with given input.

\begin{definition}
	\label{def:fbi}
	For any $\gamma\in [0,1]$, any  
	$\xi=(\xi^u )_{u\in I}$ satisfying Assumption~\ref{ass:dynamic-1} and \ref{ass:dynamic-4} and any input ${\mathcal I} \in {\mathbb I}_1$, we define the following FBSDE 
	\begin{equation}
		\label{eq:fbi}
		\left\{
		\begin{aligned}
			&dX^u_{t} = \bigl( \gamma b^u(t,\theta^u_{t}) + \mathcal{I}^{b,u}_{t} \bigr) dt + \bigl( \gamma \sigma^u(t,\theta^u_{t}) + \mathcal{I}^{\sigma,u}_{t} \bigr) d\bar{W}_{t},
			\\
			&\textstyle dY^u_{t} = - \bigl( \gamma \bigl\{ \partial_{x} H^u(t,\Phi^u_{t}) +\int_I \tilde{G}(v,u)
			\t \EE \bigl[ \partial_{\mu} H^v(t,\tilde{ \Phi}^v_{t})(X^u_{t}) \bigr]dv \bigr\}  + \mathcal{I}^{f,u}_{t}\bigr) dt
			+ Z^u_{t} d \o W_{t}, \\
			&\textstyle Y^u_T= \gamma\bigl\{\partial_xg^u(X^u_T,\cG^u_T)+\int_I\tilde{G}(v,u)\t\EE[\partial_\mu g^v(\t X^v_T,\cG^v_T)(X^u_T)]dv\bigr\}+\mathcal{I}^{g,u}, \; X^u_0= \xi^u,\; u\in I ,
		\end{aligned}
		\right.
	\end{equation}
	where, for chosen predictable representatives and \(du\)-a.e. $u\in I$,
		$\alpha^u_{t} = \hat{\alpha}^u(t,X^u_{t},\cG^u_t,Y^u_{t},Z^u_{t})$,
        \(dt\otimes d\bar\PP\)-a.e. 
	   This FBSDE will be denoted by ${\mathcal E}(\gamma,\xi,{\mathcal I})$.
	If $(X^u_{t},Y^u_{t},Z^u_{t})_{0 \leq t \leq T,u\in I}$ is a solution of ${\mathcal E}(\gamma,\xi,{\mathcal I})$, we call $(X^u_{t},\cG^u_t,Y^u_{t},Z^u_{t},\alpha^u_{t})_{0 \leq t \leq T,u\in I}$ the \emph{associated extended solution}.
\end{definition}

\begin{lemma}
	\label{lem:contra}
	Let $\gamma \in [0,1]$ and suppose ${\mathcal E}(\gamma,\xi,{\mathcal I})$ has a unique extended solution for any $\xi=(\xi^u)_{u\in I}$ satisfying Assumptions~\ref{ass:dynamic-1} and \ref{ass:dynamic-4} and any $\cI\in \mathbb{I}_1$. Then under Assumptions~\ref{ass:d--1}-\ref{ass:d--3}, there exists a constant $C$, independent of 
	$\gamma$, such that for any $\xi,\xi'$ satisfying Assumption~\ref{ass:dynamic-1} and \ref{ass:dynamic-4}, and any ${\mathcal I},{\mathcal I}' \in {\mathbb I}_1$, the corresponding extended solutions $\Phi$ and $\Phi'$ of ${\mathcal E}(\gamma,\xi,{\mathcal I})$ and ${\mathcal E}(\gamma,\xi',{\mathcal I}')$ satisfy
	\begin{equation*}
		\| \Phi - \Phi' \|_{\mathcal{S}_1} \leq C \Big( \left[ \int_I {\mathbb E}\vert \xi^u - \xi'^{,u} \vert^2du \right]^{1/2}
		+ \| {\mathcal I} - {\mathcal I}' \|_{\mathbb I_1} \Big).  
	\end{equation*}
\end{lemma}

The proof, given in Appendix~\ref{app:details-fbsde}, follows the standard
mean field FBSDE estimates of \cite{lauriere2022convergence}, with the graphon
mean field terms treated by the convexity assumptions.

\begin{theorem}
	\label{th:exo}
	Under Assumptions~\ref{ass:d--1}-\ref{ass:d--3}, the \rv{coupled canonical version of}
    \eqref{eq:fbo} is uniquely solvable in $\mathcal{S}_1$. 
\end{theorem}

\begin{proof}
The proof follows from Picard's contraction theorem by using Lemma~\ref{lem:contra}. 
When $\gamma=0$, the arguments in Lemma~\ref{lem:barx} and
Theorem~\ref{thm:exfb} give a unique solution
$$(X^{0,u},Y^{0,u},Z^{0,u})_{u\in I}\in
(L^2(I;\bar{\mathbf{S}}^2))^2\times L^2(I;\bar{\mathbf{H}}^2)$$
for every $\xi$ satisfying Assumptions~\ref{ass:dynamic-1} and
\ref{ass:dynamic-4}. It remains to verify that
$(\hat{\alpha}^{0,u})_{u\in I}\in L^2(I;\bar{\mathbb{H}}^2_T(A))$. We have
\[
 \rv{
 \partial_\alpha H^u(t,x,\mu,y,z,\alpha)
 =b_3(t)^\top y+\sigma_3(t)^\top z
   +\partial_\alpha f(t,x,\mu,\alpha).}
\]
By Assumption~\ref{ass:d--1}, \rv{$(b^u_3,\sigma^u_3)$ are the same for all $u\in I$, and $\partial_\alpha f, u\in I$ are also the same due to Assumption~\ref{ass:d--2}. Hence  $\partial_{\alpha}H^u, u\in I$ have the same form independent of $u$. Hence $u\mapsto \h\alpha^u(t,x,\mu,y,z)$ is a constant map for each $(t,x,\mu,y,z)$.} \rv{By similar arguments used in the measurability part in the proof of Lemma~\ref{lem:barx}, combining with the Lipschitz property of $\h\alpha^u(t,\cdot,\cdot,\cdot,\cdot)$, we obtain for any $t\in[0,T]$, $u\mapsto (\h\alpha^u(s, X^{0,u}_s,\cG^{0,u}_s,Y^{0,u}_s,Z^{0,u}_{s}))_{s\in[0,T]} \in\bar{\mathbb{H}}^2_T(A)$ is measurable. Together with the integrability of $(X^{0,u},Y^{0,u},Z^{0,u})$ over label, we conclude $(\hat{\alpha}^u)_{u\in I}\in L^2(I;\bar{\mathbb{H}}^2_T(A))$.} Hence $\mathcal{E}(0,\xi,\cI)$ admits a unique solution in $\mathcal{S}_1$.

Next, we proceed to the induction step. 
Suppose that for some $\gamma\in[0,1]$, for any $\xi$ satisfying Assumption~\ref{ass:dynamic-1} and \ref{ass:dynamic-4} and any ${\mathcal I}\in {\mathbb I}_1$, FBSDE $\mathcal{E}(\gamma,\xi,\cI)$ admits a unique solution. For $\eta >0$, we define a mapping $\Psi$ from ${\mathcal S_1}$ into itself as follows: Given $\Phi \in {\mathcal S_1}$, we denote by $\Theta'$ the extended solution
	of the FBSDE ${\mathcal E}(\gamma,\xi,{\mathcal I}^{\prime})$ with 
	\begin{equation*}
		\begin{split}
			&\cI'^{,b,u}_{t} = \eta b^u(t,\theta^u_{t}) + \cI^{b,u}_{t},
			\\
			&\cI'^{,\sigma,u}_{t} = \eta \sigma^u(t,\theta^u_{t}) + \cI^{\sigma,u}_{t},
			\\
			&\cI'^{,f,u}_{t} = \eta \partial_{x} H^u(t,\Phi^u_{t}) + \eta \int_I \tilde{G}(v,u)\t \EE \bigl[ \partial_{\mu} H^v(t,\tilde{ \Phi}^v_{t})(X^u_{t}) \bigr]dv + \cI^{f,u}_{t},
			\\
			&\cI'^{,g,u}_{T} = \eta \partial_{x} g^u(X^u_{T},\cG^u_T) + \eta \int_I \tilde{G}(v,u)\t \EE \bigl[ 
			\partial_{\mu} g^v(\t X^v_{T},\cG^v_{T})(X^u_{T}) \bigr]dv
			+ \cI^{g,u}_{T}. 
		\end{split}
	\end{equation*}
	It follows that $\cI'$ is in $\mathbb{I}_1$, thus the extended solution is uniquely defined and by assumption it belongs to ${\mathcal S_1}$. Observe that a process $\Phi \in {\mathcal S_1}$ is a fixed point of $\Psi$ if and only if $\Phi$ is an extended solution of ${\mathcal E}(\gamma + \eta,\xi,{\mathcal I})$. Then by using Lemma~\ref{lem:contra} and choosing $\eta$ small enough, we obtain that for any $\xi$ satisfying Assumption~\ref{ass:dynamic-1} and \ref{ass:dynamic-4} and any ${\mathcal I}\in {\mathbb I}_1$, FBSDE $\mathcal{E}(\gamma+\eta,\xi,\cI)$ admits a unique solution. Finally we conclude by induction.
\end{proof}

\rv{Finally, after obtaining the mean field profile from the canonical FBSDE at optimum, we can plug it into the original independent noise system \eqref{eq:fbo}. It is clear that when the mean field profile is fixed, the interactions between labels are canceled out and for each label $u\in I$, its solution corresponds to a single non-MKV FBSDE. Classical results guarantee the existence and uniqueness of the pathwise solution for each $u$. Then note that the assumptions in Lemmas~\ref{lem:frozen-causal-realization} and
\ref{lem:frozen-law-transfer} are verified for the law kernel obtained from the canonical solution in
Theorem~\ref{th:exo}. Hence by applying Lemmas~\ref{lem:frozen-causal-realization} and
\ref{lem:frozen-law-transfer}, we can conclude there exists a unique solution to system~\eqref{eq:fbo}.}

\subsection{Continuity and stability}\label{sec:sc}
\medskip

In this section, we prove the continuity and the stability of the FBSDE system under optimal control \eqref{eq:fbo}. The continuity result consists in comparing the differences of the solutions associated to different labels within a same system. We will use the following assumption to simplify the presentation.

\begin{assumptionp}{\ref*{ass:d}}

\begin{enumerate}[label=
\normalfont{\textbf{(\ref*{ass:d}\arabic{*}).}}, ref=\normalfont{\textbf{(\ref*
{ass:d}\arabic{*})}}]
\setcounter{enumi}{\value{counterExtraAssumptions}}

\item \label{ass:d--5}Suppose $(b^u,\sigma^u,f^u,g^u)=(b,\sigma,f,g)$ for all $u\in I$.
\end{enumerate}
\end{assumptionp}

We first give an essential supremum bound for the solution of \eqref{eq:fbo}. 

\begin{lemma}\label{lem:esup_u}
 Under Assumptions~\ref{ass:d--1}--\ref{ass:d--5}, the unique solution $(X^u,Y^u,Z^u)_{u\in I}$ to \eqref{eq:fbo} satisfies 
 \begin{equation}\label{eq:esup_u}
 \esssup_{u\in I} \EE\bigg[\sup_{t\in[0,T]}|X^u_t|^2+\sup_{t\in[0,T]}|Y^u_t|^2+\int_0^T |Z^u_t|^2dt\bigg]<\infty.    
 \end{equation}
\end{lemma}

\begin{proof}
 By the definition of the solution to \eqref{eq:fbo}, for $du$-a.e. $u\in I$, $(X^u,Y^u,Z^u)$ solves the FBSDE \eqref{eq:fbo}. In addition, Assumption~\ref{ass:d--5} guaranties that the Lipschitz constants for the coefficients are uniform for all $u\in I$. Hence, by the classical estimate for FBSDE (see \cite[Lemma 4.3]{lauriere2022convergence}), taking $K_b$ in Assumption~\ref{ass:d--1} large enough, there exists some small $\epsilon(K_b)$ such that for a.e. $u\in I$,
\begin{align}\label{eq:x_bound}
    \EE\bigg[\sup_{t\in[0,T]}  |X^u_t|^2\bigg]\leq & C\epsilon(K_b)\EE[1+|\xi^u|^2]+\epsilon(K_b)\EE\bigg[\int_0^T|Y^u_t|^2dt+\int_0^T |Z^u_t|^2dt\bigg]\\
    & \quad +\epsilon(K_b)\int_0^T\int_I\EE[|X^v_t|^2]dvdt.
    \end{align}
For the backward part, following the classical estimate, using Young's inequality and choosing $\beta$ large enough, we have for any $\epsilon_1>0$, $t\in[0,T]$ and a.e. $u\in I$,
\begin{align}
  &  \EE\bigg[e^{\beta t}|Y^u_t|^2+\int_t^T e^{\beta s}|Z^u_s|^2ds\bigg] \\
    & \quad \leq C(\epsilon_1)\EE\bigg[e^{\beta T}|Y^u_T|^2+\int_t^Te^{\beta s}|X^u_s|^2 ds\bigg]+\epsilon_1C\int_t^Te^{\beta s}\int_I\EE[|X^v_s|^2+|Y^v_s|^2+|Z^v_s|^2]dvds.
\end{align}
Now combining the above two results, taking the $\esssup$ over $u\in I$ and letting $\epsilon(K_b),\epsilon_1$ small enough, we have
\begin{align}
    \esssup_{u\in I}\sup_{t\in[0,T]}\EE\bigg[e^{\beta t}|Y^u_t|^2+\int_t^T e^{\beta s}|Z^u_s|^2ds\bigg]\leq C(K_b)e^{\beta T}\big(1+\esssup_{u\in I}\EE[|\xi^u|^2]\big)< \infty.
\end{align}
Plugging the above back to \eqref{eq:x_bound}, we get
\begin{align}
   \esssup_{u\in I} \EE\bigg[\sup_{t\in[0,T]}  |X^u_t|^2\bigg]\leq C(K_b)\bigg(1+\esssup_{u\in I}\EE[|\xi^u|^2]+\esssup_{u\in I}\sup_{t\in[0,T]}\EE\bigg[|Y^u_t|^2+\int_0^T |Z^u_s|^2ds\bigg]\bigg)<\infty,
\end{align}
Therefore, we obtain
\begin{equation}
 \esssup_{u\in I}\bigg\{ \EE\bigg[\sup_{t\in[0,T]}|X^u_t|^2\bigg]+\sup_{t\in[0,T]}\EE[|Y^u_t|^2]+\EE\bigg[\int_0^T |Z^u_t|^2dt\bigg]\bigg\}<\infty.    
 \end{equation}
Then by It\^o's formula and standard estimate for BSDE (see \cite[Theorem 4.2.1]{zhang2017backward}), we can pass the $\sup_{t\in[0,T]}$ into the expectation for $Y$, and therefore obtain \eqref{eq:esup_u}.
\end{proof}

We call a graphon $G$ \emph{piecewise continuous} if there exists a partition of $I$ into $k$ intervals $\{I_i, i=1,\ldots ,k\}$, for some $k\in \mathbb{N}$, such that $G(u,v)$ is piecewise continuous with respect to $u$ and $v$ in all intervals $I_i, i=1,\ldots,k$. Furthermore, we call it \emph{piecewise Lipschitz continuous} if for all $u_1,u_2\in I_i$, $v_1,v_2 \in I_j$, and $i,j\in\{1,\ldots,k\}$, there exists a constant $C$ such that
$|G(u_1,v_1)-G(u_2,v_2)|\leq C(|u_1-u_2|+|v_1-v_2|)$.

\begin{theorem}[Continuity]\label{thm:conti}
Recall that in the canonical coupled system \eqref{eq:fbww} (with initial condition coupled on $\o\cF_0$ and all label driven by the canonical Brownian motion), let $(\o X^u,\o Y^u, \o Z^u)$ denote the solution of the canonical coupled system and $\o \xi$ denote the coupled initial condition. Under Assumptions~\ref{ass:d--1}--\ref{ass:d--5}, for any two different labels $u_1$ and $u_2$ outside a Lebesgue null set on $I$, we have
	\begin{align}
    \begin{aligned}
			& \EE \Big[ \sup_{0 \leq t \leq T} \vert \o X^{u_1}_{t} - \o X^{u_2}_{t} \vert^2 + \sup_{0 \leq t \leq T} \vert \o Y^{u_1}_{t} - \o Y_{t}^{u_2} \vert^2 + \int_{0}^T \vert \o Z^{u_1}_{t} - \o Z_{t}^{u_2} \vert^2 dt\Big]
			\\
			&\hspace{15pt} \leq  C\bigl( \EE \bigl[ \vert \o\xi^{u_1} - \o\xi^{u_2} \vert^2 \bigr] +\|G(u_1,\cdot)-G(u_2,\cdot)\|_1 \bigr). 
	\end{aligned}	\label{eq:graphonconti}
	\end{align}
\end{theorem}

\begin{proof}
	We first bound $\cW_2(\cG^{u_1}_t,\cG^{u_2}_t)$. 
	By the definition of Wasserstein $L^2$ distance and triangle inequality, we have for some constant $C$, 
	$$\cW_2^2(\cG^{u_1}_t,\cG^{u_2}_t)\leq \int_I|\tilde{G}(u_1,v)-\tilde{G}(u_2,v)|\EE[|X^v_t|^2]\rv{dv}\leq C\int_I|G(u_1,v)-G(u_2,v)| \rv{dv},$$
	where the last inequality comes from the result of Lemma~\ref{lem:esup_u}.
    
	
		Since now under Assumption~\ref{ass:d--5}, the optimal control $\h\alpha^u$ is identical for all $u\in I$, which is Lipschitz continuous with respect to all parameters except $t$, we can plug it into the FBSDE and get rid of the control $\alpha$, obtaining new coefficients which are also Lipschitz. The estimate is similar to that in the proof of Lemma~\ref{lem:esup_u}. For the forward part, there exists some small constant $C(K_b)$ depending on $K_b$, which can be small enough with $K_b$ being chosen large enough. Hence we have
	\begin{align}
		\begin{aligned}
			\label{eq:Fes1}
			\EE \bigl[ \sup_{0 \leq t \leq T} \vert \o X^{u_1}_{t} - \o X^{u_2}_{t} \vert^2  \bigr]
			& \leq C_1C(K_b)\EE \bigl[ \vert \o\xi^{u_1} - \o\xi^{u_2} \vert^2 \bigr] + C_1C(K_b) 
			\|G(u_1,\cdot)-G(u_2,\cdot)\|_1 \\
			& \quad \quad +C(K_b)\EE\bigl[\int_{0}^T \vert \o Y^{u_1}_{t} - \o Y_{t}^{u_2} \vert^2+\vert \o Z^{u_1}_{t} -\o Z_{t}^{u_2} \vert^2 dt\bigr].
		\end{aligned} 
	\end{align}
	Next we proceed with the estimate of the backward part. Similarly by classical estimate, we have 
	\begin{equation}
		\label{eq:Bes1}
		\begin{split}
			\EE \biggl[ \sup_{0 \leq t \leq T} &\vert \o Y^{u_1}_{t} - \o Y^{u_2}_{t} \vert^2 + \int_{0}^T \vert \o Z^{u_1}_{t} - \o Z^{u_2}_{t} \vert^2 dt \biggr]
			\\
			&\hspace{25pt} \leq  C_2\EE \bigl[
			\sup_{0 \leq t \leq T} \vert \o X^{u_1}_{t} - \o X^{u_2}_{t} \vert^2
			\bigr]+ 
		C_2 \|G(u_1,\cdot)-G(u_2,\cdot)\|_1.
		\end{split}
	\end{equation}
Then by taking $C(K_b)$ small enough, such that $C_2C(K_b)<1$, we have for another constant $C_3(K_b)$,
	\begin{align}
		\begin{aligned}
			\label{eq:Fes3}
			\EE \bigl[ \sup_{0 \leq t \leq T} \vert \o X^{u_1}_{t} - \o X^{u_2}_{t} \vert^2  \bigr]
			& \leq C_3(K_b)\EE \bigl[ \vert \o\xi^{u_1} - \o\xi^{u_2} \vert^2 \bigr] + C_3(K_b)
			\|G(u_1,\cdot)-G(u_2,\cdot)\|_1. 
		\end{aligned} 
	\end{align}
Plugging \eqref{eq:Fes3} into \eqref{eq:Bes1} and adding them together, we finally obtain \eqref{eq:graphonconti}.
	
\end{proof}

The stability result lies in comparing the global solutions over all label associated to two systems induced by two different graphons.  

\begin{theorem}[Stability]\label{thm:sta}
	Let $G$ and $G'$ be two graphons  satisfying Assumptions~\ref{ass:graphon}. Let $X,Y,Z$ and $X',Y',Z'$ be the solutions of \eqref{eq:fbo} associated to graphons $G$ and $G'$, and initial conditions $\xi$ and $\xi'$ respectively. Under Assumptions~\ref{ass:d--1}--\ref{ass:d--5}, we have for some constant $C$:
	\begin{equation}\label{eq:graphonsta}
		\begin{split}
			&\int_I \EE \Big[ \sup_{0 \leq t \leq T} \vert X^u_{t} - X_{t}'^{,u} \vert^2 + \sup_{0 \leq t \leq T} \vert Y^u_{t} - Y_{t}'^{,u} \vert^2 + \int_{0}^T \vert Z^u_{t} - Z_{t}'^{,u} \vert^2 dt\Big] du
			\\
			&\hspace{15pt} \leq  C\Big(\int_I\EE \bigl[ \vert \xi^u - \xi'^{,u} \vert^2 \bigr]du +\|G-G'\|_1 \Big). 
		\end{split}
	\end{equation}
\end{theorem}

\begin{proof}
Let $\alpha$ and $\alpha'$ be the optimal controls respectively associated to $G$ and $G'$.

	\noindent{\bf Step 1:} We estimate the difference between $\alpha$ and $\alpha'$. 
	Recall the definition of $\Theta$. Further, we define the notation $\Theta'^{,u,G}:=\check{X}'^{,\vartheta^u}$ and $\Theta'^{,u}:=\check{X}'^{,\vartheta'^{,u}}$, where $\cL(X^u)=\cL(\check{X}^u),\cL(X'^{,u})=\cL(\check{X}'^{,u})$ for all $u\in I$, $\vartheta^u$ and $\vartheta'^{,u}$ are the random variables defined with graphons $G$ and $G'$ respectively (see details in Section~\ref{sec:3.1}). Here we choose a optimal coupling between $\Theta'^{,u,G}$ and $\Theta'^{,u}$ such that the $L^2$ difference endowed by the running supremum metric over $[0,T]$ between them attains the minimum. Then notice that by Assumptions~\ref{ass:d--1}--\ref{ass:d--2}, we have
	\begin{align}
		&\int_I\EE \biggl[ \biggl( \int_I\tilde{G}(v,u) \t\EE[\partial_\mu g(\t X^v_T, \cG^v_T)(X^u_{T})]dv\biggr)  \cdot (X_T'^{,u}- X^u_T) \biggr]du
		\\
		&\leq \int_I \EE\bigl[\t\EE[|\partial_\mu g(X^u_T, \cG^u_T)(\t\Theta^u_{T})-\partial_\mu g(0, \delta_0)(\t 0)|^2]^{\frac{1}{2}}\t\EE[|\t\Theta_T'^{,u,G}-\t\Theta_T'^{,u}|^2]^{\frac{1}{2}} \bigr]du\\
		& \quad + \int_I\EE   \t\EE\bigl[\partial_\mu g(X^u_T, \cG^u_T)(\t\Theta^u_{T}) \cdot (\t\Theta_T'^{,u,G}- \t\Theta^u_T) \bigr]du
		 \\
		 &\quad +\int_I\EE   \t\EE\bigl[\partial_\mu g(0, \delta_0)( 0) \cdot (\t\Theta_T'^{,u,G}- \t\Theta'^{,u}_T) \bigr]du\\
		&\leq C\int_I \EE[|X^u_T|^2+|\Theta^u_T|^2]^{\frac{1}{2}}\cdot\t\EE[|\t\Theta_T'^{,u,G}-\t\Theta_T'^{,u}|^2]^{\frac{1}{2}}du\\
		&\quad\quad+C \int_I \t\EE\bigl[|\t\Theta_T'^{,u,G}-\t\Theta_T'^{,u}|\bigr]du+\int_I\EE   \t\EE\bigl[\partial_\mu g(X^u_T, \cG^u_T)(\t\Theta^u_{T}) \cdot (\t\Theta_T'^{,u}- \t\Theta^u_T) \bigr]du.
		\label{eq:123}
	\end{align}

	By adding and subtracting terms, and using the result in \cite[Theorem 6.15]{villani2009optimal}, the optimal coupling between $\Theta'^{,u,G}$ and $\Theta'^{,u}$ gives
	\begin{align}
		\t\EE|\t\Theta_T'^{,u,G}-\t\Theta_T'^{,u}|^2&\leq\int_I\int_{\RR^d}\bigg|\frac{G(u,v)}{\|G(u,\cdot)\|_1}-\frac{G'(u,v)}{\|G'(u,\cdot)\|_1}\bigg||x|^2\mu'^{,v}_T(dx) dv
		\\
		&\quad +\bigl|\|G'(u,\cdot)\|_1-\|G(u,\cdot)\|_1\bigr|\|G(u,\cdot)\|_1^{-1}\int_I \frac{G'(u,v)}{\|G'(u,\cdot)\|_1}\int_{\RR^d}|x|^2\mu'^{,v}_T(dx)dv\\
		&\leq C \int_I\bigl|G(u,v)-G'(u,v)\bigr|dv,
		\label{eq:11}
	\end{align}
	where $\mu'^{,v}_T$ is the probability distribution of $X'^{,v}_T$, and the last inequality comes from the result of Lemma~\ref{lem:esup_u}.
    
	Then using the above two results in \eqref{eq:123} and \eqref{eq:11}, we obtain 
	\begin{align*}
		&\int_I\EE \Big[ \Big( \int_I\tilde{G}(v,u) \t\EE[\partial_\mu g(\t X^v_T, \cG^v_T)(X^u_{T})]dv\Big)  \cdot (X_T'^{,u}- X^u_T) \Big]du\\
		&\leq C \|G-G'\|_1+\int_I\EE   \t\EE\bigl[\partial_\mu g(X^u_T, \cG^u_T)(\t\Theta^u_{T}) \cdot (\t\Theta_T'^{,u}- \t\Theta^u_T) \bigr]du.
	\end{align*}
	Then it follows from the convexity assumption (Assumption~\ref{ass:d--3}) that	
	\begin{align*}
			&\int_I\EE \bigl[(X_T'^{,u}-X^u_T)\cdot Y^u_T \bigr]du
			\\
			&= \int_I\EE \bigl[ \partial_xg(X^u_T,\cG^u_T)\cdot(X'^{,u}_T-X^u_T)\bigr]du+\int_I \EE\t\EE[\partial_\mu g^u( X^u_T, \cG^u_T)(\t\Theta^u_{T})\cdot (\t\Theta_T'^{,u,G}- \t\Theta^u_T) \bigr]du
			\\
			&\leq \int_I \EE\bigl[g(X_T'^{,u},\cG'^{,u}_T)-g(X^u_T,\cG^u_T)\bigr]du  + C\|G-G'\|_1.
	\end{align*}
	We proceed similarly for $f$. As in the proof of Lemma~\ref{lem:contra}, we obtain %
	\begin{equation*}
		\begin{split}
			J(\alpha') - J(\alpha) \geq &
			\lambda\int_I \EE \Big[\int_{0}^T \vert \alpha^u_{t} - \alpha_{t}'^{,u} \vert^2 dt\Big]du + 
			\int_I \EE[(\xi'^{,u}-\xi^u)\cdot Y^u_0]du 
			-C\|G-G'\|_1. 
		\end{split}
	\end{equation*}
	Hence, following similar arguments and by Young's inequality, we have that for a constant $C$ and for any $\varepsilon >0$,
	\begin{equation}
		\label{eq:2:2:3}
		\int_I\EE \int_{0}^T \vert \alpha^u_{t} - \alpha_{t}'^{,u} \vert^2 dtdu
		\leq 
		\varepsilon\int_I 
		\EE| Y^u_0 - Y'^{,u}_0 |^2du 
		+ \frac{C}{\varepsilon} \biggl(\int_I \EE \bigl[ \vert \xi^u - \xi'^{,u} \vert^2 \bigr] du
		+\|G-G'\|_1\biggr).  
	\end{equation}
	
	\noindent{\bf Step 2:} We estimate the difference between FBSDE systems. 
	First let us estimate the value of $\cW_2(\cG^u,\cG'^{,u})$. To clarify the dependency of the population distribution $\mu$ in the definition of $\cG$, we write it as a parameter explicitly as in~\eqref{Lab}. 
	By the triangle inequality, \cite[Lemma 3.1]{phamnonlinear} (see Lemma~\ref{lem:cg} in appendix) and \eqref{eq:11}, we have that
	\begin{align*}
		\cW^2_2([\cG \mu_t]^u,[\cG'\mu'_t]^u)&\leq \cW^2_2([\cG \mu_t]^u,[\cG' \mu_t]^u)+\cW^2_2([\cG' \mu_t]^u,[\cG' \mu_t']^u)\\
		& \leq C\|G(u,\cdot)-G'(u,\cdot)\|_1+C\int_I\cW^2_2(\mu^v_t,\mu'^{,v}_t) dv.
	\end{align*}
	Now using standard estimate methods (see e.g. \cite{phamnonlinear}) for the forward part and the above inequality, we obtain that for some constant $C_1$,
	\begin{align}
			&\int_I\EE \bigl[ \sup_{0 \leq t \leq T} \vert X^u_{t} - X_{t}'^{,u} \vert^2  \bigr]du
			\\
			& \leq \int_I\EE \bigl[ \vert \xi^u - \xi'^{,u} \vert^2 \bigr]du + C_1 
			\|G-G'\|_1 + C_1 \int_I \EE \bigl[\int_{0}^T \vert \alpha^u_{t} - \alpha_{t}'^{,u} \vert^2 dt\bigr]du.
			\label{eq:Fes2}
	\end{align}
	
	Next we estimate the backward part. Using similar estimate methods for graphon BSDEs as in \cite{amicaosu2022graphonbsde}, we have that
	\begin{align}
			&\int_I\EE \biggl[ \sup_{0 \leq t \leq T} \vert Y^u_{t} - Y'^{,u}_{t} \vert^2 + \int_{0}^T \vert Z^u_{t} - Z'^{,u}_{t} \vert^2 dt \biggr]du
			\\
			&\leq C_2 \int_I \EE \biggl[
			\sup_{0 \leq t \leq T} \vert X^u_{t} - X'^{,u}_{t} \vert^2
			+  \int_{0}^T   \vert \alpha^u_{t} - \alpha'^{,u}_{t} \vert^2  dt \biggr]du+ 
			C_2 \|G-G'\|_1.
		\label{eq:Bes2}
	\end{align}
	
	Finally combining \eqref{eq:2:2:3}, \eqref{eq:Fes2} and \eqref{eq:Bes2} leads to 
	\begin{equation}
		\begin{split}
			&\int_I \EE \bigl[ \sup_{0 \leq t \leq T} \vert X^u_{t} - X_{t}'^{,u} \vert^2 + \sup_{0 \leq t \leq T} \vert Y^u_{t} - Y_{t}'^{,u} \vert^2 + \int_{0}^T \vert Z^u_{t} - Z_{t}'^{,u} \vert^2 dt\bigr] du
			\\
			&\hspace{15pt} \leq  C_3 \varepsilon\int_I 
			\EE| Y^u_0 - Y'^{,u}_0 |^2du
			+ C_3\bigl( \int_I\EE \bigl[ \vert \xi^u - \xi'^{,u} \vert^2 \bigr]du +\|G-G'\|_1 \bigr) \\
			&\hspace{15pt} \leq  C_3\bigl(\int_I\EE \bigl[ \vert \xi^u - \xi'^{,u} \vert^2 \bigr]du +\|G-G'\|_1 \bigr), 
		\end{split}
	\end{equation}
	for some constant $C_3$ where the last inequality follows from choosing $\varepsilon$ small enough.
\end{proof}

\section{Propagation of Chaos and Approximate Optimality}
\label{sec:chaos}

In this section, we establish a connection between systems with $N$ agents and the corresponding limit graphon mean field systems. We use the solution of the optimal control of the graphon mean field dynamics to obtain an approximately optimal control for the problem with $N$ agents when $N$ tends to $+\infty$. Throughout this section, we suppose Assumptions~\ref{asm:Afull-linear} and~\ref{ass:d--1}--\ref{ass:d--5} hold.

\subsection{\texorpdfstring{System of $N$ agents}{System of N agents}}

\rv{Let $\xi^{N,i}\in L^2(\cF_0)$, $i\in[N]$, be mutually independent.  Let
$(W^i)_{i\geq1}$ be mutually independent Brownian motions, all independent of
the whole sigma field $\cF_0$.  Thus the input pairs
$(\xi^{N,i},W^i)$, $i\in[N]$, are mutually independent, and every $W^i$
remains Brownian after the enlargement below.  Let $\FF^N$
be the completed filtration generated by $\cF_0$ and $(W^1,\ldots,W^N)$,
and set
\[
 \cF_t^{N,i}=\sigma(\xi^{N,i},W_s^i:s\leq t)\vee\mathcal N_{\PP}
 \subseteq\cF_t^N.
\]}
Let $\mathbb{H}^2_{N}(A)$ denote the set of $A$-valued $\mathbb{F}^{N}$ progressively measurable processes $\phi$ satisfying 
$$
    \|\phi\|_{\mathbb{H}}:=\Bigg(\EE\bigg[\int_0^T |\phi_t|^2dt\bigg]\Bigg)^{1/2}<\infty.
$$
 A strategy profile $\underline{\beta}=(\beta^1,\ldots,\beta^N)$ is called admissible if for all $i\in [N]$, $\beta^i\in\bH^2_{N}(A)$. Let $\cM\bH^2_N$ denote the set of admissible control profiles for the system with $N$ agents.

We consider the interaction matrix $\zeta^N:=(\zeta^N_{ij})_{i,j}\in \RR^{N\times N}$ with all entries nonnegative and uniformly bounded in $N,i,j$, describing interaction strength between all pairs $(i,j)$. We let
\begin{equation}
    \label{eq:def-kappaN}
    \kappa^{N,ij}:=\frac{\zeta^N_{ij}}{\sum_{j=1}^N \zeta^N_{ij}}.
\end{equation}
The dynamics of the system with $N$ agents is given by the following $N$ coupled SDEs:
\begin{equation}
	\label{fo:stateN}
	dX_{t}^i = b \bigl(t,X_{t}^i,\bar{\nu}_t^{N,i},\beta_t^i \bigr) dt + \sigma(t,X_{t}^i,\bar{\nu}_{t}^{N,i},\beta_{t}^i) dW_{t}^i, \quad X^i_0=\xi^{N,i}, \quad1 \leq i \leq N,
\end{equation}
where
$
	\bar{\nu}^{N,i}_{t} = \sum_{j=1}^N \kappa^{N,ij} \delta_{X_{t}^j}
$
is called the rescaled neighborhood empirical measure. 
We construct a step graphon $G_N$ induced by $\zeta^N$, defined as
$$
    G_N(u,v)=\zeta^N_{ij}, \quad\quad (u,v)\in \cI^N_i\times \cI^N_j,
$$
where $\cI^N_1 \coloneqq [0, \tfrac{1}{N}]$ and $\cI^N_i \coloneqq (\tfrac{i-1}{N},\tfrac{i}{N}]$ for $i>1$. \rv{We use the convention $i_N(u)=\lceil Nu\rceil$ for $u\in(0,1]$ and $i_N(0)=1$.}
The value we take for $G_N(0,0)$ does not matter for the results presented in the sequel but for the sake of definiteness, we let it be $\zeta^N_{11}$. 

For each $1 \leq i \leq N$, we denote the cost of agent $i$ by
\begin{equation}
	\label{fo:costN}
	J^{N,i}(\beta^1,\dots,\beta^N) =  {\mathbb E}
	\biggl[ \int_{0}^T f(t,X_{t}^i,\bar{\nu}_{t}^{N,i},\beta_{t}^i) dt + g\bigl(X_{T}^i,\bar{\nu}_{T}^{N,i}
	\bigr) \biggr]
\end{equation}
Notice that the roles of different $i$ are not exchangeable, and thus $J^{N,i}$ is different for different $i$. We aim to minimize the total cost over all admissible strategy profiles $\underline{\beta}=(\beta^1,\ldots,\beta^N) \in \cM\bH^2_N$: 
$$
    J^{N}(\underline{\beta}) 
    = \frac{1}{N}\sum_{i=1}^N J^{N,i}(\beta^1,\dots,\beta^N)
    = \frac{1}{N}\sum_{i=1}^N {\mathbb E}
\biggl[ \int_{0}^T f(t,X_{t}^i,\bar{\nu}_{t}^{N,i},\beta_{t}^i) dt + g\bigl(X_{T}^i,\bar{\nu}_{T}^{N,i}
\bigr) \biggr].
$$

\subsection{Limit theory and non-Markovian approximate optima}
We denote by $J^*$ the optimal GMFC cost:
\begin{equation}
	\label{eq:optimalcost}
	J^* = \int_I\EE \biggl[ \int_{0}^T f\bigl(t,X^u_{t},\cG^u_{t},\hat{\alpha}(t,X^u_{t},\cG^u_{t},Y^u_{t},Z^u_{t})\bigr) dt + g(X^u_{T},\cG^u_{T}) \biggr]du,
\end{equation} 
where $(X^u_{t},Y^u_{t},Z^u_{t})_{0 \leq t \leq T}$ is the solution to the coupled graphon mean field FBSDE \eqref{eq:fbo} with $\h\alpha$ is the minimizer of the Hamiltonian. Let $(\mu^u_{t})_{0 \leq t \leq T}$ denote the flow of probability measures $\mu^u_t=\PP_{X^u_{t}}$, for $u\in I, 0 \leq t \leq T$. 

In addition, we will sometimes make use of some of the following assumptions.

\begin{assumption}[Interaction regularity]\label{ass:e}
The sequence $(G_N)_{N\geq 1}$ and the graphon $G$ satisfy:
\begin{enumerate}[label=\normalfont{\textbf{(\ref*{ass:e}\arabic{*})}}, ref=\normalfont{\textbf{(\ref*
{ass:e}\arabic{*})}}, topsep=2pt,itemsep=0pt,partopsep=0pt, parsep=0pt]

		\item \label{ass:e--1} For any $N\in \NN$, $\min_{i\in[N]}\sum_{j=1}^N \zeta^N_{ij}>CN$, uniformly for some positive constant $C>0$. 
		\item \label{ass:e--2} $\|G-G_N\|_1\to 0$ as $N\to\infty$.
		\item \label{ass:e--3} There exist some constant $K$ and a partition of $I$ into $K$ intervals $I_k,k \in [K]$, such that the initial distribution of states is continuous on each interval of the partition, i.e. for every $k\in[K]$, $I_k\ni u\mapsto\mu^u_0\in \cP_2(\RR^d) $ is continuous w.r.t. $\cW_2$ metric, and the graphon $G$ is continuous on each block $I_i\times I_j$, for $i,j\in [K]$.
		\item \label{ass:e--4} The following conditions hold:
    \begin{enumerate}[topsep=2pt,itemsep=0pt,partopsep=0pt, parsep=0pt]
            \item The sequence $(G_N)_{N\geq 1}$ satisfies $\|G-G_N\|_1 \leq C/N$, for some constant $C$.
            \item There exist constants $L_1$ and $L_2$ such that
            \begin{enumerate}[topsep=2pt,itemsep=0pt,partopsep=0pt, parsep=0pt]
		      \item For each $i\in [K]$ and any $u,v\in I_i$, $\cW_2(\mu^u_0,\mu^v_0)\leq L_1|u-v|$.
		      \item  For each pair $(i,j)\in [K]^2$ and any $(u_1,v_1),(u_2,v_2)\in I_i\times I_j$,
		$|G(u_1,v_1)-G(u_2,v_2)|\leq L_2(|u_1-u_2|+|v_1-v_2|)$.
        \end{enumerate}
        \end{enumerate}
	\end{enumerate}
\end{assumption}

Note that in most of the existing literature on graphon interacting systems without controls (see e.g. \cite{erhangmfs,wu2022,amicaosu2022graphonbsde}), convergence results are obtained under a weaker convergence condition, in the sense of cut norm. But in graphon systems \emph{with controls}, convergence results are usually studied under stronger convergence conditions than that in the cut norm sense. For instance in~\cite{mlstatic} the convergence condition for graphon is in $L^2$ norm. In this paper, we study the limit theory under the assumption that the step graphon sequence $(G_N)_{N\geq 1}$ converges in $L^1$ norm. We leave for future work the study of the limit theory under weaker assumptions (e.g., in cut norm).
Assumption~\ref{ass:e--3}-\ref{ass:e--4} are classical conditions in graphon mean field framework to study related limit theory, see e.g. \cite{erhangmfs,phamnonlinear}: \ref{ass:e--3} is used for the propagation of chaos, and~\ref{ass:e--4} is used for the convergence rate. 

We introduce the following step graphon system, which is induced by the step graphon $G_N$ generated by $\zeta^N$. \rv{For $u\in\cI_i^N$, set $\mu_0^{N,u}=\cL(\xi^{N,i})$ and choose a square integrable initial variable $\xi^{N,u}$ with this law, independent of $W^u$. $\xi^{N,u},u\in I$ are also mutually independent. We introduce the following system} 
\begin{equation}
	\label{eq:fbn}
	\left\{
	\begin{aligned}
		&d X^{N,u}_t= b(t,X^{N,u}_t,\cG^{N,u}_t,\alpha^{N,u}_t)dt 
		+\sigma(t,X^{N,u}_t,\cG^{N,u}_t,\alpha^{N,u}_t)dW^u_t,\\
		&\textstyle dY^{N,u}_t=-\partial_xH(t,X^{N,u}_t,\cG^{N,u}_t,Y^{N,u}_t,Z^{N,u}_t,\alpha^{N,u}_t)dt + Z^{N,u}_t dW^u_t\\
		&\textstyle\qquad\qquad-\int_I\tilde{G}_N^{v,u}\t\EE[\partial_\mu  H(t,\t X^{N,v}_t,\cG^{N,v}_t,\t Y^{N,v}_t,\t Z^{N,v}_t,\t\alpha^{N,v}_t)(X^{N,u}_t)]dvdt,\\
		&\textstyle Y^{N,u}_T= \partial_xg(X^{N,u}_T,\cG^{N,u}_T)+\int_I\tilde{G}_N^{,v,u}\t\EE[\partial_\mu g(\t X^{N,v}_T,\cG^{N,v}_T)(X^{N,u}_T)]dv, 
		\\
		&X^{N,u}_0= \xi^{N,u},\; u\in I ,
	\end{aligned}
	\right.
\end{equation}
where, for chosen predictable representatives, the control is $\alpha_{t}^{N,u} = \hat{\alpha}(t,{X}_{t}^{N,u},\cG^{N,u}_t,{Y}_{t}^{N,u},{Z}_{t}^{N,u})$, \rv{for $du$-a.e. $u$ and $dt\otimes d\PP$-a.e. on the original independent noise space.} The graphon $\tilde{G}_N$ is the normalized version of $G_N$ (see~\eqref{eq:kappa}), and $\cG^N_t$ is defined analogously to~\eqref{Lab}: for any $(u,t)\in I\times [0,T]$, we set $\cG^{N,u}_t(dx)=\int_I\tilde{G}_N^{u,v}\cL(X^{N,v}_t)(dx)dv$. \rv{For $G$ and for the countable family $(G_N)_{N\geq1}$, let $I_\star$ be the intersection of the sets of labels that satisfies the characterization in Lemmas~\ref{lem:frozen-causal-realization} and \ref{lem:frozen-law-transfer} for all $G, G_N,N\geq 1$. Note that $I_\star$ is still of Lebesgue measure $1$. For each $N$ and $i\in[N]$, choose
$u_i^N\in\cI_i^N\cap I_\star$ and write $u_i=u_i^N$.  Let $\bar{\mathbf q}^{\,N}$ be the law profile of the canonical solution of \eqref{eq:fbn} and set
\[
 \Psi_N^u(x,w)=\bigl(
 \Phi_X^{\bar{\mathbf q}^{\,N},u}(x,w),
 \Phi_Y^{\bar{\mathbf q}^{\,N},u}(x,w),
 \Phi_Z^{\bar{\mathbf q}^{\,N},u}(x,w),
 [t\mapsto a_{\bar{\mathbf q}^{\,N}}^u(t,w,x)]
 \bigr),
\]
with the measurable maps in Lemma~\ref{lem:frozen-causal-realization}. Now we define 
\begin{equation}\label{eq:reference-transfer-map}
 (\bar X^i,\bar Y^i,[\bar Z^i],[\bar\alpha^i])
   =\Psi_N^{u_i^N}(\xi^{N,i},W^i),
 \qquad
 \bar\alpha_t^i=a_{\bar{\mathbf q}^{\,N}}^{u_i^N}
       (t,W^i_{\cdot\wedge t},\xi^{N,i}).
\end{equation}
Then $\bar\alpha^i$ is $\FF^{N,i}$ progressive and belongs to the finite population admissible class.} Indeed, we discretize the continuum system \eqref{eq:fbn} into $N$ pieces when the associated graphon is a step graphon. In each interval $\cI^N_i$, all labels have the same behavior and hence we use $(\o X^i,\o Y^i,\o Z^i)$ to denote the representative particle on the $i$th interval. $(\bar X^1,\cdots,\bar X^N)$ is also the solution to the system 
\eqref{fo:stateN} when the rescaled neighborhood empirical measure $\bar\nu^{N,i}_t$ is replaced by the neighborhood mean field $\cG^{N,i}_t(dx):=\sum_{j=1}^N \kappa^{N,ij}\bar{\mu}^j_t(dx)$ with $\bar{\mu}^j_t=\cL(\bar{X}^j_t)$, and $\beta^i_{t}$ is given by
$\beta_{t}^i = \bar{\alpha}_{t}^i = \hat{\alpha}(t,\bar{X}_{t}^i,\cG^{N,i}_t,\bar{Y}_{t}^i,\bar{Z}_{t}^i)$, \(dt\otimes d\PP\)-a.e.
Namely,
\begin{equation}
	\label{fo:stateN-barXi}
	d\bar{X}_{t}^i = b \bigl(t,\bar{X}_{t}^i, \cG^{N,i}_t,\bar{\alpha}_t^i \bigr) dt + \sigma(t,\bar{X}_{t}^i,\cG^{N,i}_t,\bar{\alpha}_{t}^i) dW_{t}^i, \quad \bar{X}^i_0=\xi^{N,i}, \quad1  \leq i \leq N,
\end{equation}
We denote by $\bar{J}^{N,i}$ the cost produced by $\bar{X}^i$, i.e.
$$
    \bar{J}^{N,i} = {\mathbb E}\biggl[ \int_{0}^T f(t,\bX_{t}^i,\cG_{t}^{N,i},\o \alpha_{t}^i) dt + g\bigl(\bX_{T}^i,\cG_{T}^{N,i}\bigr) \biggr],
$$
and denote by $\bar{J}^N$ the average of $\bar{J}^{N,i}, i\in[N]$. The quadruples $(\bar{X}^i,\bar Y^i,[\bar Z^i],[\bar\alpha^i])$, $1\leq i\leq N$, are independent because \eqref{eq:reference-transfer-map} realizes them as measurable functions of the mutually independent input pairs $(\xi^{N,i},W^i)$. Let us also define the weighted empirical measure of the system of particles $\bar{X}^j, j=1,\dots,N$ at time $t$,
$$\bar{\cG}^{N,i}_{t} = \sum_{j=1}^N \kappa^{N,ij} \delta_{\bar{X}_{t}^j}, \quad t \in [0,T]$$.

We first give a result measuring the difference between the neighborhood mean field $\cG^{N,i}$ and the neighborhood empirical measure $\bar{\cG}^{N,i}$ when the number of agents $N$ is large enough, which extends the classical result in \cite{fournier2015rate} on the
rate of convergence in Wasserstein distance of the empirical measure of i.i.d. random variables to the graphon framework. Let us denote
\begin{equation}
	\label{eq:def_q}
		q_{N,d,\varkappa}\coloneqq \begin{cases}
			N^{-1/2} + N^{-\varkappa/(2+\varkappa)},\; &\text{\rm if } d<4,\;  \text{\rm and } 2+\varkappa\neq 4,\\[0.3em]
			N^{-1/2}\log(1+N) + N^{-\varkappa/(2+\varkappa)},\; &\text{\rm if } d = 4,\; \text{\rm and } 2+\varkappa \neq 4,\\[0.3em]
			N^{-2/d} + N^{-\varkappa/(2+\varkappa)},\; &\text{\rm if } d>4,\; \text{\rm and } 2+\varkappa\neq d/(d-2).
		\end{cases}
	\end{equation}

To study quantitative results of propagation of chaos, we will sometimes use the following assumption.  
\begin{assumption}\label{ass:iniepsilon}
	There exists $\varkappa>0$ such that for any $t\in[0,T]$,$
        \max_{i\in [N]}\EE|\bar{X}^{i}_t|^{2+\varkappa}<\infty$, uniformly for all $N\in\NN$. 
\end{assumption}

\begin{remark}\label{rq:Lk}
Note that Assumption~\ref{ass:iniepsilon} is for every $t$ and not just time $0$. It holds for instance when the volatility is uncontrolled, see e.g. \cite[Lemma 4.3]{lauriere2022convergence}.  

\end{remark}

Lemma~\ref{lem:esempi} in appendix shows that $\EE[\cW^2_2(\o\cG^{N,i},\cG^{N,i})]\leq q_{N,d,\varkappa}$ under Assumption~\ref{ass:iniepsilon}, where $\bar{\cG}^{N,i}$ is the weighted empirical measure of the system of particles $\bar{X}^j, j=1,\dots,N$ defined above. 
We omit the proof of this lemma here for the sake of brevity. Now, we are ready to give the result of propagation of chaos. 

\begin{theorem}[Propagation of Chaos]\label{thm:poc}
	Suppose Assumptions~\ref{ass:e--1}--\ref{ass:e--3} hold. For the representatives $u_i=u_i^N\in\cI_i^N\cap I_\star$ selected above, let $W^{u_i}=W^i$. Let $\hat{\alpha}^{u_i}$ and $X^{u_i}$ be respectively  the optimal control and corresponding controlled state process associated to label $u_i$ in \eqref{eq:fbo}. Let $X^i$ be the state process of particle $i$ where the whole system \eqref{fo:stateN} is controlled by $(\h\alpha^{u_1}, \ldots,\h\alpha^{u_N})$. Assume that as $N\to \infty$, $\frac{1}{N}\sum_{i=1}^N\EE[|\xi^{u_i}-\xi^{N,i}|^2]\to 0$. Then
	$
	   \frac{1}{N}\sum_{i=1}^N\EE \Big[\sup_{0\leq t\leq T }|X^i_t-X^{u_i}_t|^2\Big]\to 0.
	$
	If, in addition, Assumptions~\ref{ass:iniepsilon} and \ref{ass:e--4} hold and $\frac{1}{N}\sum_{i=1}^N\EE[|\xi^{u_i}-\xi^{N,i}|^2]\leq \frac{C}{N}$ for some constant $C$, we have
	\begin{equation}\label{eq:poc-rate}
		\frac{1}{N}\sum_{i=1}^N\EE \Big[\sup_{0\leq t\leq T }|X^i_t-X^{u_i}_t|^2\Big]\leq Cq_{N,d,\varkappa}.
	\end{equation}
    Furthermore the above also holds if \eqref{fo:stateN} is controlled by $(\o\alpha^1,\ldots,\o\alpha^N)$. 
\end{theorem}

\begin{proof}
	We prove the second part of the statement, \eqref{eq:poc-rate}. The first part of the statement (i.e., the limit) can be proved in a similar way.  Let $\o X^u$ be the forward part of solution associated to label $u$ of \eqref{eq:fbn} with mutually independent  initial condition $\bar{\xi}$ satisfying $\cL(\o \xi^u)=\cL(\xi^{u_i})$ for $u\in \cI^N_i,i\in [N]$, and $W^{u_i}=W^i$. Let $\bar{\hat{\alpha}}^u$ denote the optimal control for label $u$ in the new system. Notice that in this new system,  $\bar{\hat{\alpha}}^u$ has the same law for all $u\in\cI^N_i,i\in N$. Let $\tilde{X}^i$ be the $i$-th particle in system \eqref{fo:stateN} where the whole system is controlled by $(\bar{\h\alpha}^{u_1}, \ldots,\bar{\h\alpha}^{u_N})$. By similar arguments as the ones used to prove propagation of chaos in non-controlled graphon systems, e.g. \cite[Theorem 4.1]{phamnonlinear}, we have
	$$
		\frac{1}{N}\sum_{i=1}^N \EE \Big[\sup_{0\leq t\leq T } |\tilde{X}^i_t-\o X^{u_i}_t|^2\Big]\leq C q_{N,d,\varkappa}+\frac{C}{N}\sum_{i=1}^N\EE[|\xi^{u_i}-\xi^{N,i}|^2],
	$$ 
	for some constant $C$. Then following the standard estimate method for graphon SDEs (see e.g. the proof of Theorem~\ref{thm:sta} for the forward part), we obtain that for all $i\in[N]$,
	$$
		\EE \Big[\sup_{0\leq t\leq T}|\tilde{X}^i_t-X^i_t|^2\Big]\leq C \int_0^T\EE[|\h\alpha^{u_i}_t-\bar{\hat{\alpha}}^{u_i}_t|^2]dt.
	$$ 
    By the definitions of $\hat{\alpha}$ and $\bar{\hat{\alpha}}$, as a consequence of Theorem~\ref{thm:conti}, we have
    for some constant $C$,
 $
 	\int_0^T\cW^2_2(\cL(\hat{\alpha}^{u_1}_t),\cL(\hat{\alpha}^{u_2}_t))dt\leq C( \cW^2_2(\cL(\xi^{u_1}),\cL(\xi^{u_2})) +\|G(u_1,\cdot)-G(u_2,\cdot)\|_1)$. Hence combining with Theorem~\ref{thm:sta}, we have
    \begin{align*}
   	 &\frac{1}{N}\sum_{i=1}^N\EE \Big[\sup_{0\leq t\leq T}|\tilde{X}^i_t-X^i_t|^2\Big]
	 \\
	 & \leq
   	 C \sum_{i=1}^N\int_{\frac{i-1}{N}}^{\frac{i}{N}}\bigg(\cW^2_2(\cL(\xi^u),\cL(\xi^{u_i})) +\|G(u,\cdot)-G(u_i,\cdot)\|_1
	  +\int_0^T\EE[|\h\alpha^u_t-\bar{\hat{\alpha}}^{u}_t|^2]dt\bigg)du
	  \\
  	  & \leq \frac{C}{N}+ \frac{CK}{N}+C\bigg(\int_{I}\cW^2_2(\cL(\bar{\xi}^u),\cL(\xi^u))du +\|G-G_N\|_1 \bigg)
	  \leq \frac{C'}{N} ,
    \end{align*}
    where the second inequality comes from the fact that there are at most $K$ intervals containing a discontinuity of $u\mapsto\cL(\xi^u)$ and $u\mapsto G(u,v)$ for every $v\in I$. On the other hand, by the stability result (Theorem~\ref{thm:sta}) again, we have
    $$
    	\int_I \EE \Big[ \sup_{0 
        \leq t \leq T} \vert \bar{X}^u_{t} - X_{t}^{u} \vert^2\Big]du \leq  C\left(\int_I\cW^2_2(\cL(\bar{\xi}^u),\cL(\xi^u))du +\|G-G_N\|_1 \right)
        \leq \frac C N. 
    $$
  Combining the above three results, we conclude the desired result. When $u_i=i/N,i\in[N]$, the state processes of system \eqref{fo:stateN} under control $(\o\alpha^1,\ldots,\o\alpha^N)$ are exactly $(\tilde{X}^1,\ldots,\tilde{X}^N)$ defined above. Hence the convergence results also follow from the analysis above.    
\end{proof}

The following theorem bounds the gap between the finite population cost under the GMFC control and the GMFC optimal cost.
\begin{theorem}\label{thm:app}
	Suppose Assumptions~\ref{ass:e--1}--\ref{ass:e--3} hold, and $\frac{1}{N}\sum_{i=1}^N \EE[|\xi^{u_i^N}-\xi^{N,i}|^2]\to 0$. Then, with $\underline{\bar{\alpha}}^N=(\bar{\alpha}^1,\cdots,\bar{\alpha}^N)$, we have
		 $|J^N( \u{\o\alpha}^N)-J^*|\to 0$ as $N$ goes to $\infty$. Moreover, if in addition Assumptions~\ref{ass:iniepsilon} and \ref{ass:e--4} hold and $\frac{1}{N}\sum_{i=1}^N \EE [|\xi^{u_i^N}-\xi^{N,i}|^2] \leq C/N$ for some constant $C$, then we have $|J^N( \u{\o\alpha}^N)-J^*|\leq C(q_{N,d,\varkappa})^{1/2}$.  
\end{theorem}

\begin{proof}
	We detail the proof for the case with convergence rate. The first assertion, without convergence rate, follows by similar arguments. 
	It suffices to prove
	$\bigl|J^N(\u{\o\alpha}^N)-J^*\bigr|\leq C (q_{N,d,\varkappa})^{1/2}$. 
    Let $X^i$ be the state process of particle $i$ in system \eqref{fo:stateN} under control $\underline{\o\alpha}^N$. By Assumption~\ref{ass:d--1}, 
	\begin{align*}
		&\EE[|f(t,X^i_t,\o\nu^{N,i}_t,\o \alpha^i_t)-f(t,X^u_t,\cG^u_t,\h\alpha^u_t)|^2]\\
		& \leq C\EE\Big[\bigl(1+|X^i_t|^2+|X^u_t|^2+|\o\alpha^i_t|^2+|\h\alpha^u_t|^2+\|\o\nu^{N,i}_t\|^2_2+\|\cG^u_t\|^2_2\bigr)
		\\
		&\qquad\qquad \times \bigl(|X^i_t-X^u_t|^2+|\o\alpha^i_t-\h\alpha^u_t|^2+\cW^2_2(\o\nu^{N,i}_t,\cG^u_t)\big)\Big]
		\\
		&\leq C \bigg(\EE[|X^i_t-X^u_t|^2]+\EE[|\o\alpha^i_t-\h\alpha^u_t|^2]+\EE[\cW^2_2(\o\nu^{N,i}_t,\bar{\cG}^{N,i}_t)]
		\\
		&\qquad\qquad +\EE[\cW^2_2(\bar{\cG}^{N,i}_t,\cG^{N,i}_t)]+\cW^2_2({\cG}^{N,i}_t,\cG^{u}_t)\bigg).
	\end{align*}
	First of all, we have $\EE[\cW^2_2(\bar{\cG}^{N,i}_t,\cG^{N,i}_t) ]\leq q_{N,d,\varkappa}$ by Lemma~\ref{lem:esempi}.
	Then by the propagation of chaos result (Theorem~\ref{thm:poc}), the continuity result (Theorem~\ref{thm:conti}), and the stability result (Theorem~\ref{thm:sta}), it follows by adding and subtracting terms that
	\begin{align*}
		\bigg|\frac{1}{N}&\sum_{i=1}^N\EE\bigg[\int_0^T f(t,X^i_t,\o\nu^{N,i}_t,\o \alpha^i_t)dt\bigg]-\int_I \EE\bigg[\int_0^T f(t,X^u_t,\cG^u_t,\h\alpha^u_t)dt\bigg]du \bigg|^2\\
		&\leq C \bigg( \int_{I}\EE\bigg[\int_0^T|\bar{\alpha}^{i_N(u)}_t-\h\alpha^{u^N_{i_N(u)}}_t|^2dt\bigg] du +\int_I \EE\bigg[\int_0^T|X^{i_N(u)}_t-X^u_t|^2 dt\bigg]du+q_{N,d,\varkappa}\\
        & \quad\quad\quad+\int_I\int_0^T\cW^2_2({\cG}^{N,i_N(u)}_t,\cG^{u}_t)dtdu\bigg)+\frac{C}{N}\sum_{i=1}^N\EE\bigg[\int_0^T|\o X^i_t-X^i_t|^2dt\bigg]\\
		& \leq C q_{N,d,\varkappa}.
	\end{align*}
	Similar arguments can be applied to $|\frac 1 N\sum_{i=1}^N\EE[g(X^i_T,\o\nu^{N,i}_T)]-\int_I \EE[g(X^u_T,\cG^u_T)]du|^2$. Hence we obtain that 
    $
        \bigl|J^N(\u{\o\alpha}^N)-J^*\bigr|\leq C (q_{N,d,\varkappa})^{1/2}.
    $
\end{proof}

Finally, we show that %
$(\bar{\alpha}^1,\cdots,\bar{\alpha}^N)$ is approximately optimal for the problem with $N$ agents
and the optimal cost for $N$ agents converges to the optimal GMFC cost. 
\begin{theorem}
	\label{th:limit cost}
	Under Assumptions~\ref{ass:e--1}--\ref{ass:e--3} and if $\frac{1}{N}\sum_{i=1}^N \EE[|\xi^{u_i^N}-\xi^{N,i}|^2]\to 0$, we have	$
		\lim_{N \rightarrow + \infty} \inf_{\u \beta^N\in \cM\bH^2_N} J^N( \u \beta^N) = J^*. 
	$
	Moreover, %
	$
        \lim_{N \rightarrow + \infty} J^N( \u{\o \alpha}^N) = J^*.
    	$
    If, in addition, Assumptions~\ref{ass:iniepsilon} and \ref{ass:e--4} hold and $\frac{1}{N}\sum_{i=1}^N\EE[|\xi^{u_i^N}-\xi^{N,i}|^2]\leq C/N$ for some constant $C$, then 
    \begin{equation}\label{eq:jj}
		\Big|J^*-\inf_{\u \beta^N\in \cM\bH^2_N} J^N( \u \beta^N)\Big| 
		\leq C (q_{N,d,\varkappa})^{1/2}, 
	\end{equation}
	     In particular, $\u{\o\alpha}^N$ is approximately optimal in the sense $$|J^N( \u{\o \alpha}^N) -\inf_{\u \beta^N\in \cM\bH^2_N} J^N( \u \beta^N)| \le C(q_{N,d,\varkappa})^{1/2}$$.
\end{theorem}

\begin{proof}
	  We prove in detail the case with convergence rate, i.e., \eqref{eq:jj}. The case without convergence rate follows by taking limits. Fix $\underline{\beta}^N\in\cM\bH_N^2$. We first compare $J^{N,i}(\underline{\beta}^N)$ and $\bar{J}^{N,i}$.  Recall that $\bar{\nu}^{N,i}$ is the rescaled neighborhood empirical measure of the particles in \eqref{fo:stateN} under $\underline{\beta}^N$. We obtain
	\begin{align*}
		&J^{N,i}(\u \beta)-\bar{J}^{N,i}
		\\
		&= \EE \biggl[ \int_{0}^T \bigl( f(s,X_{s}^i,\bar{\nu}_{s}^{N,i},\beta_{s}^i)  
		- f(s,\bX_{s}^i,\cG^{N,i}_{s},\balpha_{s}^{i})  \bigr) ds \biggr]
        + \EE \bigl[ g(X^i_{T},\bar{\nu}^{N,i}_{T}) - g(\bX_{T}^i,\cG^{N,i}_{T}) \bigr].
	\end{align*}
	Then we proceed to estimate the above. %
	By subtracting terms, we write
	\begin{equation}
		\label{eq:28:2:30}
		J^{N,i}(\u \beta)-\bar{J}^{N,i}
		= T_{1}^i + T_{2}^i,
	\end{equation} 
	with
	\begin{equation*}
		\begin{split}
			T_{1}^i &= \EE \t\EE \bigl[( X^i_{T}-\bX^i_{T}) \cdot \bar{Y}^i_{T}  \bigr]
			+ \EE \biggl[ \int_{0}^T \bigl( f(s,X_{s}^i,\bar{\nu}_{s}^{N,i},\beta_{s}^i)  - f(s,\bX_{s}^i,\cG^{N,i}_{s},\balpha_{s}^{i})  \bigr) ds \biggr],
			\\
			T_{2}^i &= \EE \bigl[ g(X^i_{T},\bar{\nu}_{T}^{N,i}) - g(\bX_{T}^i,\cG^{N,i}_{T}) \bigr] 
			- \EE \bigl[(X_{T}^i - \bX_{T}^i)\cdot \partial_{x} g(\bX_{T}^i,\cG^{N,i}_{T}) \bigr]
			\\
			&\phantom{??????????}-  \sum_{j=1}^N \kappa^{N,ji}\t\EE \EE \bigl[(\t X_{T}^{i} - \t{\bX}_{T}^{i})\cdot \partial_{\mu} g(\bX_{T}^j,\cG^{N,j}_{T})(\t{\bX}_{T}^{i}) \bigr]
			\\
			&=: T_{2,1}^i - T_{2,2}^i - T_{2,3}^i.
		\end{split}
	\end{equation*}
	
	\noindent{\bf Analysis of $T_{2}^i$.} 
	Let us first analyze the third term in $T^i_2$, i.e., $T^i_{2,3}$. Let $(\t{\o X}^{j,k})_{k=1}^N$ be $N$ independent copies of $\o X^j$. By subtracting and adding $\frac{1}{N}\sum_{j=1}^N\sum_{k=1}^N\kappa^{N,ji}(\t X_{T}^{i} - \t{\bX}_{T}^{i})\cdot \partial_{\mu} g(\t{\bX}_{T}^{j,k},\cG^{N,j}_{T})(\t{\bX}_{T}^{i})$, we have\footnote{For three functions $f,g,h$, we use the notation $f(N) \ge g(N)\mathcal{O}(h(N))$ to mean that $f(N)$ is greater than a function which is of order $g(N)\mathcal{O}(h(N))$.}
	\begin{align*}
		&\sum_{j=1}^N\kappa^{N,ji}\t\EE \EE  \bigl[(\t X_{T}^{i} - \t{\bX}_{T}^{i})\cdot  \partial_{\mu} g(\bX_{T}^j,\cG^{N,j}_{T})(\t{\bX}_{T}^{i}) \bigr]
		\\
		&= \sum_{j=1}^N\kappa^{N,ji}\EE \bigl[(X_{T}^{i} - \bX_{T}^{i})\cdot \partial_{\mu} g(\bX_{T}^{j},\cG^{N,j}_{T})(\bX_{T}^{i}) \bigr]+[\EE|X^i_T-\bX^i_T|^2]^{1/2}\mathcal{O}(N^{-1/2}),
	\end{align*}
	where we used the Lipschitz continuity of $\partial_\mu g$, the Cauchy--Schwarz inequality and the uniformly finite $L^2$ norm of $\bX^i_T$ over all $i\in[N]$. For each $i\in [N]$, we define $\vartheta_i$, a random variable taking values in $\{1,\dots,N\}$, defined on $(\tilde{\Omega},\tilde{\cF},\tilde{\PP})$ with distribution  $\tilde{\PP}(\vartheta_i=j)=\kappa^{N,ij}$, where we recall that $\kappa^{N,ij}$ is defined in~\eqref{eq:def-kappaN}.
	Hence, we have
	\begin{align*}
		&\frac{1}{N}\sum_{i=1}^N T^i_{2,3} 
		\\
		&=\frac{1}{N}\sum_{i=1}^N\EE\tilde{\EE}[(X^{\vartheta_i}_T-\bX^{\vartheta_i}_T)\cdot\partial_{\mu} g(\bX_{T}^i,\cG^{N,i}_{T})(\bX_T^{\vartheta_i})]+\frac{1}{N}\sum_{i=1}^N[\EE|X^i_T-\bX^i_T|^2]^{1/2}\mathcal{O}(N^{-1/2}).
	\end{align*}

	Recall the definition of $\o\cG^{N,i}$ and Lemma~\ref{lem:esempi}. %
    Using Assumption~\ref{ass:e--2}, we have
	\begin{align*}
		\frac{1}{N} \sum_{i=1}^N T_{2,3}^i
		&= \frac{1}{N}\sum_{i=1}^N\EE\tilde{\EE}[(X^{\vartheta_i}_T-\bX^{\vartheta_i}_T)\cdot\partial_{\mu} g(\bX_{T}^i,\o\cG^{N,i}_{T})(\bX_T^{\vartheta_i})]
		\\
		&\qquad+\frac{1}{N}\sum_{i=1}^N[\EE|X^i_T-\bX^i_T|^2]^{1/2}\mathcal{O}(N^{-1/2}+\big(q_{N,d,\varkappa})^{1/2}\big).
	\end{align*}
	
	Since $\partial_xg$ is Lipschitz, $(\o X^i)_{i\in[N]}$ are bounded in the $\mathbf{S}^2$ norm and using Assumption~\ref{ass:d--1}, by applying similar arguments for terms $T^i_{2,1}$ and $T^i_{2,2}$ and noting that $N^{-1/2}\leq (q_{N,d,\varkappa})^{1/2}$, we conclude that
	\begin{equation*}
		\begin{split}
			\frac{1}{N} \sum_{i=1}^N T_{2}^i &= \frac{1}{N} \sum_{i=1}^N  
			\biggl\{ \EE \bigl[ g(X^i_{T},\bar{\nu}_{T}^{N,i}) - g(\bX_{T}^i,\bar{\cG}_{T}^{N,i}) \bigr] 
			- \EE \bigl[ (X_{T}^i - \bX_{T}^i)\cdot \partial_{x} g(\bX_{T}^i,\bar{\cG}_{T}^{N,i})  \bigr]
			\\
			&\hspace{15pt} - \EE \t \EE \bigl[ (X_{T}^{\vartheta_i} - \bX_{T}^{\vartheta_i})\cdot \partial_{\mu}g
			\bigl(\bX_{T}^{i},\bar{\cG}_{T}^{N,i})(\bX_{T}^{\vartheta_i}) \bigr]
			\biggr\} \\
			&\hspace{15pt}  +  \bigl( 1 +  \frac{1}{N}\sum_{i=1}^N\bigl[{\mathbb E} \vert X_{T}^{i} - \bX_{T}^{i}
			\vert^2 \bigr]^{1/2} \bigr) {\mathcal O}\big((q_{N,d,\varkappa})^{1/2}\big),
		\end{split}
	\end{equation*}
	where we replace $\cG^{N,i}_T$ by $\o\cG^{N,i}_T$ in $g$ by using the local Lipschitz property of $g$, and Assumption~\ref{ass:d--1}. %
	Then by the convexity property of $g$, we get
	\begin{equation}
		\label{eq:14:3:3}
		\frac{1}{N} \sum_{i=1}^N T_{2}^i \geq -C (q_{N,d,\varkappa})^{1/2} \bigl( 1 +  \bigl[\frac{1}{N}\sum_{i=1}^N{\mathbb E} \vert X_{T}^{i} - \bX_{T}^{i}
		\vert^2 \bigr]^{1/2} \bigr) .
	\end{equation}
	
	\noindent{\bf Analysis of $T_{1}^i$.} 
	It\^o's formula and Fubini's theorem give
	\begin{equation}
		\begin{split}
			T_{1}^i &=  \EE \biggl[ \int_{0}^T 
			\bigl( H(t,X_{t}^i,\bar{\nu}_{t}^{N,i},\bar{Y}^i_t,\bar{Z}_{t}^i,\beta_{t}^i)  
			- H(t,\bX_{t}^i,\cG^{N,i}_{t},\bar{Y}^i_t,\bar{Z}_{t}^i,\balpha_{t}^{i})  \bigr) dt \biggr]
			\\
			&\hspace{15pt}
			- \EE \biggl[ \int_{0}^T (X_{t}^i - \bX_{t}^i)\cdot \partial_{x} H(t,\bX_{t}^i,\cG^{N,i}_{t},\bar{Y}_{t}^i,
			\bar{Z}_{t}^i,\balpha_{t}^{i})
			dt \biggr]
			\\
			&\hspace{15pt}
			- \sum_{j=1}^N\kappa^{N,ji}\EE \t\EE \biggl[ \int_{0}^T (\t X_{t}^{i} - \t{\bX}_{t}^{i})\cdot 
			\partial_{\mu} H(t,\bX_{t}^j,\cG^{N,j}_{t},\bar{Y}_{t}^j,\bar{Z}_{t}^j,\balpha_{t}^{j})(\t{X}_{t}^{i})
			dt \biggr]
			\\
			&=: T_{1,1}^i - T_{1,2}^i - T_{1,3}^i. 
		\end{split}
	\end{equation}
	The Hamiltonian $H$ is locally Lipschitz with respect to all parameters except $t$, the processes $(\bar{X}^i,\bar{Y}^i)_{i\in[N]}$ are bounded in $\mathbf{S}^2$, and $(\bar{Z}^i,\bar{\alpha}^i)_{i\in[N]}$ are bounded in $\mathbf{H}^2$. By Assumption~\ref{ass:d--1} %
	and Lemma~\ref{lem:esempi}, we have
	\begin{align*}
		\frac{1}{N}\sum_{i=1}^N T_{1,1}^i 
		&=  \frac{1}{N}\sum_{i=1}^N\EE \biggl[ \int_{0}^T 
		\bigl(  H(t,X_{t}^i,\bar{\nu}_{t}^{N,i},\bar{Y}^i_t,\bar{Z}_{t}^i,\beta_{t}^i)  
		- H(t,\bX_{t}^i,\bar{\cG}_{t}^{N,i},\bar{Y}^i_t,\bar{Z}_{t}^i,\balpha_{t}^{i})  \bigr) dt \biggr] 
		\\
		&\qquad 
		+ {\mathcal O}\big((q_{N,d,\varkappa})^{1/2}\big).
	\end{align*}
	Similarly, by Lipschitz property of $\partial_x H$, we have
	\begin{equation*}
		\begin{split}
			\frac{1}{N}\sum_{i=1}^NT_{1,2}^i
			&= 
			\frac{1}{N}\sum_{i=1}^N\EE \biggl[ \int_{0}^T 
			( X_{t}^i - \bX_{t}^i)\cdot \partial_{x} H(t,\bX_{t}^i,\bar{\cG}_{t}^{N,i},\bar{Y}_{t}^i,\bar{Z}_{t}^i,\balpha_{t}^{i})
			dt \biggr]  
			\\
			&\qquad +  \frac{1}{N}\sum_{i=1}^N\biggl(\EE\bigl[ \int_{0}^T \vert  X_{t}^i - \bX_{t}^i \vert^2 dt\bigr] \biggr
			)^{1/2} {\mathcal O}\big((q_{N,d,\varkappa})^{1/2}\big).
		\end{split}
	\end{equation*}
	Finally, using similar arguments for $\partial_\mu H$ as the ones used for $\partial_\mu g$, we have
	\begin{equation*}
		\begin{split}
			\frac{1}{N} \sum_{i=1}^N T_{1,3}^i			
			&= \frac{1}{N} \sum_{i=1}^N \EE \t \EE \biggl[ \int_{0}^T (X_{t}^{\vartheta_i} -\bX_{t}^{\vartheta_i})\cdot 
			\partial_{\mu} H(t,\bX_{t}^{i},\bar{\cG}^{N,i}_{t},\bar{Y}_{t}^{i},\bar{Z}_{t}^i,\balpha_{t}^{i})(\bX_{t}^{\vartheta_i}) dt \biggr]
			\\
			&\hspace{30pt}
			+ \frac{1}{N}\sum_{i=1}^N\biggl( \EE\bigl[ \int_{0}^T \vert  X_{t}^{i} - \bX_{t}^i \vert^2 dt\bigr] \biggr)^{1/2} {\mathcal O}\big((q_{N,d,\varkappa})^{1/2}\big).
		\end{split}
	\end{equation*} 
	
		Since $\o\alpha^i_t$ is the unique minimizer of $\alpha\mapsto H(t,\bX_{t}^i,\cG^{N,i}_{t},\bar{Y}_{t}^i,\bar{Z}_{t}^i,\alpha)$ and $A=\RR^k$ in this section, for the chosen predictable representative,
	    $
	        \partial_{\alpha} H(t,\bX_{t}^i,\cG^{N,i}_{t},\bar{Y}_{t}^i,\bar{Z}_{t}^i,\balpha_{t}^{i})=0
	    $ 
	    $dt\otimes d\PP$-a.e.  Thus
	\begin{align*}
		&\EE \biggl[ \int_{0}^T 
		(\beta^i_{t} - \balpha_{t}^{i}) \cdot \partial_{\alpha} H(t,\bX_{t}^i,\bar{\cG}_{t}^{N,i},\bar{Y}_{t}^i,\bar{Z}_{t}^i,
		\balpha_{t}^{i}) dt \biggr]
            \\
		&=  \EE \biggl[ \int_{0}^T 
		(\beta^i_{t} - \balpha_{t}^{i}) \cdot \partial_{\alpha} H(t,\bX_{t}^i,\cG^{N,i}_{t},\bar{Y}_{t}^i,\bar{Z}_{t}^i,\balpha_{t}^{i}) dt \biggr] 
		+  \biggl(  \EE \int_{0}^T \vert  \beta_{t}^i - \balpha_{t}^{i}  \vert^2 dt \biggr)^{1/2} {\mathcal O}\big((q_{N,d,\varkappa})^{1/2}\big)
            \\
		&=   \biggl(  \EE \int_{0}^T \vert  \beta_{t}^i - \balpha_{t}^{i}  \vert^2 dt \biggr)^{1/2} {\mathcal O}\big((q_{N,d,\varkappa})^{1/2}\big).
	\end{align*}
	Hence using again the convexity of $H$, we get that 
	\begin{align*}
			\frac{1}{N} \sum_{i=1}^N T_{1}^i 
			&\geq 
			-C(q_{N,d,\varkappa})^{1/2} \biggl( 1+	\frac{1}{N} \sum_{i=1}^N \sup_{0 \leq t \leq T} \EE  \bigl[ \vert  X_{t}^i - \bX_{t}^i \vert^2 \bigr] 
			+	\frac{1}{N} \sum_{i=1}^N \EE \int_{0}^T  \vert  \beta_{t}^i - \balpha_{t}^i   \vert^2  dt \biggr)^{1/2}\\
            &\quad + \frac{C}{N} \sum_{i=1}^N \EE \int_{0}^T  \vert  \beta_{t}^i - \balpha_{t}^i   \vert^2  dt.
	\end{align*}
	Finally by \eqref{eq:28:2:30}, \eqref{eq:14:3:3} and the above inequality, we deduce that 
	\begin{align*}
			&J^N( \u \beta^N) 
			\\
			&
			\geq \o J^N
			-C (q_{N,d,\varkappa})^{1/2} \biggl(  1+\frac{1}{N} \sum_{i=1}^N \sup_{0 \leq t \leq T} \EE  \bigl[ \vert  X_{t}^i - \bX_{t}^i \vert^2 \bigr]  +	\frac{1}{N} \sum_{i=1}^N \EE \int_{0}^T  \vert  \beta_{t}^i - \balpha_{t}^i   \vert^2  dt\biggr)^{1/2}\\
            &\quad +\frac{C}{N} \sum_{i=1}^N \EE \int_{0}^T  \vert  \beta_{t}^i - \balpha_{t}^i   \vert^2  dt.
	\end{align*}
	By a standard estimate for the forward SDEs, together with Lemma~\ref{lem:esempi} for the empirical measure term, we have for some constant $C$ independent of $N$,
\begin{equation*}
		\frac{1}{N} \sum_{i=1}^N \EE\left[\sup_{0 \leq t \leq T}\vert  X_{t}^i -\bX_{t}^i \vert^2\right] \leq C\left( \frac{1}{N} \sum_{i=1}^N \EE \int_{0}^T 
		\vert  \beta_{t}^i - \balpha_{t}^i  \vert^2 dt+q_{N,d,\varkappa}\right). 
\end{equation*}
Notice that for any $N$, we have 
\begin{align*}
 &\frac{C_1}{N} \sum_{i=1}^N \EE \int_{0}^T 
		\vert  \beta_{t}^i - \balpha_{t}^i  \vert^2 dt  -C_2 (q_{N,d,\varkappa})^{1/2}\bigg(1+\frac{C_3}{N} \sum_{i=1}^N \EE \int_{0}^T 
		\vert  \beta_{t}^i - \balpha_{t}^i  \vert^2 dt\bigg)^{1/2}\\
        & \geq \bigg(\frac{1}{N} \sum_{i=1}^N \EE \int_{0}^T 
		\vert  \beta_{t}^i - \balpha_{t}^i  \vert^2 dt\bigg)^{1/2}\bigg\{C_1\bigg(\frac{1}{N} \sum_{i=1}^N \EE \int_{0}^T 
		\vert  \beta_{t}^i - \balpha_{t}^i  \vert^2 dt\bigg)^{1/2}-C_2 (q_{N,d,\varkappa})^{1/2}\sqrt{C_3}\bigg\}-C_2(q_{N,d,\varkappa})^{1/2}\\
        &\geq -(C_2+C_2\sqrt{C_3})(q_{N,d,\varkappa})^{1/2}.
\end{align*}
It then follows that for a new constant $C$,  
	\begin{equation}
		J^N(\u \beta^N) \geq \o J^N - C (q_{N,d,\varkappa})^{1/2}.
	\end{equation}
	 Finally, combining this with the stability result of graphon mean field FBSDE system (Theorem~\ref{thm:sta}), we have that
	\begin{equation}
		J^N(\u \beta^N) \geq J^* - C (q_{N,d,\varkappa})^{1/2}-C\Big(\|G_N-G\|_{1}+\int_I\cW_2^2(\mu_0^{N,u},\mu_0^u)du\Big)^{1/2},
	\end{equation}
	which shows by our assumptions that 
	\begin{equation}\label{eq:j}
		 \inf_{\u \beta^N\in \cM\bH^2_N} J^N (\u \beta^N) \geq J^*- C (q_{N,d,\varkappa})^{1/2}. 
	\end{equation}
	Combining this with Theorem~\ref{thm:app}, 
    we obtain $\Big|J^*-\inf_{\u \beta^N\in \cM\bH^2_N} J^N( \u \beta^N)\Big| \leq C (q_{N,d,\varkappa})^{1/2}$, since 
    $$\inf_{\u \beta^N\in \cM\bH^2_N} J^N( \u \beta^N)\leq J^N (\u {\o\alpha}^N)\leq J^*+C (q_{N,d,\varkappa})^{1/2}.$$
    

\end{proof}

{
\small

\bibliographystyle{abbrv}
\bibliography{bib_control}

}

\begin{appendix}

\section{Additional details for Section~\ref{sec:eu}}
\label{app:proofs-sec2}

\rv{
\begin{lemma}[Blackwell--Dubins]\label{le:BlackwellDubins}
For any Polish space $B$ and any standard probability space
$(\mathcal U,\mathcal F_{\mathcal U},\mathbb P_{\mathcal U})$, there exists
a measurable function $\rho_B:\mathcal P(B)\times\mathcal U\to B$ satisfying:
\begin{enumerate}[label=\roman*),topsep=0pt,itemsep=0pt,partopsep=0pt,parsep=0pt]
\item for each $\nu\in\mathcal P(B)$, the random variable with values in $B$
$\rho_B(\nu,\cdot)$ has distribution $\nu$.
\item For almost every $\mathfrak u\in\mathcal U$ with respect to $\mathbb P_{\mathcal U}$,
the function $\mathcal P(B)\ni\nu\mapsto\rho_B(\nu,\mathfrak u)\in B$ is
continuous.
\end{enumerate}
\end{lemma}
}

\rv{
\begin{proof}[Proof of Lemma~\ref{lem:coupleini}]
Apply Lemma~\ref{le:BlackwellDubins} with
\((\mathcal U,\mathcal F_{\mathcal U},\mathbb P_{\mathcal U})
=(\Upsilon,\mathcal B(\Upsilon),dr)\).
Let
\[
 C_0=\esssup_{u\in I}\int_{\RR^d}|x|^{2+\varepsilon}\mu_0^u(dx),
 \qquad
 \mathfrak P=\left\{\nu\in\mathcal P(\RR^d):
 \int_{\RR^d}|x|^{2+\varepsilon}\nu(dx)\leq C_0\right\}.
\]
The family
\((\rho_{\RR^d}(\nu,\cdot))_{\nu\in\mathfrak P}\) is uniformly integrable
in \(L^2(\Upsilon;\RR^d)\).  The a.e. pointwise continuity supplied by
Lemma~\ref{le:BlackwellDubins} and uniform integrability therefore imply
that
\[
 \mathfrak P\ni\nu\longmapsto
 \rho_{\RR^d}(\nu,\cdot)\in L^2(\Upsilon;\RR^d)
\]
is continuous.  Since \(u\mapsto\mu_0^u\) is measurable and
\(\mu_0^u\in\mathfrak P\), the asserted measurability follows.
\end{proof}
}

 \begin{proof}[Proof of Lemma~\ref{lem:barx}]
 We follow similar arguments as in the proof of Theorem~\ref{thm:exf}. For  each $u\in I$, let  $\o X^{u,0}_t=\o\xi^u$ for all $t\in[0,T]$. For $n \ge 1$, denote $\mu^{u,n-1}:=\cL(\o X^{u,n-1})$, and for each $u\in I$, define $\cG^{u,n-1}: [0,T] \to \cP(\RR^d)$ as
	$$
        \cG^{u,n-1}_t(dx)=\t G\mu^{n-1}_t[u](dx):=\int_I \t G(u,v)\mu^{v,n-1}_t(dx)dv,
    $$
 and define
	\begin{equation*}
		\o X^{u,n}_t=\o X^{u,n-1}_0+\int_0^t b^u(s,\o X^{u,n-1}_s,\cG^{u,n-1}_s,\o\alpha^u_s) ds+\int_0^t \sigma^u(s,\o X^{u,n-1}_s,\cG^{u,n-1}_s,\o\alpha^u_s) d\o W_s.
	\end{equation*} 
By classical iteration arguments, it suffices to prove that for any $t\in[0,T]$, $I\ni u\mapsto \bar{X}^{u,n}_t\in L^2(\bar{\Omega};\RR^d)$ is measurable provided $(\o X^{u,n-1})_{u\in I}$ satisfy the same property. \rv{Let us first analyze the drift term. By \cite[Lemma A.3]{wu2022}, for a fixed $\o\beta\in\bar{\bH}^2_T(A)$ we have for any $t\in[0,T]$, the functions $u\mapsto b^u(t,\o X^{u,n-1}_t,\cG^{u,n-1}_t,\o\beta_t) \in L^2(\o\Omega;\RR^d)$ is measurable. Then by applying \cite[Lemma A.4]{wu2022}, $u\mapsto \int_0^t b^u(s,\o X^{u,n-1}_s,\cG^{u,n-1}_s,\o\beta_s)ds \in L^2(\o\Omega;\RR^d)$ is also measurable for any $t\in[0,T]$. On the other hand, for fixed $u\in I$, we have for any $t\in[0,T]$ and any $\o\beta^1,\o\beta^2\in \o\bH^2_T(A)$,
 \begin{align*}
  &  \EE\bigg[\bigg|\int_0^t \big(b^u(s,\o X^{u,n-1}_s,\cG^{u,n-1}_s,\o\beta^1_s)-b^u(s,\o X^{u,n-1}_s,\cG^{u,n-1}_s,\o\beta^2_s)\big)ds\bigg|^2\bigg] \\
  & \leq CT \EE\bigg[\int_0^t|\o \beta^1_s-\o\beta^2_s|^2ds\bigg]=C\|\o\beta^1-\o\beta^2\|^2_{\bar{\mathbb{H}}^2_T}.
\end{align*}
The above shows that the mapping $$\tilde{b}:I\times \bar{\bH}^2_T(A)\to L^2(\o\Omega;\RR^d);(u,\o\beta)\mapsto \int_0^t b^u(s,\o X^{u,n-1}_s,\cG^{u,n-1}_s,\o\beta_s)ds$$ 
is measurable for every $\o\beta$ and is Lipschitz continuous in $\bar{\beta}$ for every $u$. It then follows from the definition of admissible control profile (see Definition~\ref{def:admisscontrol}) that $u\mapsto \int_0^t b^u(s,\o X^{u,n-1}_s,\cG^{u,n-1}_s,\o\alpha^u_s)ds \in L^2(\o\Omega;\RR^d)$ is measurable. Then we can similarly prove the measurability of $$u \mapsto\int_0^t \sigma^u(s,\o X^{u,n-1}_s,\cG^{u,n-1}_s,\o\alpha^u_s)d\bar{W}_s \in L^2(\o\Omega;\RR^{d})$$ by following the idea used above. Noticing that for each $\o\beta\in \bar{\bH}^2_T(A)$, the It\^o integral $$\int_0^t\sigma^u(s,\o X^{u,n-1}_s,\cG^{u,n-1}_s,\o\beta_s)d\o W_s$$ is the $L^2$ limit of a finite sum of measurable mappings. On the other side, for each $u\in I$, we also have $\bar{\bH}^2_T(A)\ni\o\beta\mapsto \int_0^t\sigma^u(s,\o X^{u,n-1}_s,\cG^{u,n-1}_s,\o\beta)d\o W_s\in L^2(\o\Omega;\RR^d)$ is Lipschitz continuous by It\^o isometry. Hence we get $u \mapsto\int_0^t \sigma^u(s,\o X^{u,n-1}_s,\cG^{u,n-1}_s,\o\alpha^u_s)d\bar{W}_s \in L^2(\o\Omega;\RR^{d})$ is also measurable. Thus we have 
$$
	u\mapsto \bigg(\int_0^tb^u(s,\o X^{u,n-1}_s,\cG^{u,n-1}_s,\o\alpha^u_s)ds,\int_0^t\sigma^u(s,\o X^{u,n-1}_s,\cG^{u,n-1}_s,\o\alpha^u_s)d\bar{W}_s\bigg)\in L^2(\o\Omega;\RR^d\times \RR^{d})
$$ 
is measurable. Hence, for every $t\in[0,T]$, the map $I\ni u\mapsto \bar{X}^{u,n}_t\in L^2(\bar{\Omega};\RR^d)$ is measurable. Since $\o\cF_T$ is countably generated by the initial condition and an $m$-dimensional Brownian motion up to null sets, $L^2(\o\Omega;\cC^d)$ is separable. An approximation by simple functions then yields the measurability of $I\ni u\mapsto \bar{X}^{u,n}\in \bar{\mathbf{S}}^2$.
}
 \end{proof}

\rv{
\begin{proof}[Proof of Lemma~\ref{lem:frozen-causal-realization}]
By \eqref{eq:frozen-profile-moment}--\eqref{eq:frozen-optimal-L2-data}
and our assumptions, the profile moments and the coefficient
\(L^2\) quantities in before integration over
\(u\), are finite for almost every \(u\in I\). We can choose a Borel set
\(I_{\mathbf q}\subset I\) of Lebesgue measure $1$ such that on which these
quantities are finite. Fix a \(u\in I_{\mathbf q}\) and consider the canonical probability space
\[
 \Omega^u=\RR^d\times\cC_0^m,\qquad
 \nu_u=\mu_0^u\otimes\mathbb W.
\]
Define the initial coordinate \(\xi(x,w):=x\) and the coordinate
Brownian motion \(W_t(x,w):=w(t)\), \(0\leq t\leq T\).
Let \(\mathbb G^u\) be the completion of the filtration generated by \((\mathcal G_t^{u,0}=\sigma(\xi^u,W_s:0\leq s\leq t))_{t\geq 0}\). With \(u,\mathbf q\) fixed, by following similar arguments in the proofs of Lemma~\ref{lem:contra}
and Theorem~\ref{th:exo}, there exists a unique \(\mathbb G^u\)-adapted solution
\((X^u,Y^u,Z^u)\in\mathbf S^2\times\mathbf S^2\times\mathbf H^2\)
of \({\rm FBSDE}(u,\mathbf q;\xi,W)\).

Since the completed coordinate filtration is right continuous,
each \(\mathcal G_t^u\)-measurable random variable has a
\(\mathcal G_t^{u,0}\)-measurable version. We take a dyadic time partition. At each partition time \(r\), the Doob-Dynkin lemma guarantees that there exists measurable functions of
\((x,w_{\cdot\wedge r})\) equaling to \(X_r^u,Y_r^u\), \(\nu_u\)-a.s.
Denote their left endpoint step interpolation of the \(n\) partition by \(\chi_n^u,\upsilon_n^u\) respectively. Note that these functions are progressive. Then letting $n$ goes to infinity, the continuity of \(X^u,Y^u\) on $[0,T]$ implies
the uniform convergence in time to \((X^u,Y^u)\), \(\nu_u\)-a.s.
Define the progressive set
\[
 D^u=\bigl\{(t,x,w)\in[0,T]\times\Omega^u:
   (\chi_n^u(t,x,w),\upsilon_n^u(t,x,w))_{n\geq1}
   \text{ converges in }\RR^{2d}\bigr\}.
\]
Set \((\chi^{\mathbf q,u},\upsilon^{\mathbf q,u})
=\lim_n(\chi_n^u,\upsilon_n^u)\) on \(D^u\), and set the pair
equal to \((0,0)\) on \(([0,T]\times\Omega^u)\setminus D^u\).
These progressive measurable functions satisfy
\[
 \chi^{\mathbf q,u}(t,\xi,W)=X_t^u,\qquad
 \upsilon^{\mathbf q,u}(t,\xi,W)=Y_t^u
 \quad\hbox{for all }t,\quad\nu_u\hbox{-a.s.}
\]

From the above analysis, outside a Borel null set $N^u$ on $\Omega^u$, the paths \(\chi^{\mathbf q,u}(\cdot,x,w)\) and
\(\upsilon^{\mathbf q,u}(\cdot,x,w)\) are continuous. Define
\[
 (\Phi_X^{\mathbf q,u}(x,w),\Phi_Y^{\mathbf q,u}(x,w))
 =\bigl(\chi^{\mathbf q,u}(\cdot,x,w),
         \upsilon^{\mathbf q,u}(\cdot,x,w)\bigr)
 \quad\text{on }\Omega^u\setminus N^u,
\]
and assign the pair of zero paths on \(N^u\). Since these evaluations generate the Borel sigma-field of
\(\cC^d\times\cC^d\), the path maps are Borel measurable and they coincide with \(\chi^{\mathbf q,u},\upsilon^{\mathbf q,u}\)
simultaneously in time, \(\nu_u\)-a.s..

For \(Z^u\), choose elementary predictable approximations
\[
 z^{u,k}(t,x,w)=\sum_{j=0}^{N_k-1}
 c_j^{u,k}(x,w_{\cdot\wedge t_j^k})
 \mathbf 1_{(t_j^k,t_{j+1}^k]}(t),\qquad
 \|z^{u,k}-Z^u\|_{\mathbf H^2}^2\leq 2^{-k},
\]
where \(0=t_0^k<\cdots<t_{N_k}^k=T\). We can take \(c_j^{u,k}\) Borel measurable. Define the set
\[
 D_Z^u=\bigl\{(t,x,w)\in[0,T]\times\Omega^u:
   (z^{u,k}(t,x,w))_{k\geq1}
   \text{ converges in }\RR^{d\times m}\bigr\},
\]
which is predictable. Then let \(z^{\mathbf q,u}=\lim_k z^{u,k}\) on
\(D_Z^u\), and \(z^{\mathbf q,u}=0\) on
\(([0,T]\times\Omega^u)\setminus D_Z^u\). Define
\[
 E_Z^u=\left\{(x,w)\in\Omega^u:
   \int_0^T|z^{\mathbf q,u}(t,x,w)|^2\,dt<\infty\right\},
\]
which is a Borel set and has full \(\nu_u\)-measure. Then define
\(\Phi_Z^{\mathbf q,u}(x,w)=[z^{\mathbf q,u}(\cdot,x,w)]\)
on \(E_Z^u\), and \(\Phi_Z^{\mathbf q,u}(x,w)\) being the zero element of \(\mathsf Z\) on
\(\Omega^u\setminus E_Z^u\). We can choose a countable orthonormal
basis \((e_n)_{n\geq1}\) of \(\mathsf Z\), with Borel measurable
representatives on \([0,T]\). Its scalar coordinates are
\[
 a_n^{\mathbf q,u}(x,w)
 :=\langle\Phi_Z^{\mathbf q,u}(x,w),e_n\rangle_{\mathsf Z}
 =\begin{cases}
 \displaystyle\int_0^T
 \langle z^{\mathbf q,u}(t,x,w),e_n(t)\rangle_{\RR^{d\times m}}\,dt,
   &(x,w)\in E_Z^u,\\[4pt]
 0,&(x,w)\in\Omega^u\setminus E_Z^u.
 \end{cases}
\]
Joint measurability of the integrand
and measurability of \(E_Z^u\) imply that each
\(a_n^{\mathbf q,u}\) is measurable. Hence the finite sum
\[
 (x,w)\longmapsto\sum_{n=1}^N a_n^{\mathbf q,u}(x,w)e_n
\]
is a measurable map into \(\mathsf Z\). They converge to \(\Phi_Z^{\mathbf q,u}(x,w)\) for every
\((x,w)\in\Omega^u\). Hence \(\Phi_Z^{\mathbf q,u}\) is Borel measurable, and
it coincides with \([Z^u]\), \(\nu_u\)-a.s. 

Since \(\cL(\zeta^u,B^u)=\nu_u\), evaluation at
\((\zeta^u,B^u)\) preserves the above construction features. The evaluated processes solve
\({\rm FBSDE}(u,\mathbf q;\zeta^u,B^u)\), and the uniqueness follows from the
previous arguments. Finally,
\eqref{eq:frozen-feedback-map} defines an \(A\)-valued measurable
non-anticipative function. Its square integrability follows
from the property of function \(\hat\alpha^u\) and the representation for $(X^u,Y^u,Z^u)$.
\end{proof}
}


\section{Additional details for Section~\ref{sec:pontryagin}}
\label{app:proofs-sec4}

In this section, we provide auxiliary arguments for the proof of Theorem~\ref{thm:nece}.

We start with the following result.
\begin{lemma}\label{lem:sd}
	For $\epsilon>0$ small enough, given $\alpha\in \cM\bH^2_T$ and any admissible direction, we denote by $\widehat{\alpha}$ the admissible control defined by $\widehat{\alpha}=\alpha+\epsilon\beta$, and denote by $ \widehat{X}= X^{\widehat{\alpha}}$ the corresponding controlled dynamics. Under Assumptions~\ref{ass:c-1}--\ref{ass:c-3}, we have
	\begin{equation}
		\label{fo:variation}
		\lim_{\epsilon\searrow 0}\int_{I}\EE\left[\sup_{0\le t\le T}\Bigl| \frac{\widehat{X}^u_t-X^u_t}{\epsilon}-V^u_t
		\Bigr|^2\right]du=0.
	\end{equation}
\end{lemma}

\begin{proof}
    
    We use the notations
	$\widehat{\Phi}_{t}^{u}=(\widehat{X}^u_{t},\widehat{\cG}^u_t,\widehat{\alpha}^u_{t}))$ and 
	$\widehat{V}^u_t= \epsilon^{-1}(\widehat{X}^u_t-X^u_t)-V^u_t$.
	Notice that $\widehat{V}^u_0=0$ for all $u\in I$ and we have
	\begin{equation*}
		\begin{split}
			&d\widehat{V}^u_t
			\\
			&= \biggl[\frac{1}{\epsilon}\bigl[b^u(t,\widehat{\Phi}^u_{t})-b^u(t,\Phi^u_{t})
			\bigr]-\partial_x b^u(t,\Phi^u_{t}) \cdot V^u_t-\partial_\alpha b^u(t,\Phi^u_{t})\cdot \beta^u_t
			\\
			&\qquad\qquad -\int_I\tilde{G}(u,v)\t\EE
			\bigl[
			\partial_\mu b^u(t,\Phi^u_{t})(\t X^v_t)\cdot \t V^v_t \bigr]dv \biggr]dt\\
			&\phantom{?}+\biggl[\frac{1}{\epsilon}\bigl[\sigma^u(t,\widehat{\Phi}^u_{t})-\sigma^u(t,\Phi^u_{t})
			\bigr]-\partial_x \sigma^u(t,\Phi^u_{t}) \cdot V^u_t-\partial_\alpha \sigma^u(t,\Phi^u_{t})\cdot \beta^u_t
			\\
			&\qquad\qquad-\int_I\tilde{G}(u,v)\t\EE
			\bigl[
			\partial_\mu \sigma^u(t,\Phi^u_{t})(\t X^v_t)\cdot \t V^v_t \bigr]dv \biggr]dW^u_t\\
			&=:\w V^{u,1}_tdt + \w V^{u,2}_t dW^u_t.
		\end{split}
	\end{equation*} 
	    Let us first compute the $dt$-term $\w V^{u,1}_t$. Note that, for each $t\in[0,T]$ and $\epsilon>0$, we have
	\begin{equation*}
		\begin{split}
			\frac{1}{\epsilon}\left[b^u(t,\w \Phi_t^{u})-b^u(t,\Phi^u_t)\right]
			&=\int_0^1\partial_x b^u\bigl(t,\Phi_t^{\lambda,u}\bigr) \cdot (\w V^u_t+ V^u_t)d\lambda
			+\int_0^1\partial_\alpha b^u\bigl(t,\Phi_t^{\lambda,u}\bigr) \cdot \beta^u_t d\lambda
			\\
			&\hspace{15pt}+\int_0^1\int_I \tilde{G}(u,v) \t\EE\bigl[
			\partial_ \mu b^u \bigl(t,\Phi_t^{\lambda,u}\bigr)\bigl(\t X_t^{\lambda,v}\bigr)\cdot (\tilde{\widehat{V}}^v_t+\t V^v_t) \bigr]dv d\lambda,
		\end{split}
	\end{equation*} 
	where we set $X^{\lambda,u}_t=X^u_t+\lambda\epsilon(\w V^u_t+V^u_t)$, 
	$\alpha^{\lambda,u}_t=\alpha^u_t+\lambda\epsilon\beta^u_t$, $ \cG_{t}^{\lambda,u}(dx) = \int_{ I} \tilde{G}(u,v)\mu^{\lambda,v}_t(dx)dv$, $\mu^{\lambda,u}_t=\cL(X^{\lambda,u}_t)$ and $\Phi^{\lambda,u}=(X^{\lambda, u},\cG^{\lambda,u},\alpha^{\lambda,u})$.
    Then, we deduce the following expression:
	\begin{equation*}
		\label{fo:VVepsilon}
		\begin{split}
			\w V^{u,1}_t&=\int_0^1\partial_xb^u\bigl(t,\Phi_{t}^{\lambda,u}\bigr) \cdot \w V^u_t d\lambda
			+\int_0^1\int_I \tilde{G}(u,v)\t\EE \bigl[ \partial_\mu b^u\bigl(t,\Phi^{\lambda,u}_t\bigr)(\t X^{\lambda,v}_t)\cdot \tilde{\widehat{V}}^v_t \bigr] dvd\lambda
			\\
			&\hspace{15pt}
			+\int_0^1 \bigl[\partial_xb^u\bigl(t,\Phi_{t}^{\lambda,u}\bigr)-\partial_x b^u(t,\Phi^u_{t}) \bigr] \cdot V^u_t d\lambda
			+\int_0^1\bigl[\partial_\alpha b^u\bigl(t,\Phi_{t}^{\lambda,u} \bigr)-\partial_\alpha b^u(t,\Phi^u_t) \bigr] \cdot \beta^u_t d\lambda
			\\
			&\hspace{30pt}+\int_0^1\int_I\tilde{G}(u,v)\t\EE
			\bigl[
			\bigl( \partial_\mu b^u \bigl(t,\Phi^{\lambda,u}_t\bigr)
			(\t X^{\lambda,v}_t)
			-\partial_{\mu} b^u(t,\Phi^u_t)(\t X^v_t) \bigr)\cdot\t V^v_t \bigr]
			dvd\lambda
			\\
			&= \int_0^1\partial_xb^u\bigl(t,\Phi^{\lambda,u}_t\bigr) \cdot \w V^u_t d\lambda
			+\int_0^1\int_I \tilde{G}(u,v)\t\EE \bigl[ \partial_\mu b^u\bigl(t,\Phi^{\lambda,u}_t\bigr)(\t X^{\lambda,v}_t)\cdot \tilde{\widehat{V}}^v_t \bigr] dvd\lambda
			\\
			&\qquad +I^{u,1}_t+ I^{u,2}_t+ I^{u,3}_t.
		\end{split}
	\end{equation*}
	First note that by Assumption~\ref{ass:c-2} and the definition of $\beta\in\cM\bH^2_T$, employing a standard estimate method for mean field systems and by the definition of $\widehat{V}^u_t$, we verify that
    $$
        \int_{I}\EE[\sup_{0\leq t\leq T}|\widehat{V}^u_t|^2]du\leq C
    $$ 
    for some constant $C$. Combining with the property $\int_{I}\EE\bigl[\sup_{0\le t\le T}\left| V^u_t
	\right|^2\bigr]du<\infty$, we have
	\begin{equation}\label{eq:xconti}
	\int_{I}\EE\bigl[\sup_{0\leq\lambda\leq 1}\sup_{0\leq t\leq T}|X^{\lambda,u}_t-X^u_t|^2\bigr]du\xrightarrow[\epsilon \to 0]{} 0. 
	\end{equation}
	 
	We then prove that $\int_{I}\EE[\int_0^T |I^{u,1}_t|^2dt] du$, $\int_{I}\EE[\int_0^T |I^{u,2}_t|^2dt] du$ and $\int_{I}\EE[\int_0^T |I^{u,3}_t|^2dt] du$ converge to $0$. First we have 
	\begin{eqnarray*}
		\EE\int_0^T|I^{u,1}_t|^2dt&=&\EE\int_0^T\left|\int_0^1[\partial_xb^u\bigl(t,\Phi^{\lambda,u}_t \bigr)-\partial_x b^u(t,\Phi^u_t)]V^u_t d\lambda
		\right|^2dt
		\\
		&\le&\EE\int_0^T\int_0^1|\partial_xb^u \bigl(t,\Phi_{t}^{\lambda,u} \bigr)-\partial_x b^u(t,\Phi^u_{t})|^2|V^u_t|^2 d\lambda dt.
	\end{eqnarray*} 
	By our assumption on the partial derivatives of coefficients, $\partial_{x} b$ is bounded and continuous in $x$, $\mu$ and $\alpha$. Combining with $\int_{I}\EE[\sup_{0\le t\le T}\left| V^u_t
	\right|^2]du<\infty$ and the continuity \eqref{eq:xconti}, we have $\int_{I}\EE[\int_0^T |I^{u,1}_t|^2dt] du$ converges to $0$ as $\epsilon\searrow 0$. Similar arguments apply to $I^{u,2}_{t}$ and $I^{u,3}_{t}$. Hence we have, for $S \in [0,T]$,
	\begin{equation*}
		\int_I\EE\bigl[\sup_{0\le t\le S}|\int_0^t\w V^{u,1}_sds|^2\bigr]du \leq \delta^{1}_{\epsilon} + C \int_I\int_0^S\EE\bigl[\sup_{0\le s\le t}|\w V^u_s|^2\bigr] dtdu,
	\end{equation*} 
	where $\delta^{1}_\epsilon$ is some small number converging to $0$ as $\epsilon\searrow 0$.
	
	Next, for the diffusion part $\w V^{u,2}_t$, we obtain similar inequalities in the same way. By the Burkholder--Davis--Gundy inequality, we get again that, for any $S \in [0,T]$,
	\begin{equation*}
		\int_I\EE\bigl[\sup_{0\le t\le S}|\int_0^t\w V^{u,2}_sdW^u_s|^2\bigr]du \leq \delta^{2}_{\epsilon} + C \int_I\int_0^S\EE\bigl[\sup_{0\le s\le t}|\w V^u_s|^2\bigr] dtdu,
	\end{equation*}
	for a sequence $\delta^{2}_\epsilon$ converging to $0$. 
	Adding up the above two results, we obtain that there exists a sequence $\delta_\epsilon$ converging to $0$ such that
	\begin{equation*}
		\int_I\EE\bigl[\sup_{0\le t\le S}|\w V^{u}_t|^2\bigr]du \leq \delta_{\epsilon} + C \int_0^S\int_I\EE\bigl[\sup_{0\le s\le t}|\w V^u_s|^2\bigr] dudt,
	\end{equation*} 
	By Gronwall's inequality, we get that  
	$
		\lim_{\epsilon \searrow 0} \int_I\EE \bigl[ \sup_{0 \le t \le T} \bigl\vert \w V^u_t \bigr\vert^2 \bigr]du = 0. 
	$
\end{proof}

\begin{proof}[Proof of Lemma~\ref{le:gateaux}]
	Using the same notations as in the proof of Lemma~\ref{lem:sd}, we have
	\begin{equation*}
		\begin{split}
			\frac{d}{d\epsilon}J(\alpha+\epsilon\beta)\big|_{\epsilon=0}&=\lim_{\epsilon\searrow 0}\frac{1}{\epsilon}\int_I\EE\bigg[\int_0^T
			\bigl[f^u(t,\w\Phi_t^u )- f^u(t,\Phi^u_t) \bigr]dt\bigg]du
            + \lim_{\epsilon\searrow 0}\frac{1}{\epsilon}\int_I\EE \bigl[g^u(\w X_T^u,\w \cG^u_T)-g^u(X^u_T,\cG^u_T) \bigr]du.
		\end{split}
	\end{equation*}
	We start with the first term on the right hand side of the above equation. We get
	\begin{equation*}
		\begin{split}
			&\lim_{\epsilon\searrow 0}\frac{1}{\epsilon}\int_I\EE\bigg[\int_0^T
			\bigl[f^u(t,\w\Phi_t^u )- f^u(t,\Phi^u_t) \bigr]dt\bigg]du\\
			&\hspace{15pt}=\lim_{\epsilon\searrow 0}\frac{1}{\epsilon}\int_I\EE\bigg[\int_0^T\int_0^1\frac{d}{d\lambda}f^u(t,\Phi^{\lambda,u}_t)d\lambda dt\bigg]du \\
			&\hspace{15pt}=\lim_{\epsilon\searrow 0}\int_I\EE\bigg[\int_0^T\int_0^1 \Big[\partial_x f^u \bigl(t,\Phi^{\lambda,u}_t \bigr) \cdot (\w V^u_t+V^u_t)\\
			&\hspace{50pt}+\int_ I\tilde{G}(u,v)\t\EE \bigl[\partial_\mu f^u \bigl(t,\Phi_{t}^{\lambda,u} \bigr)(\t X_t^{\lambda,v})\cdot (\tilde{\widehat{V}}^v_t+\t V^v_t) \bigr]dv  +\partial_\alpha f^u \bigl(t,\Phi_{t}^{\lambda,u} \bigr) \cdot \beta^u_t
			\Big]d\lambda dt \bigg]du
            \\
			&\hspace{15pt}=\int_I\EE\bigg[\int_0^T \bigl[\partial_x f^u(t,\Phi^u_t) \cdot V^u_t+\int_I\tilde{G}(u,v) \t\EE \bigl[\partial_\mu f^u(t,\Phi^u_t)(\t X^v_t)\cdot \t V^v_t \bigr]dv+\partial_\alpha f^u(t,\Phi^u_t) \cdot \beta^u_t \bigr]dt\bigg]du.
		\end{split}
	\end{equation*}
	The last equality follows by using the continuity and the linear growth of the partial derivatives of $f$ in Assumption~\ref{ass:c-1}-\ref{ass:c-2}, the Cauchy--Schwarz inequality, the results proven in Lemma~\ref{lem:sd}, and the standard uniform integrability arguments. Similar arguments also apply to the terminal cost.
\end{proof}

\begin{proof}[Proof of Lemma~\ref{le:duality}]
	Integration by parts gives
	\begin{equation*}
		\begin{split}
			&Y^u_T \cdot V^u_T 
            \\
            &=Y^u_0 \cdot V^u_0+\int_0^TY^u_t \cdot dV^u_t+\int_0^TdY^u_t \cdot V^u_t+\int_0^Td[ Y^u, V^u]_t
			\\
			&= M^u_{T} + \int_0^T \biggl[Y^u_t \cdot \bigl( \partial_x b^u(t,\Phi^u_t) \cdot V^u_t \bigr) 
			+Y^u_t \cdot \int_I \tilde{G}(u,v)\t\EE \bigl[ \partial_\mu b^u(t,\Phi^u_t)(\t X^v_t)\cdot\t V^v_t \bigr]dv\\
			&\qquad +Y^u_t \cdot \bigl( \partial_\alpha b^u(t,\Phi^u_t) \cdot \beta^u_t \bigr) 
			-\partial_x H^u(t,\Phi^u_t,Y^u_t, Z^u_t) \cdot V^u_t
            \\
            &\qquad -\int_{I} \tilde{G}(v,u)\t\EE \bigl[ \partial_\mu H^v(t,\t \Phi^v_t,\t Y^v_t,\t Z^v_t)(X^u_t)\cdot V^u_t \bigr]dv +Z^u_t \cdot \bigl( \partial_x\sigma^u(t,\Phi^u_t) \cdot V^u_t \bigr)
			\\
			&\qquad +Z^u_t \cdot \int_I\tilde{G}(u,v)\t\EE \bigl[ \partial_\mu \sigma^u(t,\Phi^u_t)(\t X^v_t)\cdot\t V^v_t \bigr]dv+Z^u_t \cdot \bigl( \partial_\alpha\sigma^u(t,\Phi^u_t) \cdot \beta^u_t \bigr) \biggr]dt,
		\end{split}
	\end{equation*}
	where $(M^u_t)_{0\le t\le T}$ is a mean zero integrable martingale. By Fubini's theorem, we have
	\begin{equation*}
		\begin{split}
			&\EE\t\EE \bigl[ \partial_\mu H^v(t,\t \Phi^v_t,\t Y^v_t,\t Z^v_t)(X^u_t)\cdot V^u_t \bigr]
			\\
			&\phantom{?}=\t \EE\EE\bigl[ \partial_\mu H^v(t,\t \Phi^v_t,\t Y^v_t,\t Z^v_t)(X^u_t)\cdot V^u_t \bigr]
			\\
			&\phantom{?}=\EE\t\EE \bigl[ \partial_\mu H^v(t,\Phi^v_t,Y^v_t,Z^v_t)(\t X^u_t)\cdot \t V^u_t \bigr]
			\\
			&\phantom{?}=\EE\t\EE \bigl[ \bigl( \partial_\mu b^v(t,\Phi^v_t)(\t X^u_t)\cdot\t V^u_t \bigr) \cdot Y^v_t+ \bigl( \partial_\mu\sigma^v(t,\Phi^v_t)(\t X^u_t)\cdot\t V^u_t \bigr) \cdot Z^v_t+\partial_\mu f^v(t,\Phi^v_t)(\t X^u_t)\cdot\t V^u_t \bigr].
		\end{split}
	\end{equation*}
	By taking expectations on both sides of the above equation, we can conclude.
\end{proof}

\begin{proof}[Proof of Lemma~\ref{thm:j}]
	By Fubini's theorem,
\begin{align*}
    &\EE \left[ \partial_x g^u(X^u_T,\cG^u_T) \cdot V^u_T + \int_I \tilde{G}(v,u)\t\EE\bigl[ \bigl( \partial_\mu g^v(\t X^v_T,\cG^v_T)( X^u_T)\cdot  V^u_T \bigr) \bigr]dv\right]
    =\EE[ Y^u_T\cdot V^u_T].
\end{align*}
	Furthermore,
	\begin{align*}
		\int_I\EE\Bigl[\int_I \tilde{G}(u,v)\t\EE \bigl[ \partial_\mu g^u(X^u_T,\cG^u_T)(\t X^v_T)\cdot \t V^v_T \bigr]dv\Bigr]du
		&=\int_I\EE\Bigl[\int_I \tilde{G}(v,u)\t\EE\bigl[ \partial_\mu g^v(\t X^v_T,\cG^v_T)( X^u_T)\cdot  V^u_T \bigr]dv\Bigr]du.
	\end{align*}
    Integrating over $u$ the second expectation in \eqref{eq:gat} gives
	\begin{eqnarray*}
		\int_I\EE \left[ \partial_x g^u(X^u_T,\cG^u_T) \cdot V^u_T + \int_I \tilde{G}(u,v)\t\EE \left[ \bigl( \partial_\mu g^u(X^u_T,\cG^u_T)(\t X^v_T)\cdot \t V^v_T \bigr)\right] dv \right]du
        = \int_I\EE[ Y^u_T\cdot V^u_T] du.
	\end{eqnarray*}
	Finally, using the expression derived in Lemma~\ref{le:duality} for $\EE[ Y_T\cdot V_T]$, combining it with \eqref{eq:gat} and canceling terms, we get the desired result.
\end{proof}

\section{Additional details for Section~\ref{sec:fbsde}}
\label{app:details-fbsde}

We first recall the following result on $[\tilde{G}\mu]$ from \cite[Lemma 3.1]{phamnonlinear}.




\begin{lemma}\label{lem:cg}
	For any $(\mu^u)_{u\in I},(\nu^u)_{u\in I}\in(\cP_2(\RR^d))^I$ such that $[\tilde{G}\mu]^u$ and $[\tilde{G}\nu]^u$, $u\in I$, are well defined in $\cP_2(\RR^d)$,  we have 
        $
        \int_{u\in I}\cW^2_2([\t G\mu]^u,[\t G\nu]^u )du\leq C \int_{u\in I}\cW^2_2(\mu^u,\nu^u )du$.
    Without supposing Assumption~\ref{ass:graphon}, we have
	$
        \sup_{u\in I}\cW_2([\t G \mu]^u,[\t G\nu]^u )\leq \sup_{u\in I}\cW_2(\mu^u,\nu^u ).
    $
\end{lemma}

We then turn to the proof of Lemma~\ref{lem:contra}. 
\begin{proof}[Proof of Lemma~\ref{lem:contra}]
	
	The proof follows standard estimation techniques for graphon mean field FBSDEs with the convexity arguments of cost functions. Compared to the classical mean field one, the difficulty lies in handling the graphon mean field parameter $\cG$ and estimating terms involving the partial derivative with respect to $\cG$. We use the same notations $\Phi^u$ representing  
	$(X^u_{t},\cG^u_{t},Y^u_{t},Z^u_{t},\alpha^u_{t})_{0 \leq t \leq T}$ and 
	$\theta^u$ representing  $(X^u_{t},\cG^u_{t},\alpha^u_{t})_{0 \leq t \leq T}$. First, by integration by parts, we have
	\begin{equation*}
		\begin{split}
			&\EE \bigl[ (X_T'^{,u}-X^u_T)\cdot Y^u_T \bigr] 
			\\
            &= \EE \bigl[ (\xi'^{,u}-\xi^u)\cdot Y^u_0 \bigr] 
			\\
			&\hspace{5pt}- \gamma \biggl\{  \EE \int_0^T
			\Big[ 
			\partial_xH^u(t,\Phi^u_{t}) \cdot (X_t'^{,u}-X^u_t)  
            \\
            &\qquad\qquad + \int_I\tilde{G}(v,u)\t\EE \bigl[ \partial_\mu H^v(t,\t \Phi^v_t)(X^u_t) \cdot (X_t'^{,u}-X^u_t)\bigr]dv  \Big] dt
			\\
			&\qquad\qquad- \EE \int_0^T 
			\Big[
			[b^u(t,\theta_t'^{,u})-b^u(t,\theta^u_t)] \cdot Y^u_{t}   + [\sigma^u(t,\theta_t'^{,u})-\sigma^u(t,\theta^u_t)]\cdot  Z^u_t \Big] dt \biggr\}
			\\
			&\hspace{5pt} - \biggl\{ 
			\EE \int_{0}^T \Big[ (X_{t}'^{,u}-X^u_{t}) \cdot \cI^{f,u}_{t} + (\cI^{b,u}_{t} - \cI'^{,b,u}_{t}) \cdot Y^u_{t} + (\cI^{\sigma,u}_{t}-\cI'^{,\sigma,u}_{t}) \cdot Z^u_{t} \Big] 
			dt \biggr\}
			\\
			&= \mathcal{J}^u_{0} - \gamma \mathcal{J}^u_{1} - \mathcal{J}^u_{2}.  
		\end{split}
	\end{equation*}
	By convexity Assumption~\ref{ass:d--3}, we have
	\begin{equation*}
		\begin{split}
			&\int_I\EE \bigl[(X_T'^{,u}-X^u_T)\cdot Y^u_T \bigr]du
			\\
            &= \gamma \int_I \EE \left[  \partial_xg^u(X^u_T,\cG^u_T)\cdot (X_T'^{,u}-X^u_T)+\int_I \tilde{G}(u,v)\t\EE[\partial_\mu g^u( X^u_T, \cG^u_T)(\t X^v_{T})] \cdot (\t X_T'^{,v}- \t X^v_T)dv \right]du 
            \\
            &\qquad + \int_I \EE \bigl[(\cI'^{,g,u}_T-\cI^{g,u}_T)\cdot Y^u_T \bigr]du
			\\
			&\leq \gamma \int_I\EE\bigl[g^u(X_T'^{,u},\cG'^{,u}_T)-g^u(X^u_T,\cG^u_T)\bigr]du + \int_I\EE \bigl[(\cI'^{,g,u}_T-\cI^{g,u}_T)\cdot Y^u_T \bigr]du.
		\end{split}
	\end{equation*}
Then similarly as the proof of Theorem~\ref{th:suffi} and applying the same convexity arguments to $H^u, u\in I$, we obtain
	\begin{equation*}
		\gamma J(\alpha') - \gamma J(\alpha) \geq 
		\gamma \lambda \int_I \EE \biggl[\int_{0}^T \vert \alpha^u_{t} - \alpha_{t}'^{,u} \vert^2 dt\biggr]du + \int_I
		\bigl(\cJ^u_{0} - \cJ^u_{2} + \EE \bigl[(\cI^{g,u}_T-\cI'^{,g,u}_T)\cdot Y^u_T \bigr]\bigr)du. 
	\end{equation*}
	Now, we reverse the roles of $\alpha$ and $\alpha'$ in the above equation and denote by $\cJ'^{,u}_{0}$ and $\cJ'^{,u}_{2}$ the 
	corresponding terms (defined similarly as $\cJ^{u}_{0}$ and $\cJ^{u}_{2}$ ). Summing both inequalities, we get
	\begin{align*}
		&2 \gamma \lambda \int_I\EE \bigl[\int_{0}^T \vert \alpha^u_{t} - \alpha_{t}'^{,u} \vert^2 dt\bigr]du
		\\
		&\quad +\int_I
		\bigl(\cJ^u_{0}+\cJ'^{,u}_0 - (\cJ^u_{2}+\cJ'^{,u}_{2}) + \EE \bigl[(\cI^{g,u}_T-\cI'^{,g,u}_T)\cdot (Y^u_T-Y'^{,u}_T) \bigr]\bigr)du \leq 0.  
	\end{align*}
		Young's inequality gives, for every $\varepsilon>0$ and a constant $C$
		independent of $\gamma$,
	\begin{equation}
		\label{eq:2:2:3_}
		\gamma \int_I\EE\biggl[ \int_{0}^T \vert \alpha^u_{t} - \alpha_{t}'^{,u} \vert^2 dt\biggr]du
		\leq 
		\varepsilon 
		\| \Phi - \Phi^{\prime} \|_{\mathcal S_1}^2 
		+ \frac{C}{\varepsilon} \biggl(\int_I \EE \bigl[ \vert \xi^u - \xi'^{,u} \vert^2 \bigr]du +
		\| {\mathcal I} - {\mathcal I}' \|_{\mathbb I_1}^2 \biggr).  
	\end{equation}

		Standard BSDE estimates, Lemma~\ref{lem:cg}, the Cauchy--Schwarz
		inequality, and the Lipschitz assumptions yield constants $\delta>0$ and
		$C>0$, independent of $u$, such that
	\begin{equation}
		\label{eq:Bes}
		\begin{split}
			&\EE \biggl[ \sup_{0 \leq t \leq T} \vert Y^u_{t} - Y'^{,u}_{t} \vert^2 + \int_{0}^T \vert Z^u_{t} - Z'^{,u}_{t} \vert^2 dt \biggr]
			\\
			&\leq C \gamma \EE \biggl[
			\sup_{0 \leq t \leq T} \vert X^u_{t} - X'^{,u}_{t} \vert^2
			\\
            &\qquad+  \int_{0}^T   \vert \alpha^u_{t} - \alpha'^{,u}_{t} \vert^2  dt \biggr]+ 
			C \| {\mathcal I}^u - {\mathcal I}'^{,u} \|_{\mathbb I}^2 + C \gamma 
			\int_I  \EE \bigl[ \sup_{0 \leq t \leq T} \vert X^u_{t} - X_{t}'^{,u} \vert^2  \bigr]du \\
			&\qquad+ \gamma\EE \bigl|\int_I\tilde{G}(v,u)\t\EE[\partial_{\mu} g^v(\t X^v_T,\cG^v_T)(X^u_T)]dv-\int_I\tilde{G}(v,u)\t\EE[\partial_{\mu} g^v(\t X'^{,v}_T,\cG'^{,v}_T)(X'^{,u}_T)]dv\bigr|^2\\
			&\qquad 
			+\delta \gamma\int_0^T\EE\Bigl[\Bigl|\int_I \tilde{G}(v,u)\t\EE\bigl[\partial_\mu H^v(t,\t\Phi^v_t)(X^u_t)\bigr]dv-\int_I\tilde{G}(v,u)\t\EE\bigl[\partial_\mu H^v(t,\t\Phi'^{,v}_t)(X'^{,u}_t)\bigr]dv\Bigr|^2\Bigr]dt.
		\end{split}
	\end{equation}
	Denote the last two terms (without the coefficients $\gamma$ and $\delta \gamma$ respectively) in the above inequality by $\t{\mathcal{J}}^u_1$ and $\t{\mathcal{J}}^u_2$ respectively. Let $\Theta^u, \Theta'^{,u}$ be the representative random variables of distributions $\cG^u,\cG'^{,u}$ respectively, for all $u\in I$. 
	By Fubini's theorem, Assumptions~\ref{ass:d--1}--\ref{ass:d--2} and the duality property, we have
	\begin{align*}
		\int_I \t{\mathcal{J}}^u_2 du
		& =
		\int_I\t\EE\EE[|\partial_\mu H^u(t,\t\Phi^u_t)(\Theta^u_t)-\partial_\mu H^u(t,\t\Phi'^{,u}_t)(\Theta'^{,u}_t)|^2]du\\
		&\leq  C\int_I\Big(\EE\bigl[|X^u_t-X'^{,u}_t|^2+|\alpha^u_t-\alpha'^{,u}_t|^2+|\Theta^u_t-\Theta'^{,u}_t|^2\bigr]\\
		& \quad \quad +\EE\bigl[|Y^u_t-Y'^{,u}_t|^2\bigr]+\EE\bigl[|Z^u_t-Z'^{,u}_t|^2\bigr]\Big)du.
	\end{align*}
Applying the same argument to $\widetilde{\mathcal J}_1^u$ and taking
$\delta$ sufficiently small, we obtain a constant $C$, independent of
$\gamma$, such that
	\begin{equation}
		\label{eq:Bes_2}
		\begin{split}
			&\int_I\EE \biggl[ \sup_{0 \leq t \leq T} \vert Y^u_{t} - Y'^{,u}_{t} \vert^2 + \int_{0}^T \vert Z^u_{t} - Z'^{,u}_{t} \vert^2 dt \biggr]du
			\\
			&\hspace{15pt} \leq C \gamma\int_I \EE \biggl[
			\sup_{0 \leq t \leq T} \vert X^u_{t} - X'^{,u}_{t} \vert^2
			+  \int_{0}^T   \vert \alpha^u_{t} - \alpha'^{,u}_{t} \vert^2  dt \biggr]du+ 
			C \| {\mathcal I} - {\mathcal I}' \|_{\mathbb I_1}^2.
		\end{split}
	\end{equation}

	Similarly, by standard estimates for graphon mean field type SDEs, we have for each $u\in I$, 
	\begin{align}
			\EE \bigl[ \sup_{0 \leq t \leq T} \vert X^u_{t} - X_{t}'^{,u} \vert^2  \bigr]
			& \leq C^u_1\EE \bigl[ \vert \xi^u - \xi'^{,u} \vert^2 \bigr] + C^u_1\gamma  \EE \int_{0}^T \vert \alpha^u_{t} - \alpha_{t}'^{,u} \vert^2 dt + C^u_1 
			\|   {\mathcal I} - {\mathcal I}'\|_{\mathbb I_1}^2
            \\
			& \quad\quad + C \gamma  
			\int_I \int_0^T \EE \bigl[ \sup_{0 \leq s \leq t} \vert X^v_{s} - X_{s}'^{,v} \vert^2  \bigr]dtdv,
			\label{eq:Fes}
	\end{align}
    with some constant $C^u_1$. 
	From here, by Gronwall's inequality, the uniform bound of $C^u_1, u\in I$ by Assumption~\ref{ass:d--1}, we have for some constant $C_1$,
	\begin{align}
			\int_I\EE \bigl[ \sup_{0 \leq t \leq T} \vert X^u_{t} - X_{t}'^{,u} \vert^2  \bigr]du
			& \leq \int_I\EE \bigl[ \vert \xi^u - \xi'^{,u} \vert^2 \bigr]du 
            \\
            &\qquad + C_1\gamma \int_I \EE \int_{0}^T \vert \alpha^u_{t} - \alpha_{t}'^{,u} \vert^2 dtdu  + C_1 \|   {\mathcal I} - {\mathcal I}'\|_{\mathbb I_1}^2.
			\label{eq:Fes_2}
	\end{align}
	Finally, by \eqref{eq:Bes_2}, \eqref{eq:Fes_2} and \eqref{eq:2:2:3_}, we deduce that for some new constant $C_2$, 
	\begin{equation}
		\begin{split}
			&\int_I \EE \Big[ \sup_{0 \leq t \leq T} \vert X^u_{t} - X_{t}'^{,u} \vert^2 + \sup_{0 \leq t \leq T} \vert Y^u_{t} - Y_{t}'^{,u} \vert^2 + \int_{0}^T \vert Z^u_{t} - Z_{t}'^{,u} \vert^2 dt\Big] du
			\\
			&\hspace{15pt} \leq (C_2\gamma+1)\int_I \EE \bigl[\int_{0}^T \vert \alpha^u_{t} - \alpha_{t}'^{,u} \vert^2 dt\bigr]du+ C_2 
			\bigl( \int_I\EE \bigl[ \vert \xi^u - \xi'^{,u} \vert^2 \bigr]du +
			\|   {\mathcal I} - {\mathcal I}'\|_{\mathbb I_1}^2 \bigr)
			\\
			&\hspace{15pt} \leq  C_2 \varepsilon 
			\| \Phi - \Phi^{\prime} \|_{\mathcal{S}_1}^2 
			+ \frac{C_2}{\varepsilon}\bigl( \int_I\EE \bigl[ \vert \xi^u - \xi'^{,u} \vert^2 \bigr]du +
			\| {\mathcal I} - {\mathcal I}' \|_{\mathbb I_1}^2 \bigr). 
		\end{split}
	\end{equation}
	We conclude using the Lipschitz property of $\alpha^u$ and $\alpha'^{,u}$, $u\in I$, and choosing $\varepsilon$ small enough.
\end{proof}

\section{Additional details for Section~\ref{sec:chaos}}

\begin{lemma}\label{lem:esempi}
	Suppose Assumptions~\ref{ass:iniepsilon} and~\ref{ass:e--1} hold. For all $i\in[N]$ and any $t\in[0,T]$, we have
	$\EE[\cW^2_2(\o\cG^{N,i}_t,\cG^{N,i}_t)]\leq C q_{N,d,\varkappa},$
	where $q_{N,d,\varkappa}\to 0$ as $N\to \infty$ is defined in~\eqref{eq:def_q}. %
\end{lemma}
\begin{proof}[Proof of Lemma~\ref{lem:esempi}]
Under Assumptions~\ref{ass:iniepsilon} and~\ref{ass:e--1}, it is seen that for each $i\in[N]$, the normalized weight $\kappa^{N,ij}, j\in[N]$ are uniformly bounded by $C/N$, and the required high moment condition is in force. Hence all conditions in \cite[Lemma 4.1]{phamnonlinear} are verified and the desired result follows.
\end{proof}
\end{appendix}

\end{document}